\documentclass[10pt]{article}

\usepackage{color}
\usepackage{amsmath}
\usepackage{graphicx}
\usepackage{footnote}
\usepackage{subscript}
\usepackage{amsfonts,color}
\usepackage{amsmath,amssymb}
\usepackage{epsfig}
\usepackage{amssymb} 
\usepackage{mathtools}
\usepackage{amsthm,wasysym}
\usepackage{framed}
\usepackage{cancel}

\usepackage[english]{babel}
\usepackage[utf8]{inputenc}
\makeatletter
\usepackage{float}
\usepackage{wrapfig}

\makeatletter
\let\@fnsymbol\@arabic
\makeatother

\usepackage{bbm}
\newcommand{\id}{{\boldsymbol{\mathbbm{1}}}}
\usepackage{multicol}

\usepackage{nicefrac,enumitem}

\newcommand{\tr}{{\rm tr}}
\newcommand{\dev}{{\rm dev}}
\newcommand{\sym}{{\rm sym}}
\newcommand{\skw}{{\rm skew}}

\newcommand{\Curl}{{\rm Curl}}

\newcommand{\norm}[1]{\|#1\|}

\def\dd{\displaystyle}

\newtheorem{theorem}{Theorem}[section]
\newtheorem{lemma}[theorem]{Lemma}
\newtheorem{remark}[theorem]{Remark}
\newtheorem{proposition}[theorem]{Proposition}

\newtheorem{definition}[theorem]{Definition}

\theoremstyle{definition}
\newtheorem{example}{Example}[section]

\def\barr{\begin{array}}
\usepackage{lscape}

	\def\earr{\end{array}}
\def\bec#1{\begin{equation}\label{#1}}
\def\becn{\begin{equation*}}
\def\endec{\end{equation}}
\def\endecn{\end{equation*}}
\def\dd{\displaystyle}
\def\bfm#1{\mbox{\boldmat}}

\allowdisplaybreaks
\begin{document}
	
	\title{A constrained Cosserat shell model up to order $O(h^5)$: Modelling, existence of minimizers, relations to classical shell models and scaling invariance of the bending tensor.}
\author{  Ionel-Dumitrel Ghiba\thanks{Corresponding author: Ionel-Dumitrel Ghiba,  \ Department of Mathematics,  Alexandru Ioan Cuza University of Ia\c si,  Blvd.
		Carol I, no. 11, 700506 Ia\c si,
		Romania; and  Octav Mayer Institute of Mathematics of the
		Romanian Academy, Ia\c si Branch,  700505 Ia\c si, email:  dumitrel.ghiba@uaic.ro} \quad and \quad Mircea B\^irsan\thanks{Mircea B\^irsan, \ \  Lehrstuhl f\"{u}r Nichtlineare Analysis und Modellierung, Fakult\"{a}t f\"{u}r Mathematik,
		Universit\"{a}t Duisburg-Essen, Thea-Leymann Str. 9, 45127 Essen, Germany; and Department of Mathematics,  Alexandru Ioan Cuza University of Ia\c si,  Blvd.
		Carol I, no. 11, 700506 Ia\c si,
		Romania;  email: mircea.birsan@uni-due.de} \quad
	and  \quad     Peter Lewintan\thanks{Peter Lewintan,  \ \  Lehrstuhl f\"{u}r Nichtlineare Analysis und Modellierung, Fakult\"{a}t f\"{u}r
		Mathematik, Universit\"{a}t Duisburg-Essen,  Thea-Leymann Str. 9, 45127 Essen, Germany, email: peter.lewintan@uni-due.de} \\
	and  \quad      Patrizio Neff\,\thanks{Patrizio Neff,  \ \ Head of Lehrstuhl f\"{u}r Nichtlineare Analysis und Modellierung, Fakult\"{a}t f\"{u}r
		Mathematik, Universit\"{a}t Duisburg-Essen,  Thea-Leymann Str. 9, 45127 Essen, Germany, email: patrizio.neff@uni-due.de}
}

\maketitle
\begin{abstract}

We consider  a   recently introduced  geometrically nonlinear elastic Cosserat shell model incorporating effects up to order $O(h^5)$ in the  shell thickness $h$. We develop the corresponding geometrically nonlinear  constrained Cosserat shell model, we show the existence of minimizers for the $O(h^5)$ and $O(h^3)$ case and we draw some connections to existing models and classical shell strain measures. Notably, the role of the appearing new bending tensor is highlighted and investigated with respect to an  invariance condition of Acharya [Int.  J. Solids and Struct., 2000] which will be further strengthened.
  \medskip
  
\end{abstract}

\begin{footnotesize}\renewcommand{\contentsname}{}
	\tableofcontents
\end{footnotesize}

\addtocontents{toc}{\protect\setcounter{tocdepth}{2}}

\section{Introduction}
Recently \cite{GhibaNeffPartI}
  we have established a novel geometrically nonlinear Cosserat shell model including terms up to order $O(h^5)$ in the shell-thickness $h$,  by extending the techniques from  \cite{Neff_plate04_cmt,Neff_membrane_plate03,Neff_membrane_existence03,Neff_plate07_m3as,Neff_membrane_Weinberg07}. The dimensional descent was obtained starting with a 3D-parent Cosserat model and assuming an appropriate 8-parameter ansatz for the shell-deformation through the thickness. This is the derivation approach and it has allowed us to arrive at specific novel strain and curvature measures.  See also \cite{birsan2020derivation} for an alternative derivation of a $O(h^3)$ Cosserat shell model.  In this way, we obtain a kinematical model which is similar, but it does not coincide, with the kinematical model of 6-parameter shells. The theory of 6-parameter shells was developed for shell-like bodies made of Cauchy materials, see the monographs \cite{Libai98,Pietraszkiewicz-book04} or the papers \cite{Eremeyev06,Pietraszkiewicz11}. We have remarked that even if we restrict our model to order $O(h^3)$, the obtained minimization problem is not the same as that previously considered in the literature, since  the influence of the curved initial shell configuration appears explicitly  in the expression of the coefficients  of the energies for the reduced two-dimensional variational problem and additional bending-curvature and curvature terms are present.

In order to improve our understanding of the new Cosserat shell model, it is useful to consider certain extreme limit cases. The investigated limit case considered in this paper, i.e., letting  the Cosserat couple modulus $\mu_{\rm c}\to \infty$, is similar to that considered in the case of the Cosserat plate model \cite{Neff_plate04_cmt} and is naturally suggested by the situation for the three dimensional Cosserat model, which will now be explained. 

The underlying nonlinear elastic 3D-problem is the two-field Cosserat variational  problem 
\begin{equation}\label{minprob}
I(\varphi_\xi,F_\xi,\overline{R}_\xi, \alpha_\xi)=\dd\int_{\Omega_\xi}\left[W_{\rm{mp}}(\overline U _\xi)+
W_{\rm{curv}}(\alpha_\xi)\right]dV(\xi)
- \Pi(\varphi_\xi,\overline{R}_\xi)\quad 
{\to}
\textrm{\ \ min.} \quad  {\rm   w.r.t. }\quad (\varphi_\xi,\overline{R}_\xi)\, ,
\end{equation}
where
\begin{align}
F_\xi\coloneqq\,&\nabla_\xi\varphi_\xi=\!\!\!\!\!\!\!\!\!\!\!\!\!\!\!\!\!\!\!\!\underbrace{{\rm polar}(F)}_{\quad \quad \quad \quad \quad \neq \overline{R}_\xi\ \text{in general}}\,\!\!\!\!\!\!\!\!\!\!\!\!\!\!\!\!\!\!\!\! \sqrt{F_\xi^TF_\xi}\in\mathbb{R}^{3\times3}\ \ \ \  \textrm{\it the deformation gradient}\,,  \notag\\
\overline{R}_\xi\in\  &\, {\rm SO}(3)\, \qquad \qquad \qquad \qquad \qquad \qquad\quad \textrm{\it microrotation tensor (independent tensor field)\,},  \notag\\
\overline U _\xi\coloneqq\,&\dd\overline{R}^T_\xi F_\xi\in\mathbb{R}^{3\times3} \ \ \  \ \, \, \, \,  \qquad \qquad \qquad \qquad \textrm{\it the non-symmetric Biot-type stretch tensor\,},  \notag\\
\alpha_\xi\coloneqq\,&\overline{R}_\xi^T\, \Curl_\xi \,\overline{R}_\xi\in\mathbb{R}^{3\times 3} \,  \qquad \qquad\qquad  \quad  \textrm{{\it the second order  dislocation density tensor }\cite{Neff_curl08}}\, ,  \\
\dd W_{\rm{mp}}(\overline U _\xi)\coloneqq\,&\dd\mu\,\lVert \dev\,\text{sym}(\overline U _\xi-\id_3)\rVert^2+\mu_{\rm c}\,\lVert \text{skew}(\overline U _\xi-\id_3)\rVert^2+
\dd\frac{\kappa}{2}\,[{\rm tr}(\text{sym}(\overline U _\xi-\id_3))]^2\ \ \ \textrm{\it physically linear}\, ,  \notag\\
\dd W_{\rm{curv}}( \alpha_\xi)\coloneqq\,&\mu\,{L}_{\rm c}^2\left( b_1\,\lVert \dev\,\text{sym}\, \alpha_\xi\rVert^2+b_2\,\lVert \text{skew}\, \alpha_\xi\rVert^2+  b_3\,
[{\rm tr}(\alpha_\xi)]^2\right)\qquad \qquad\qquad\qquad\quad \textrm{\it curvature energy\,},\notag\\
 &\Pi(\varphi_\xi,\overline{R}_\xi) \qquad \qquad\qquad\ \,\,\, \qquad\qquad\quad \textrm{\it represents the external loading potential\,}. \notag
\end{align}  
Here,  $\Omega_\xi\subset\mathbb{R}^3$ is a three-dimensional domain, see Figure \ref{Fig1}. The elastic material constituting the body is assumed to be homogeneous and isotropic and the reference configuration $\Omega_\xi$ is assumed to be a natural state.  All the body configurations are referred to a  fixed right Cartesian coordinate frame with unit vectors $
e_i$ along the axes $Ox_i$. A generic point of $\Omega_\xi$ will be denoted by $(\xi_1,\xi_2,\xi_3)$. 
The deformation of the body occupying the domain $\Omega_\xi$ is described by a vector map $\varphi_\xi:\Omega_\xi\subset\mathbb{R}^3\rightarrow\mathbb{R}^3$ (\textit{called deformation}) and by a \textit{microrotation}  tensor
$
\overline{R}_\xi:\Omega_\xi\subset\mathbb{R}^3\rightarrow {\rm SO}(3)\, 
$ attached at each point. 
We denote the current configuration (deformed configuration) by $\Omega_c\coloneqq \varphi_\xi(\Omega_\xi)\subset\mathbb{R}^3$. 
Moreover, $\mu>0$ denotes the shear modulus, $\kappa>0$ is the bulk modulus and $\mu_{\rm c}\geq0$ is the Cosserat couple modulus.

 For Cosserat couple modulus $\mu_{\rm c}\to \infty$ and characteristic length $L_{
\rm c}\to 0$ the previous Cosserat model approximates the classical (Cauchy-elastic) Biot variational problem
\begin{equation}\label{minprobB}
I(\varphi_\xi,F_\xi)=\dd\int_{\Omega_\xi} W_{\rm{mp}}( U _\xi)\,dV(\xi)
- \Pi(\varphi_\xi)\quad 
{\to}
\textrm{\ \ min.} \quad  {\rm   w.r.t. }\quad \varphi_\xi\, ,
\end{equation}
where
\begin{align}
F_\xi&\coloneqq {\nabla_\xi\varphi_\xi}={\rm polar}(F_\xi)\, \sqrt{F_\xi^TF_\xi}\in\mathbb{R}^{3\times3} &&\textrm{\it the deformation gradient},  \notag\\
{R}_\xi&\coloneqq {\rm polar}(F_
\xi)\in\, {\rm SO}(3) &&\textrm{\it continuum rotation (dependent tensor field)},  \notag\\
 U _\xi&\coloneqq\sqrt{F_\xi^T F_\xi}\in{\rm Sym}^+(3) &&\textrm{\it the positive definite symmetric Biot stretch tensor},  \notag\\
\dd W_{\rm{mp}}(\overline U _\xi)&\coloneqq\dd\mu\,\lVert \dev\,( U _\xi-\id_3)\rVert^2+
\dd\frac{\kappa}{2}\,[{\rm tr}( U _\xi-\id_3)]^2&& \textrm{\it physically linear isotropic response}\,, \notag
\notag\\
&\quad\  \Pi(\varphi_\xi) &&\textrm{\it represents the external loading potential\,}. \notag
\end{align}

\begin{figure}[h!]
	\begin{center}
		\includegraphics[scale=1.6]{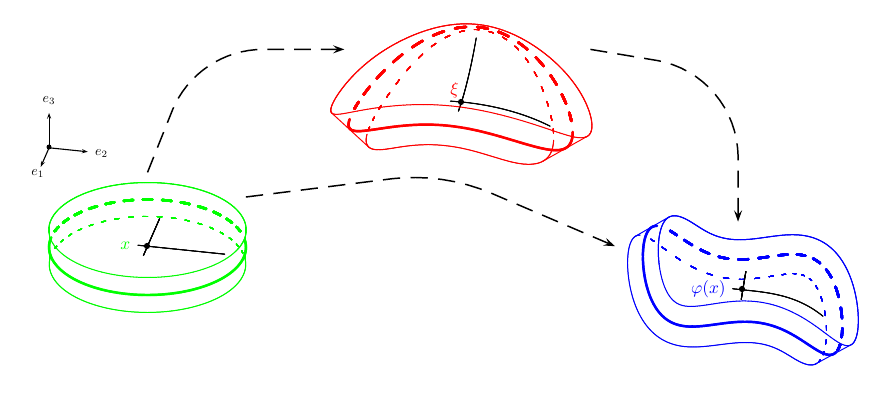}
		\put(-320,165){\footnotesize{$\Theta  ,  Q_0={\rm polar}(\nabla \Theta)$}} 
		\put(-275,82){\footnotesize{$\varphi  , \overline{ R}$}}
		\put(-275,72){\footnotesize{$ \mu_{\rm c}\to \infty:\ \overline{ R}\to R={\rm polar}(\nabla \varphi)$}}
		\put(-65,125){\footnotesize{$\varphi_\xi  ,   \overline{ R}_\xi$}}
		\put(-110,110){\footnotesize{$ \mu_{\rm c}\to \infty:\ \ \  \overline{ R}_\xi\to { R}_\xi= {\rm polar}(\nabla_\xi \varphi_\xi)$}}
			\put(-400,50){\footnotesize{$\Omega_h$}}  
			\put(-5,30){\footnotesize{$\Omega_c$}}  
				\put(-180,175){\footnotesize{$\Omega_\xi$}}  
		\caption{\footnotesize Kinematics of the 3D-Cosserat model and the 3D-constrained Cosserat model (${\boldsymbol{\mu_{\rm c}\to \infty}}$).}
		\label{Fig1}       
	\end{center}
\end{figure}
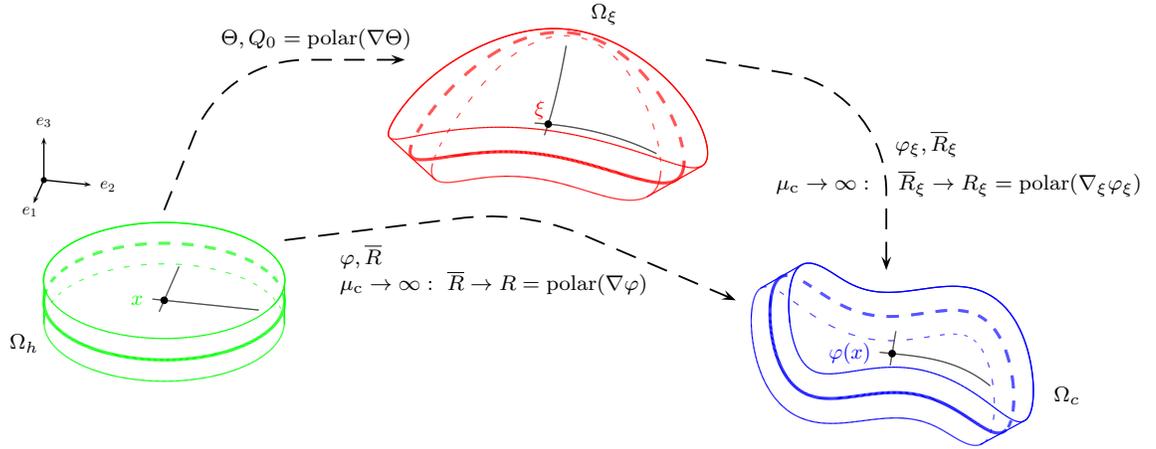

Now, we consider  $\Omega_\xi\subset\mathbb{R}^3$ to be a three-dimensional curved {\it shell-like thin domain}, see Figure \ref{Fig1}. We take  the \textit{fictitious Cartesian (planar) configuration} of the body $\Omega_h $. This parameter domain $\Omega_h\subset\mathbb{R}^3$ is a right cylinder of the form
$$\Omega_h=\left\{ (x_1,x_2,x_3) \,\Big|\,\, (x_1,x_2)\in\omega, \,\,\,-\dfrac{h}{2}\,< x_3<\, \dfrac{h}{2}\, \right\} =\,\,\dd\omega\,\times\left(-\frac{h}{2},\,\frac{h}{2}\right),$$
where  $\omega\subset\mathbb{R}^2$ is a bounded domain with Lipschitz boundary
$\partial \omega$ and the constant length $h>0$ is the \textit{thickness of the shell}.
For shell--like bodies we consider   the  domain $\Omega_h $ to be {thin}, i.e., the thickness $h$ is {small}.  
The diffeomorphism $\Theta:\mathbb{R}^3\rightarrow\mathbb{R}^3$ describing the reference configuration (i.e., the curved surface of the shell),  will be chosen in the specific form
\begin{equation}\label{defTheta}
\Theta(x_1,x_2,x_3)\,=\,y_0(x_1,x_2)+x_3\ n_0(x_1,x_2), \ \ \ \ \ \qquad  n_0\,=\,\dd\frac{\partial_{x_1}y_0\times \partial_{x_2}y_0}{\lVert \partial_{x_1}y_0\times \partial_{x_2}y_0\rVert}\, ,
\end{equation}
where $y_0:\omega\to \mathbb{R}^3$ is a function of class $C^2(\omega)$.  If not otherwise indicated, by $\nabla\Theta$ we denote
$\nabla\Theta(x_1,x_2,0)$, so that $\nabla\Theta=(\nabla y_0\,|\,n_0)$.  We use the polar decomposition \cite{neff2013grioli} of $\nabla \Theta$  and write 
\begin{equation}\label{dec}
\nabla \Theta\,=\,{Q}_0 \,U_0\, ,\qquad 
{Q}_0\,=\,{\rm polar}{(\nabla \Theta)}\in {\rm SO}(3 ),\qquad   U _0\in \rm{Sym}^+(3).
\end{equation} 

Further, let us  define the map
$
\varphi:\Omega_h\rightarrow \Omega_c,\  \varphi(x_1,x_2,x_3)=\varphi_\xi( \Theta(x_1,x_2,x_3)).
$
We view $\varphi$ as a function which maps the fictitious  planar reference configuration $\Omega_h$ into the deformed configuration $\Omega_c$.
We also consider the \textit{elastic microrotation}
$
\overline{Q}_{e,s}:\Omega_h\rightarrow{\rm SO}(3),\  \overline{Q}_{e,s}(x_1,x_2,x_3)\coloneqq \overline{R}_\xi(\Theta(x_1,x_2,x_3))\,.
$

  In \cite{GhibaNeffPartI}, by   assuming that  the elastic microrotation is constant through the thickness, i.e., 
$
\overline{Q}_{e,s}(x_1,x_2,x_3)=\overline{Q}_{e,s}(x_1,x_2), 
$
and  considering an \textit{8-parameter quadratic ansatz} in the thickness direction for the reconstructed total deformation $\varphi_s:\Omega_h\subset \mathbb{R}^3\rightarrow \mathbb{R}^3$ of the shell-like body, i.e., 
\begin{align}\label{ansatz}
\varphi_s(x_1,x_2,x_3)\,=\,&m(x_1,x_2)+\bigg(x_3\varrho_m(x_1,x_2)+\dd\frac{x_3^2}{2}\varrho_b(x_1,x_2)\bigg)\overline{Q}_{e,s}(x_1,x_2)\nabla\Theta.e_3\, ,
\end{align}
where $m:\omega\subset\mathbb{R}^2\to\mathbb{R}^3$ represents the total
deformation of the midsurface,  $\varrho_m,\,\varrho_b:\omega\subset\mathbb{R}^2\to \mathbb{R}$ allow in principal for symmetric thickness stretch  ($\varrho_m\neq1$) and asymmetric thickness stretch ($\varrho_b\neq 0$) about the midsurface\footnote{They may be written in terms of the  Poisson ratio of the isotropic and homogeneous material  $ \nu=\frac{\lambda}{2\,(\lambda+\mu)}
	$.} \begin{align}\label{final_rho}
\varrho_m\,=\,&1-\frac{\lambda}{\lambda+2\,\mu}[\bigl\langle  \overline{Q}_{e,s}^T(\nabla m|0)[\nabla\Theta ]^{-1},\id_3 \bigr\rangle -2]\;,\\
\dd\varrho_b\,=\,&-\frac{\lambda}{\lambda+2\,\mu}\bigl\langle  \overline{Q}_{e,s}^T(\nabla (\,\overline{Q}_{e,s}\nabla\Theta .e_3)|0)[\nabla\Theta ]^{-1},\id_3 \bigr\rangle   +\frac{\lambda}{\lambda+2\,\mu}\bigl\langle  \overline{Q}_{e,s}^T(\nabla m|0)[\nabla\Theta ]^{-1}(\nabla n_0|0)[\nabla\Theta ]^{-1},\id_3 \bigr\rangle ,\notag
\end{align}
we have obtained a completely two-dimensional minimization problem, see Figure \ref{Fig2}, in which the energy density is expressed in terms of the  following tensor fields (the same strain measures are also considered in \cite{Libai98,Pietraszkiewicz-book04,Eremeyev06,NeffBirsan13,Birsan-Neff-L54-2014} but with different motivations) on the surface $\omega\,$  
\begin{align}\label{e55}
\mathcal{E}_{m,s} & \coloneqq\,    \overline{Q}_{e,s}^T  (\nabla  m|\overline{Q}_{e,s}\nabla\Theta .e_3)[\nabla\Theta ]^{-1}-\id_3\not\in {\rm Sym}(3),\qquad \qquad\qquad \ \   \text{{\it elastic shell strain tensor}} ,  \\
\mathcal{K}_{e,s} & \coloneqq\,  \Big(\mathrm{axl}(\overline{Q}_{e,s}^T\,\partial_{x_1} \overline{Q}_{e,s})\,|\, \mathrm{axl}(\overline{Q}_{e,s}^T\,\partial_{x_2} \overline{Q}_{e,s})\,|0\Big)[\nabla\Theta ]^{-1}\not\in {\rm Sym}(3) \quad \text{\it\!  elastic shell bending--curvature tensor}.\notag
\end{align}

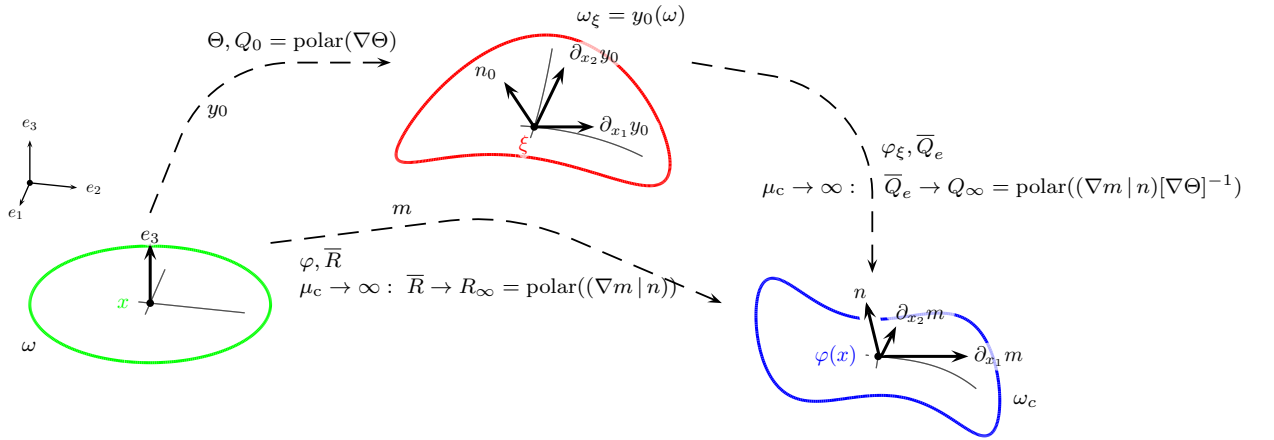
\begin{figure}[h!]
	\hspace*{-1cm}
	\begin{center}
		\includegraphics[scale=1.6]{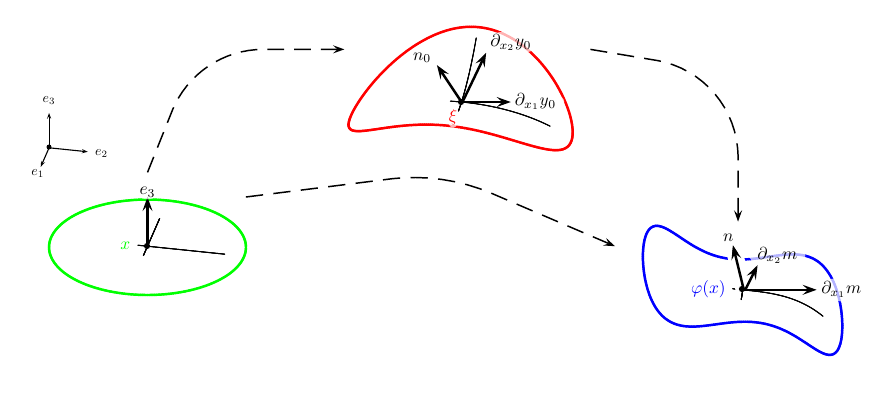}
		\put(-320,165){\footnotesize{$\Theta  ,  Q_0={\rm polar}(\nabla \Theta)$}} 
			\put(-320,140){\footnotesize{$y_0$}} 
		\put(-285,82){\footnotesize{$\varphi  , \overline{ R}$}}
		\put(-285,72){\footnotesize{$\mu_{\rm c}\to \infty:\ \overline{ R}\to R_\infty={\rm polar}((\nabla m\,|\,n))$}}
			\put(-250,102){\footnotesize{$m$}}
		\put(-65,125){\footnotesize{$\varphi_\xi,\overline{ Q}_e $}}
		\put(-110,110){\footnotesize{$\mu_{\rm c}\to \infty:\ \ \  \overline{ Q}_e\to Q_\infty={\rm polar}((\nabla m\,|\,n)[\nabla \Theta]^{-1})$}}
		\put(-390,50){\footnotesize{$\omega$}}  
		\put(-15,30){\footnotesize{$\omega_c$}}  
		\put(-180,175){\footnotesize{$\omega_\xi=y_0(\omega)$}}  
		\caption{\footnotesize Kinematics of the 2D-constrained Cosserat shell model (${\boldsymbol{\mu_{\rm c}\to \infty}}$).  Here,  ${Q}_{ \infty }  $ is the elastic rotation field, ${Q}_{0}$ is the  initial rotation from the fictitious planar Cartesian reference $\omega$ configuration to the initial  configuration $\omega_\xi$, and $R_{\infty}$ is the total rotation field from the fictitious planar Cartesian reference configuration $\omega$ to the deformed configuration $\omega_c$.}
		\label{Fig2}       
	\end{center}
\end{figure}
In the constrained Cosserat shell model, i.e.,  letting $\mu_{\rm c}\to \infty$, the elastic microrotation $\overline{Q}_{e,s}$ is not any more an independent tensor field (trièdre cach\'e, Cosserat \cite{Cosserat09}). A direct consequence of the assumption $\mu_{\rm c}\to \infty$ is that  the elastic microrotation $\overline{Q}_{e,s}$  is coupled to the midsurface displacement vector field $m$, through 
\begin{align}
\overline{Q}_{e,s}\to{Q}_{ \infty }\coloneqq {\rm polar}\big((\nabla  m|n) [\nabla\Theta ]^{-1}\big)=(\nabla m|n)[\nabla\Theta ]^{-1}\,\sqrt{[\nabla\Theta ]\,\widehat {\rm I}_{m}^{-1}\,[\nabla\Theta ]^{T}}\in {\rm SO}(3).
\end{align}


\begin{footnotesize}
	\begin{figure}
		\setlength{\unitlength}{1mm}
		\begin{center}
			\begin{picture}(0,176)
			\thicklines

			\put(-85,175){ \fbox{\begin{minipage}{5cm}
					\begin{center}\footnotesize $3D$-Cosserat elasticity\\
					$W(F,\overline{R})\sim\dd\mu\,\lVert \text{sym}( \overline{R}^TF-\id_3)\rVert^2$\qquad~ \\\qquad\quad\,$+\mu_{\rm c}\,\lVert \text{skew}( \overline{R}^TF-\id_3)\rVert^2$\\ $\quad\qquad\quad +L_{\rm c}^2\lVert\Curl \overline{R}\rVert^2$,~ $\overline{R}\in {\rm SO}(3)$\\
					Existence for:\\
					$\mu_{\rm c}>0, L_{\rm c}>0$
					\end{center}
					\end{minipage}}}
			\put(48,149.5){\vector(0,-1){25.5}}
			\put(7,140){\tiny  $\mu_{\rm c}\to \infty$}
			\put(-2,137){\tiny  $Q_\infty={\rm polar}\big((\nabla m|n)[\nabla \Theta]^{-1}\big)\in {\rm SO}(3)$}

			\put(49,140){\tiny $\mu_{\rm c}\to \infty$}
			
			\put(49,137){\tiny  $\overline{A}_{\infty}=\skw\big( (\nabla v|\delta n)[\nabla \Theta]^{-1}\big)\in\!\mathfrak{s\!o}(3)$}
			
			\put(6,86){\tiny $L_{\rm c}\to 0$}
			\put(6,84)	{\tiny (no curvature energy)}
			\put(-29,162){\tiny  dimensional }
			\put(-29,160){\tiny   reduction}
			\put(-29,158){\tiny  engineering }
			\put(-29,156){\tiny   ansatz}
				\put(-29,153){\tiny  $\frac h2\kappa<1 $}
			\put(-31.3,175){\vector(4,-3){19}}
					\put(-31.3,175){\vector(4,1){71}}
					\put(0,187){\tiny line{a}rization}
					\put(0,180){\tiny $\overline{R}=\id_3+\overline{A}_\vartheta+\textrm{h.o.t.}$}
					\put(-1.75,176){\tiny $\overline{A}_\vartheta={\rm Anti} (\vartheta)\in\! \mathfrak{s\!o}(3)$, $\vartheta\in \mathbb{R}^3$}
				
			
			\put(-65,140){\tiny $\mu_{\rm c}\to \infty$}
			\put(-65,137){\tiny $\overline{R}={\rm polar}(F)\in {\rm SO}(3)$}
			
			\put(-12,158.3){\fbox{\begin{minipage}{3cm}\begin{center}\footnotesize 
					$2D$-Cosserat shell model\\
					Existence for: $\frac{h}{2}\, \kappa< 0.48$ \\
					$\mu_{\rm c}>0, L_{\rm c}>0$
					\end{center}	\end{minipage}}}
			\put(20.2,157){\vector(1,0){19.5}}
			\put(22,159){\tiny linearization}
			\put(5,149.5){\vector(0,-1){18}}

				\put(40,190){\fbox{\begin{minipage}{5.1cm}	\begin{center}\footnotesize linearized $3D$-Cosserat elasticity\\
						$W^{\rm lin}(\nabla u,\overline{A}_\vartheta)\sim\dd\mu\,\lVert \text{sym}\nabla u\rVert^2 \qquad\qquad $\\ $\qquad\qquad\qquad+\mu_{\rm c}\,\lVert \text{skew}(\nabla u-\overline{A}_\vartheta)\rVert^2$\\ $\qquad\ \ +L_{\rm c}^2\lVert\Curl \overline{A}_\vartheta\rVert^2$\\
						Existence for: \\
						$\mu_{\rm c}>0, L_{\rm c}>0$
						\end{center}	\end{minipage}}}
			\put(48,180){\vector(0,-1){15.5}}
			\put(50,174){\tiny dimensional}
			\put(50,172){\tiny reduction}
			
			\put(40,156){\fbox{\begin{minipage}{4cm}\begin{center}\footnotesize 
					linearized	$2D$-Cosserat  \\
					shell model \cite{GhibaNeffPartIV}\\
					Existence for:\\
					$\mu_{\rm c}>0, L_{\rm c}>0$
					\end{center}	\end{minipage}}}
			\put(-66,164){\vector(0,-1){38}}

			\put(-85,110){\fbox{\begin{minipage}{4.5cm}\begin{center}\footnotesize Constrained 3D Cosserat elasticity\\
					(Toupin couple stress model)\\
					constraint: \ $\overline{R}={\rm polar}(F)$\\
					\hspace{-1.3cm}$W(F,{\rm polar}(F))\sim$ \\ \hspace{1cm}$\dd\mu\,\lVert  [{\rm polar}(F)]^TF-\id_3\rVert^2$ \\ \hspace{0.5cm}$\quad\, +L_{\rm c}^2\lVert\Curl \, {\rm polar}(F)\rVert^2$\\
					Existence for: $ L_{\rm c}>0$\\
					includes curvature energy\\
					\end{center}
					\end{minipage}}}
			\put(-66,96.8){\vector(0,-1){28}}
			\put(-65,86){\tiny $L_{\rm c}\to 0$}
			\put(-65,84){\tiny (no curvature energy)}
			
			\put(-37.7,110){\vector(1,0){17}}
			\put(-36,115){\tiny direct }
			\put(-36,113){\tiny dimensional }
			\put(-36,111){\tiny  reduction}
			\put(-36,107){\tiny  not yet done}

			\put(-85,60){\fbox{\begin{minipage}{4cm}
					\begin{center}\footnotesize  $3D$-Biot elasticity\\
					$W(F)\sim \mu\,\lVert \sqrt{F^TF}-\id_3\rVert^2$
					\\
					{\bf not well-posed}, since non-elliptic in $F$\end{center}
					\end{minipage}}}
			\put(-66,53.5){\vector(0,-1){11}}
			\put(-80,48){\tiny linearization}
			\put(-66,15.5){\vector(0,1){12}}
			\put(-80,21){\tiny linearization}
			
			\put(-85,34){\fbox{\begin{minipage}{4cm}
					\begin{center}\footnotesize   Classical linear elasticity
					\\ $ W_{\rm lin}(\varepsilon)\sim\mu\, \lVert \varepsilon\rVert^2$\\ $  \varepsilon=\sym \nabla u=\sym (F-\id_3)$\\
					well-posed\end{center}
					\end{minipage}}}
			
			\put(-42.5,29){\vector(4,1){82}}
			
			\put(-10,35){\tiny dimensional reduction }
			\put(-10,33){\tiny under explicit Kirchhoff-Love assumptions }
			\put(-10,30){\tiny  $\frac h2\,\kappa<1 $}
			\put(-39,67){\tiny direct}
			\put(-39,65){\tiny dimensional  }
			\put(-39,63){\tiny reduction }
			\put(-39,58){\tiny   not yet done}
			\put(-42.7,62){\vector(1,0){22}}
			
			\put(-20,110){\fbox{\begin{minipage}{4cm}\begin{center}\footnotesize Constrained Cosserat shell model
					\\
					quadratic in terms of the differences\\
					{	\tiny	$\sqrt{[\nabla\Theta ]^{-T}\,{\rm I}_m^{\flat }\,[\nabla\Theta ]^{-1}}$
						\quad -$\sqrt{[\nabla\Theta ]^{-T}\,{\rm I}_{y_0}^{\flat }\,[\nabla\Theta ]^{-1}}$ and $\sqrt{[\nabla\Theta ]^{-T}\,\widehat{\rm I}_m\,[\nabla\Theta ]^{-1}}$
						\quad    $\times[\nabla\Theta ]\Big({\rm L}_{y_0}^\flat - {\rm L}_m^\flat\Big)[\nabla\Theta ]^{-1}$}\\
					mixed terms\\
					includes curvature energy\\
					well-posed for $L_c>0$\end{center}
					\end{minipage}}}
			\put(5,91){\vector(0,-1){9}}

			\put(-20,64){\fbox{\begin{minipage}{4cm}\begin{center}\footnotesize Koiter-type shell model
					\\
					quadratic in terms of the differences\\
					{	\tiny	$\sqrt{[\nabla\Theta ]^{-T}\,{\rm I}_m^{\flat }\,[\nabla\Theta ]^{-1}}$
						$-\sqrt{[\nabla\Theta ]^{-T}\,{\rm I}_{y_0}^{\flat }\,[\nabla\Theta ]^{-1}}$ and $\sqrt{[\nabla\Theta ]^{-T}\,\widehat{\rm I}_m\,[\nabla\Theta ]^{-1}}$
						\qquad$\times[\nabla\Theta ]\Big({\rm L}_{y_0}^\flat - {\rm L}_m^\flat\Big)[\nabla\Theta ]^{-1}$}\\
					mixed terms \\
					 {\bf not well-posed}\end{center}
					\end{minipage}}}
			\put(22.5,60){\vector(1,0){17.5}}
			\put(23,62){\tiny  linearization}
			
			\put(22.5,110){\vector(1,0){17.5}}
			\put(25,112){\tiny  linearization}
			\put(40,110){\fbox{\begin{minipage}{4.7cm}
					\begin{center}\footnotesize  linearized constrained Cosserat shell model,
					quadratic in terms of 
					$\mathcal{G}_{\rm{Koiter}}^{\rm{lin}}$ and \\
					$\mathcal{R}_{\rm{AL}}^{\rm{lin}}=\mathcal{R}_{\rm{Koiter}}^{\rm{lin}} -{\bf 2}\, \sym[\,\mathcal{G}_{\rm{Koiter}}^{\rm{lin}} \,{\rm L}_{y_0}]$ \\
					mixed terms
					\\
					includes curvature energy\\ well-posed 
					(not yet shown)\end{center}
					\end{minipage}}}
			\put(48,98){\vector(0,-1){25.5}}
			\put(50,86){\tiny  $L_{\rm c}\to 0$}
				\put(50,84){\tiny (no curvature energy)}

			\put(40,60){\fbox{\begin{minipage}{4.7cm}
					\begin{center}\footnotesize   Anicic-L\`eger linear Kirchhoff-Love  shell model \cite{anicic1999formulation},
					quadratic in terms of 
					$\mathcal{G}_{\rm{Koiter}}^{\rm{lin}}$ and \\
					$\mathcal{R}_{\rm{AL}}^{\rm{lin}}=\mathcal{R}_{\rm{Koiter}}^{\rm{lin}} -{\bf 2}\, \sym[\,\mathcal{G}_{\rm{Koiter}}^{\rm{lin}} \,{\rm L}_{y_0}]$ \\
					mixed terms \\ well-posed \cite{anicic1999formulation}
					\end{center}
					\end{minipage}}}
			
			\put(48,50){\vector(0,-1){31}}
			\put(50,39.7){\tiny exclude all mixed energies}
			\put(50,36.7){\tiny in terms of	$\mathcal{G}_{\rm{Koiter}}^{\rm{lin}}$} \put(50,33.7){\tiny and $\mathcal{R}_{\rm{Koiter}}^{\rm{lin}}$,}
			\put(50,30.7){\tiny exclude dependence}
			\put(50,27.7){\tiny of coefficients}
			\put(50,24.7){\tiny on ${\rm H}$ and ${\rm K}$}
			
			\put(-85,7){\fbox{\begin{minipage}{4cm}
					\begin{center}\footnotesize $3D$-SVK model \\
					$W(F)\sim \frac{\mu}{4}\lVert F^TF-\id_3\rVert^2$
					\\
					{\bf not well-posed}, since non-elliptic in $F$\end{center}
					\end{minipage}}}
			\put(-42.7,10){\vector(1,0){22}}
			\put(-41,14){\tiny dimensional }
			\put(-41,11){\tiny reduction}
			
			\put(-41.5,6){\tiny + ad hoc}
			\put(-41.75,3){\tiny  assumptions \cite{Steigmann13}}
			\put(-41.5,0){\tiny  $\frac{h}{2}\,\kappa \ll 1$}
			
			\put(-20,10){\fbox{\begin{minipage}{4cm}
					\begin{center}\footnotesize classical finite-strain Koiter shell model \cite{Steigmann13}
					\\
					quadratic in terms of the differences\\
					$[\nabla\Theta ]^{-T}({\rm I}_m^\flat-{\rm I}_{y_0}^\flat) [\nabla\Theta ]^{-1}$ and $[\nabla\Theta ]^{-T}({\rm II}_m^\flat-{\rm II}_{y_0}^\flat)[\nabla\Theta ]^{-1}$\\
					{\bf not well-posed} \end{center}
					\end{minipage}}}
			\put(22.5,10){\vector(1,0){17.5}}
			\put(23,12){\tiny  linearization}
			\put(40,10){\fbox{\begin{minipage}{4.9cm}
					\begin{center}\footnotesize classical linearized Koiter model \cite{Ciarlet00},
					quadratic in terms of \\
					$\mathcal{G}_{\rm{Koiter}}^{\rm{lin}}$ and 
					$\mathcal{R}_{\rm{Koiter}}^{\rm{lin}}$ \\ well-posed \end{center}
					\end{minipage}}}
			
			\end{picture}
		\end{center}
		\caption{\footnotesize A schematic representation of the appearing models. Here, $\varphi=x+u(x)$ is the 3D-deformation, $u$ is the three-dimensional displacement,  $F=\nabla \varphi$ is the deformation gradient, $\overline{R}\in{\rm SO}(3)$ represents the Cosserat microrotation, $m=y_0+v(x)$ is the midsurface deformation, $v$ is the midsurface displacement,  $\overline{A}_\vartheta\in \mathfrak{so}(3)$ is the infinitesimal microrotation, ${\rm I}_{m}\coloneqq [{\nabla  m}]^T\,{\nabla  m}\in \mathbb{R}^{2\times 2}$ and  ${\rm II}_{m}\coloneqq\,-[{\nabla  m}]^T\,{\nabla  n}\in \mathbb{R}^{2\times 2}$ are  the matrix representations of the {\it first fundamental form (metric)} and the  {\it  second fundamental form}  {on $m(\omega)$}, respectively, $n\,=\,\dd\frac{\partial_{x_1}m\times \partial_{x_2}m}{\lVert \partial_{x_1}m\times \partial_{x_2}m\rVert} $,  {and with} ${\rm L}_{m}$  {we identify} the {\it Weingarten map (or shape operator)}  {on $m(\omega)$ with its associated matrix} defined by 
			$
			{\rm L}_{m}\,=\, {\rm I}_{m}^{-1} {\rm II}_{m}\in \mathbb{R}^{2\times 2}
			$, with similar definitions for ${\rm I}_{y_0}$, ${\rm II}_{y_0}$, ${n_0}$ and ${\rm L}_{y_0}$  {on the surface $y_0(\omega)$}. The tensor  \ $\mathcal{G}_{\rm{Koiter}}^{\rm{lin}} =\frac{1}{2}\big[{\rm I}_m - {\rm I}_{y_0}\big]^{\rm{lin}}
			= \sym\big[ (\nabla y_0)^{T}(\nabla v)\big]$ represents the infinitesimal change of metric, $\mathcal{R}_{\rm{Koiter}}^{\rm{lin}}=\big[{\rm II}_m - {\rm II}_{y_0}\big]^{\rm{lin}} $ is the infinitesimal change of curvature,  	
			and we have used the linear approximation of the normal $n=n_0+\delta n+{\rm h.o.t.}$, where $\delta n=\frac{1}{\sqrt{\det {\rm I}_{y_0}}}\left(\partial_{x_1} y_0\times \partial_{x_2} v+\partial_{x_1} v\times \partial_{x_2} y_0\right) -\tr({\rm I}_{y_0}^{-1}\, \sym ((\nabla y_0)^T\nabla v) )\,n_0$ is the increment of the normal when $y_0\to y_0+v(x)$.  The dimensional reduction is undertaken in the sense of an engineering ansatz and $\kappa$ denotes a typical principal curvature. The relations concerning the third column of this table will be made explicit in \cite{GhibaNeffPartIV}.}\label{diagram}
	\end{figure}
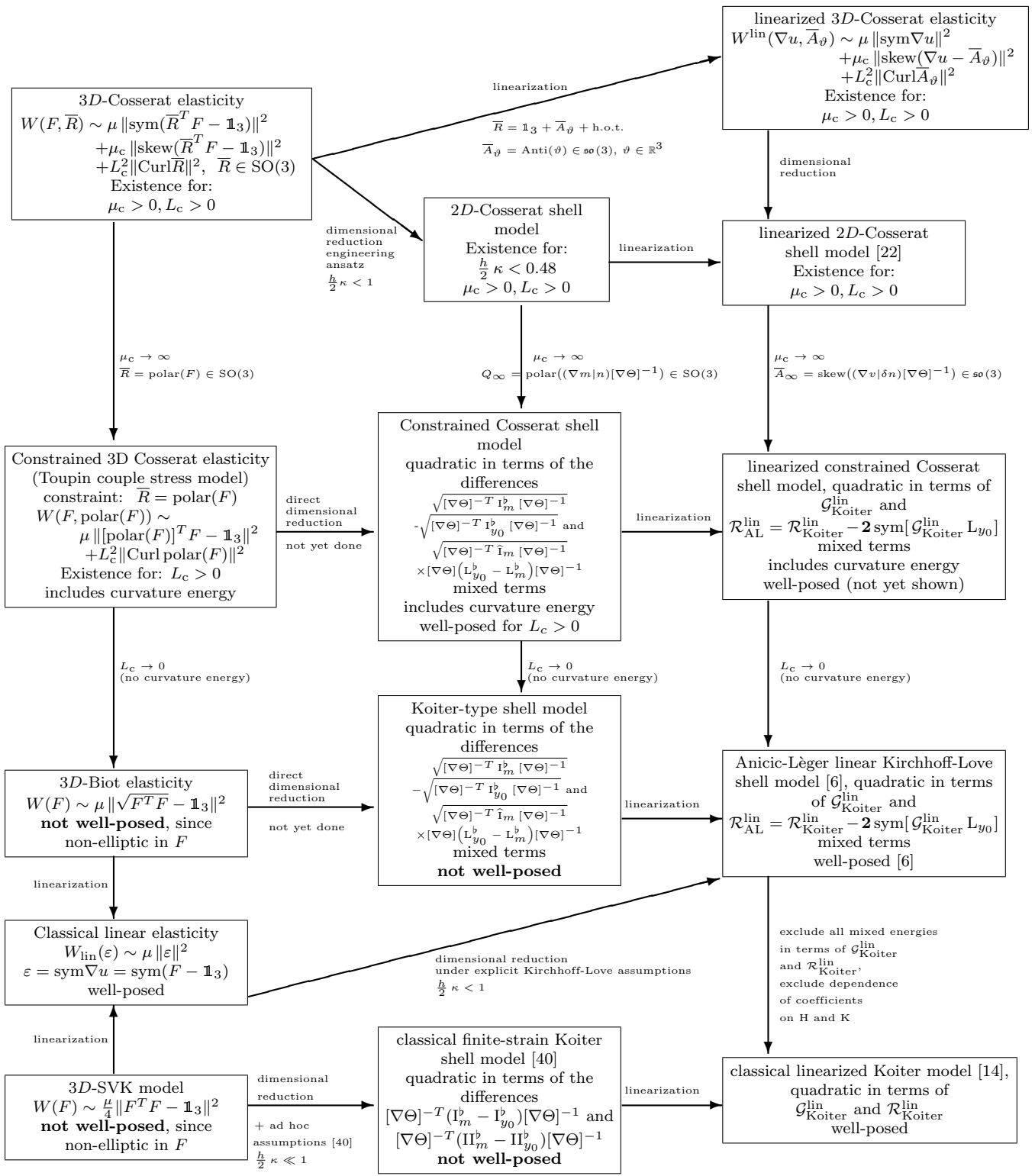
\end{footnotesize}

Considering Figure \ref{diagram}, in this paper we determine the precise form of the Koiter-type limit problem appearing for $\mu_{\rm c}\to \infty$. We recall that (see Appendix \ref{AppendixKoiter}), in matrix format and for a nonlinear elastic shell, the  variational problem for the classical isotropic Koiter shell model  is to find a deformation of the midsurface
$m:\omega\subset\mathbb{R}^2\to\mathbb{R}^3$  minimizing:
\begin{align}\label{Ap7matrix}
\dd\int_\omega& \bigg\{h\bigg(
\mu\rVert    [\nabla\Theta]^{-T} \,\frac{1}{2}\big({\rm I}_m^\flat-{\rm I}_{y_0}^\flat\big) \, [\nabla\Theta]^{-1}\rVert^2  +\dfrac{\,\lambda\,\mu}{\lambda+2\,\mu} \, \mathrm{tr} \Big[ [\nabla\Theta]^{-T} \,\big({\rm I}_m^\flat-{\rm I}_{y_0}^\flat\big) \, [\nabla\Theta]^{-1}\Big]^2\bigg) \vspace{6pt} \\
&+\dd\frac{h^3}{12}\bigg(
\mu\rVert    [\nabla\Theta]^{-T} \big({\rm II}_m^\flat-{\rm II}_{y_0}^\flat\big)  [\nabla\Theta]^{-1}\rVert^2  +\dfrac{\,\lambda\,\mu}{\lambda+2\,\mu} \, \mathrm{tr} \Big[ [\nabla\Theta]^{-T} \big({\rm II}_m^\flat-{\rm II}_{y_0}^\flat\big)  [\nabla\Theta]^{-1}\Big]^2\bigg)\bigg\}\,{\rm det}\nabla \Theta \, da.
\notag\end{align}
Here  ${\rm I}_{m}\coloneqq [{\nabla  m}]^T\,{\nabla  m}\in \mathbb{R}^{2\times 2}$ and  ${\rm II}_{m}\coloneqq\,-[{\nabla  m}]^T\,{\nabla  n}\in \mathbb{R}^{2\times 2}$ are  the matrix representations of the {\it first fundamental form (metric)} and the  {\it  second fundamental form}  {on $m(\omega)$}, respectively, $n\,=\,\dd\frac{\partial_{x_1}m\times \partial_{x_2}m}{\lVert \partial_{x_1}m\times \partial_{x_2}m\rVert} $,  {and with} ${\rm L}_{m}$  {we identify} the {\it Weingarten map (or shape operator)}  {on $m(\omega)$ with its associated matrix}  defined by 
$
{\rm L}_{m}\,=\, {\rm I}_{m}^{-1} {\rm II}_{m}\in \mathbb{R}^{2\times 2}
$, with similar definitions for ${\rm I}_{y_0}$, ${\rm II}_{y_0}$, ${n_0}$ and ${\rm L}_{y_0}$  {on the surface $y_0(\omega)$}, see Section \ref{Intro} for further notations.

In comparison with this nonlinear Koiter model, in the obtained  new constrained elastic Cosserat shell model, the pure membrane energy is expressed in terms of the difference 
\begin{align}\sqrt{[\nabla\Theta ]^{-T}\,{\rm I}_m^{\flat }\,[\nabla\Theta ]^{-1}}-
\sqrt{[\nabla\Theta ]^{-T}\,{\rm I}_{y_0}^{\flat }\,[\nabla\Theta ]^{-1}}\end{align} and not in terms of \begin{align}[\nabla\Theta]^{-T} \,\frac{1}{2}\,\big({\rm I}_m^\flat-{\rm I}_{y_0}^\flat\big)[\nabla\Theta ]^{-1}=[\nabla\Theta ]^{-T}\mathcal{G}_{\rm{Koiter}} \, [\nabla\Theta ]^{-1},\end{align} 
as it is the case in the classical Koiter shell model. This is a consequence of the fact that in the parent three-dimensional energy, we are starting from the non-symmetric  Biot-type stretch tensor (similar to the right stretch tensor $U=\sqrt{F^TF}$), while the Koiter shell model is typically constructed  considering a quadratic form in terms of the right  Cauchy-Green deformation tensor $C=U^2=F^TF$. 
We notice that in our constrained Cosserat shell model there does  not exist a pure bending energy, the bending terms (those involving the second fundamental form) are always coupled with membrane terms (those involving the first fundamental form). The presence of energies depending on the difference of the square roots of the first fundamental forms is consistent with new estimates of the distance between two surfaces 
\cite{ciarlet2015nonlinear,ciarlet2019new} obtained in  \cite{malin2018nonlinear}. There,  it is shown that the difference $v=m-y_0$ is completely controlled by  
\begin{align}\lVert\sqrt{[\nabla\Theta ]^{-T}\,{\rm I}_m^{\flat }\,[\nabla\Theta ]^{-1}}-
\sqrt{[\nabla\Theta ]^{-T}\,{\rm I}_{y_0}^{\flat }\,[\nabla\Theta ]^{-1}}\rVert
\end{align} and 
\begin{align}
\lVert\dd\sqrt{  [\nabla\Theta ]\,\widehat{\rm I}_{m}^{-1}[\nabla\Theta ]^{T}}[\nabla\Theta ]^{-T}\, {\rm II}_{m}^{\flat }[\nabla\Theta ]^{-1}-\sqrt{  [\nabla\Theta ]\,\widehat{\rm I}_{y_0}^{-1}[\nabla\Theta ]^{T}}[\nabla\Theta ]^{-T}\, {\rm II}_{y_0}^{\flat }[\nabla\Theta ]^{-1}\rVert.
\end{align} 

The sum of the two expressions appears in the constrained Cosserat plate model ($\nabla\Theta=\id$), too, see Eq. \eqref{minplate}, while an additional term is present in our membrane-bending energy which has a format that cannot be guessed: it is a quadratic form in terms of the change of curvature tensor  \begin{align} \label{eq:gleichunten}\sqrt{[\nabla\Theta ]^{-T}\,\widehat{\rm I}_m\,[\nabla\Theta ]^{-1}}&\,  [\nabla\Theta ]\Big( {\rm L}_m^\flat-{\rm L}_{y_0}^\flat \big)[\nabla\Theta ]^{-1}\notag\\&=
\sqrt{[\nabla\Theta ]^{-T}\,\widehat{\rm I}_m\,[\nabla\Theta ]^{-1}}\,  [\nabla\Theta ]\Big( \widehat{\rm I}_{m}^{-1}{\rm II}_{m}^{\flat }-\widehat{\rm I}_{y_0}^{-1}{\rm II}_{y_0}^{\flat }\Big)[\nabla\Theta ]^{-1}\,\id_3\\&=\sqrt{[\nabla\Theta ]^{-T}\,\widehat{\rm I}_m\,[\nabla\Theta ]^{-1}}\,  [\nabla\Theta ]\Big( \widehat{\rm I}_{m}^{-1}{\rm II}_{m}^{\flat }-\widehat{\rm I}_{y_0}^{-1}{\rm II}_{y_0}^{\flat }\Big)[\nabla\Theta ]^{-1}\sqrt{[\nabla\Theta ]^{-T}\,\widehat{\rm I}_{y_0}[\nabla\Theta ]^{-1}}.\notag
\end{align}
Therefore, our particular cases follow the recent trends  of considering new shell models which are similar (but not equivalent) to the Koiter shell model with the aim to lead  to improved modelling results, especially for not so thin shells. Of course, for in-extensional deformations
${\rm I}_m={\rm I}_{y_0}$ (pure bending mode, flexure), our change of curvature tensor \eqref {eq:gleichunten} turns into
\newpage
\begin{align}
&\sqrt{[\nabla\Theta ]^{-T}\,\widehat{\rm I}_{y_0}\,[\nabla\Theta ]^{-1}} [\nabla\Theta ]\Big( \widehat{\rm I}_{y_0}^{-1}{\rm II}_{m}^{\flat }-\widehat{\rm I}_{y_0}^{-1}{\rm II}_{y_0}^{\flat }\Big)[\nabla\Theta ]^{-1}\notag\\
~&=\sqrt{  [\nabla\Theta ]\,\widehat{\rm I}_{y_0}^{-1}[\nabla\Theta ]^{T}}[\nabla\Theta ]^{-T}\Big( {\rm II}_{m}^{\flat }-{\rm II}_{y_0}^{\flat }\Big)[\nabla\Theta ]^{-1}=[\nabla\Theta ]^{-T}\Big( {\rm II}_{m}^{\flat }-{\rm II}_{y_0}^{\flat }\Big)\, [\nabla\Theta ]^{-1}=[\nabla\Theta ]^{-T}\mathcal{R}_{\rm{Koiter}} \, [\nabla\Theta ]^{-1}\notag
\end{align}
and coincides with the curvature tensor considered  in the Koiter shell modell \eqref{Ap7matrix}, cf.~\cite{malin2018nonlinear}.

In a forthcoming paper \cite{GhibaNeffPartIV} we will see that the linearization of the strain measures of our constrained Cosserat shell model naturally leads to the same strain measures that are   preferred  in the later works by Sanders and Budiansky \cite{budiansky1962best,budiansky1963best} and by Koiter and Simmonds \cite{koiter1973foundations}, who called the resulting theory  the ``best first-order linear elastic shell theory''. 

In \cite{GhibaNeffPartI} the   geometrically nonlinear constrained Cosserat shell model including terms up to order $O(h^5)$  is constructed under the assumption
\begin{align}\label{ch5i}
 h\,\max \{\sup_{(x_1,x_2)\in {\omega}}|{\kappa_1}|,\sup_{(x_1,x_2)\in {\omega}}|{\kappa_2}|\}<\frac{1}{2},\end{align}
where  $\kappa_1,\kappa_2$ are the  principal curvatures. Condition \eqref{ch5i} guarantees that $
\det \nabla \Theta(x_3)=1-2\,{\rm H}\, x_3+{\rm K}\, x_3^2\neq 0
\  \text{ for all} \  x_3\in \left[-\frac{h}{2},\frac{h}{2}\right]$, i.e.,  it excludes self-intersection of the initially curved shell parametrized by $\Theta$, see  \cite[Proposition A.2.]{GhibaNeffPartI}. However, condition \eqref{ch5i} can be weakened since the classical condition \eqref{ch5in}
\begin{figure}[h!]
	\begin{minipage}{7cm}
	\centering\includegraphics[scale=1]{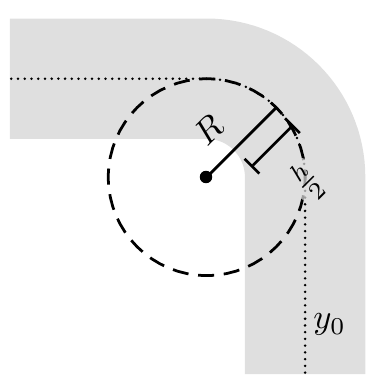}
	\end{minipage}
	\begin{minipage}{9.6cm}
			\begin{align}\label{ch5in}
		h\,\max \{\sup_{(x_1,x_2)\in {\omega}}|{\kappa_1}|,\sup_{(x_1,x_2)\in {\omega}}|{\kappa_2}|\}<2\end{align}
	\end{minipage}
	\caption{\footnotesize Left: $\frac{h}{2}<R$. Here, $R=\frac1\kappa$ denotes a typical radius of principal curvature. Right: The classical condition ensuring just injectivity of the parametrization.}	\label{PfigR}
\end{figure}
is necessary and  sufficient (see the Appendix \ref{invertAppendix})
to assure that $
\det \nabla \Theta(x_3)=1-2\,{\rm H}\, x_3+{\rm K}\, x_3^2\neq 0
\  \text{ for all} \  x_3\in \left[-\frac{h}{2},\frac{h}{2}\right]$. Therefore, without further remarks and computations, the model presented in \cite{GhibaNeffPartI} is valid under weakened conditions \eqref{ch5in} on the thickness $h$. 
Clearly, in terms of the principal radii of curvature $R_1=\frac{1}{{|\kappa_1|}}$, $R_2=\frac{1}{{|\kappa_2|}}$, see Figure \ref{PfigR}, the  condition \eqref{ch5in} is equivalent to 
\begin{align}\label{relaxedthick}
h <2\, R_1,\quad  h <2\,R_2\quad \text{in}\quad \omega \quad 
\Leftrightarrow\quad h<2\,\min \{\inf_{(x_1,x_2)\in {\omega}}{R_1},\inf_{(x_1,x_2)\in {\omega}}{R_2}\}.
\end{align}
(and not $\dd h\ll2\,\min \{\inf_{(x_1,x_2)\in {\omega}} R_1,\inf_{(x_1,x_2)\in {\omega}} R_2\}$ as is the modelling thin shell assumption for the classical Koiter model). For models which coincide to leading order with the classical Koiter model for small enough thickness \cite{anicic2018polyconvexity,anicic2019existence}, the existence of the solution is proven  under the conditions \eqref{relaxedthick}.

For the   geometrically nonlinear Cosserat shell model including terms up to order $O(h^5)$ \cite{GhibaNeffPartI},  we have  shown  the existence of the solution \cite{GhibaNeffPartII} for the theory including $O(h^5)$ terms, as well as the existence of the solution for the theory including  terms up to order $O(h^3)$.  In Appendix \ref{Appendixrelaxh} we show that condition \eqref{ch5i} on the thickness, under which the existence result presented in \cite{GhibaNeffPartII} was shown, can be weakened, but the new condition still remains more restrictive than  \eqref{ch5in}.   We noted that, in order to prove the existence of the solution,  while in the theory including $O(h^5)$ the condition on the thickness $h$ is similar to that originally considered in the  modelling process, in the sense that it is independent of the constitutive parameters,   in the $O(h^3)$-case the coercivity is proven under more restrictive conditions on the thickness $h$  which are depending on the constitutive parameters. This  remains also true  when the problem of the existence of solution in  the nonlinear constrained Cosserat shell model  is considered, as we show in Sections \ref{4.1} and \ref{4.2}.
 
Even if the condition \eqref{ch5i}  suggest that the existence of the solution is still valid for $L_c\to 0$, this turns out to be false, since the  presence of the extra Cosserat curvature energy, i.e., the condition $L_{c}>0$, is essential in our proof. The missing extra curvature energy, which in classical shell models is not present, is responsible for the typical non-well-posedness of the classical models in the nonlinear case. Note that this also applies to Naghdi-type shell models with one independent director.

This paper is now structured as follows. After fixing our notation we briefly recapitulate the unconstrained Cosserat shell model. Then in Section \ref{fin.1} we turn our attention to the constraint Cosserat shell model which introduces several symmetry constraints in the model. These symmetry constraints are put into perspective with the underlying modeling of the three-dimensional problem. We prove conditional existence results since after all the problem may now be over-constrained. This motivates to introduce a modified model in Section \ref{sec:modified}, in which certain symmetry requirements are waived. Unconditional existence theorems are then presented. In Section \ref{SmC} we express the strain measures in the Cosserat shell model in terms of classical quantities and finally in Section \ref{sec:invariance} we discuss the new invariance condition for bending tensors.

 \section{The geometrically nonlinear Cosserat shell model up to $O(h^5)$}\setcounter{equation}{0}\label{Intro}

 \subsection{Notation}
   In this paper, 
 for $a,b\in\mathbb{R}^n$ we let $\bigl\langle {a},{b} \bigr\rangle _{\mathbb{R}^n}$  denote the scalar product on $\mathbb{R}^n$ with
 associated  {(squared)} vector norm $\lVert a\rVert _{\mathbb{R}^n}^2=\bigl\langle {a},{a} \bigr\rangle _{\mathbb{R}^n}$.
 The standard Euclidean scalar product on  the set of real $n\times  {m}$ second order tensors $\mathbb{R}^{n\times  {m}}$ is given by
 $\bigl\langle  {X},{Y} \bigr\rangle _{\mathbb{R}^{n\times  {m}}}={\rm tr}(X\, Y^T)$, and thus the  {(squared)} Frobenius tensor norm is
 $\lVert {X}\rVert ^2_{\mathbb{R}^{n\times  {m}}}=\bigl\langle  {X},{X} \bigr\rangle _{\mathbb{R}^{n\times  {m}}}$. The identity tensor on $\mathbb{R}^{n \times n}$ will be denoted by $\id_n$, so that
 ${\rm tr}({X})=\bigl\langle {X},{\id}_n \bigr\rangle $, and the zero matrix is denoted by $0_n$. We let ${\rm Sym}(n)$ and ${\rm Sym}^+(n)$ denote the symmetric and positive definite symmetric tensors, respectively.  We adopt the usual abbreviations of Lie-group theory, i.e., 
 ${\rm GL}(n)=\{X\in\mathbb{R}^{n\times n}\;|\det({X})\neq 0\}$ the general linear group {,} ${\rm SO}(n)=\{X\in {\rm GL}(n)| X^TX=\id_n,\,\det({X})=1\}$ with
 corresponding Lie-algebras $\mathfrak{so}(n)=\{X\in\mathbb{R}^{n\times n}\;|X^T=-X\}$ of skew symmetric tensors
 and $\mathfrak{sl}(n)=\{X\in\mathbb{R}^{n\times n}\;| \,\tr({X})=0\}$ of traceless tensors. For all $X\in\mathbb{R}^{n\times n}$ we set ${\rm sym}\, X\,=\frac{1}{2}(X^T+X)\in{\rm Sym}(n)$, $\skw\,X\,=\frac{1}{2}(X-X^T)\in \mathfrak{so}(n)$ and the deviatoric part $\dev \,X\,=X-\frac{1}{n}\;\,\tr(X)\cdot\id_n\in \mathfrak{sl}(n)$  and we have
 the orthogonal Cartan-decomposition  of the Lie-algebra
 $
 \mathfrak{gl}(n)=\{\mathfrak{sl}(n)\cap {\rm Sym}(n)\}\oplus\mathfrak{so}(n) \oplus\mathbb{R}\!\cdot\! \id_n,$ $
 X=\dev\, \sym \,X\,+ \skw\,X\,+\frac{1}{n}\,\tr(X)\, \id_n\,.
 $ For vectors $\xi,\eta\in\mathbb{R}^n$, we have the tensor product
 $(\xi\otimes\eta)_{ij}=\xi_i\,\eta_j$. A matrix having the  three  column vectors $A_1,A_2, A_3$ will be written as 
 $
 (A_1\,|\, A_2\,|\,A_3).
 $ 
 For a given matrix $M\in \mathbb{R}^{2\times 2}$ we define the {\it 3D-lifted quantities}
 \begin{align}
 \widehat{M} =\begin{footnotesize}\begin{pmatrix}
 M_{11}& M_{12}&0 \\
 M_{21}&M_{22}&0 \\
 0&0&1
 \end{pmatrix}\in \mathbb{R}^{3\times 3}
 \qquad \text{and} \qquad
 \end{footnotesize}M^\flat =\begin{footnotesize}\begin{pmatrix}
 M_{11}& M_{12}&0 \\
 M_{21}&M_{22}&0 \\
 0&0&0
 \end{pmatrix}  \end{footnotesize}
\in \mathbb{R}^{3\times 3}, \qquad M^\flat =\widehat{M}\, \id_2^\flat.
 \end{align}
  We make use of the operator $\mathrm{axl}: \mathfrak{so}(3)\to\mathbb{R}^3$ associating with a skew-symmetric matrix $A\in \mathfrak{so}(3)$ the vector $\mathrm{axl}({A})\coloneqq(-A_{23},A_{13},-A_{12})^T$.  The corresponding inverse operator will be denoted by ${\rm Anti}: \mathbb{R}^3\to \mathfrak{so}(3)$.

  For  an open domain  $\Omega\subseteq\mathbb{R}^{3}$,
 the usual Lebesgue spaces of square integrable functions, vector or tensor fields on $\Omega$ with values in $\mathbb{R}$, $\mathbb{R}^3$, $\mathbb{R}^{3\times 3}$ or ${\rm SO}(3)$, respectively will be denoted by ${\rm L}^2(\Omega;\mathbb{R})$, ${\rm L}^2(\Omega;\mathbb{R}^3)$, ${\rm L}^2(\Omega; \mathbb{R}^{3\times 3})$ and ${\rm L}^2(\Omega; {\rm SO}(3))$, respectively. Moreover, we use the standard Sobolev spaces ${\rm H}^{1}(\Omega; \mathbb{R})$ \cite{Adams75,Raviart79,Leis86}
 of functions $u$.  For vector fields $u=\left(    u_1, u_2, u_3\right)^T$ with  $u_i\in {\rm H}^{1}(\Omega)$, $i=1,2,3$,
 we define
 $
 \nabla \,u\coloneqq \left(
 \nabla\,  u_1\,|\,
 \nabla\, u_2\,|\,
 \nabla\, u_3
 \right)^T.
 $
 The corresponding Sobolev-space will be denoted by
 $
 {\rm H}^1(\Omega; \mathbb{R}^{3})$. A tensor $Q:\Omega\to {\rm SO}(3)$ having the components in ${\rm H}^1(\Omega; \mathbb{R})$ belongs to ${\rm H}^1(\Omega; {\rm SO}(3))$. For tensor fields $P$ with rows in ${\rm H}({\rm curl}\,; \Omega)$, i.e., 
 $
 P=\begin{footnotesize}\begin{pmatrix}
 P^T.e_1\,|\,
 P^T.e_2\,|\,
 P^T.e_3
 \end{pmatrix}\end{footnotesize}^T$ with $(P^T.e_i)^T\in {\rm H}({\rm curl}\,; \Omega)$, $i=1,2,3$,
 we define
 $
 {\rm Curl}\,P\coloneqq \begin{footnotesize}\begin{pmatrix}
 {\rm curl}\, (P^T.e_1)^T\,|\,
 {\rm curl}\, (P^T.e_2)^T\,|\,
 {\rm curl}\, (P^T.e_3)^T
 \end{pmatrix}\end{footnotesize}^T
 .
 $
 The corresponding Sobolev-space will be denoted by
   ${\rm H}(\Curl;\Omega)$.

 In writing the norm in the corresponding  Sobolev-space we will specify the space. The space will be omitted only when the Frobenius norm or scalar product is considered. 
 In the formulation of the minimization problem we  have considered the  {\it Weingarten map (or shape operator)}  {on $y_0(\omega)$}  defined by  {its associated matrix}
$
{\rm L}_{y_0}\,=\, {\rm I}_{y_0}^{-1} {\rm II}_{y_0}\in \mathbb{R}^{2\times 2},
$
where ${\rm I}_{y_0}\coloneqq [{\nabla  y_0}]^T\,{\nabla  y_0}\in \mathbb{R}^{2\times 2}$ and  ${\rm II}_{y_0}\coloneqq\,-[{\nabla  y_0}]^T\,{\nabla  n_0}\in \mathbb{R}^{2\times 2}$ are  the matrix representations of the {\it first fundamental form (metric)} and the  {\it  second fundamental form} of the surface  {$y_0(\omega)$}, respectively.  
Then, the {\it Gau{\ss} curvature} ${\rm K}$ of the surface  {$y_0(\omega)$} is determined by
$
{\rm K} \coloneqq \,{\rm det}{({\rm L}_{y_0})}\, 
$
and the {\it mean curvature} $\,{\rm H}\,$ through
$
2\,{\rm H}\, \coloneqq {\rm tr}({{\rm L}_{y_0}}).
$  We have also used  the  tensors defined by
\begin{align}\label{AB}
{\rm A}_{y_0}&\coloneqq (\nabla y_0|0)\,\,[\nabla\Theta ]^{-1}\in\mathbb{R}^{3\times 3}, \qquad \qquad 
{\rm B}_{y_0}\coloneqq -(\nabla n_0|0)\,\,[\nabla\Theta ]^{-1}\in\mathbb{R}^{3\times 3},
\end{align}
and the so-called \textit{{alternator tensor}} ${\rm C}_{y_0}$ of the surface \cite{Zhilin06}
\begin{align}
{\rm C}_{y_0}\coloneqq \det(\nabla\Theta )\, [	\nabla\Theta ]^{-T}\,\begin{footnotesize}\begin{pmatrix}
0&1&0 \\
-1&0&0 \\
0&0&0
\end{pmatrix}\end{footnotesize}\,  [	\nabla\Theta ]^{-1}.
\end{align}

 \subsection{The unconstrained Cosserat shell model}\label{sec:unCoss}
 In   \cite{GhibaNeffPartI}, we have obtained the following two-dimensional minimization problem   for the
 deformation of the midsurface $m:\omega
 \,{\to}\,
 \mathbb{R}^3$ and the microrotation of the shell
 $\overline{Q}_{e,s}:\omega
 \,{\to}\,
 \textrm{SO}(3)$ solving on $\omega
 \,\subset\mathbb{R}^2
 $: minimize with respect to $ (m,\overline{Q}_{e,s}) $ the  functional
 \begin{equation}\label{e89}
 I(m,\overline{Q}_{e,s})\!=\!\! \int_{\omega}   \!\!\Big[  \,
 W_{\mathrm{memb}}\big(  \mathcal{E}_{m,s}  \big) +W_{\mathrm{memb,bend}}\big(  \mathcal{E}_{m,s} ,\,  \mathcal{K}_{e,s} \big)   +
 W_{\mathrm{bend,curv}}\big(  \mathcal{K}_{e,s}    \big)
 \Big] \,{\rm det}\nabla\Theta        \, d a - \overline{\Pi}(m,\overline{Q}_{e,s})\,,
 \end{equation}
 where the  membrane part $\,W_{\mathrm{memb}}\big(  \mathcal{E}_{m,s} \big) \,$, the membrane--bending part $\,W_{\mathrm{memb,bend}}\big(  \mathcal{E}_{m,s} ,\,  \mathcal{K}_{e,s} \big) \,$ and the bending--curvature part $\,W_{\mathrm{bend,curv}}\big(  \mathcal{K}_{e,s}    \big)\,$ of the shell energy density are given by
 \begin{align}\label{e90}
 W_{\mathrm{memb}}\big(  \mathcal{E}_{m,s} \big)=& \, \Big(h+{\rm K}\,\dfrac{h^3}{12}\Big)\,
 W_{\mathrm{shell}}\big(    \mathcal{E}_{m,s} \big),\vspace{2.5mm}\notag\\    
 W_{\mathrm{memb,bend}}\big(  \mathcal{E}_{m,s} ,\,  \mathcal{K}_{e,s} \big)=& \,   \Big(\dfrac{h^3}{12}\,-{\rm K}\,\dfrac{h^5}{80}\Big)\,
 W_{\mathrm{shell}}  \big(   \mathcal{E}_{m,s} \, {\rm B}_{y_0} +   {\rm C}_{y_0} \mathcal{K}_{e,s} \big)  \\&
 -\dfrac{h^3}{3} \mathrm{ H}\,\mathcal{W}_{\mathrm{shell}}  \big(  \mathcal{E}_{m,s} ,
 \mathcal{E}_{m,s}{\rm B}_{y_0}+{\rm C}_{y_0}\, \mathcal{K}_{e,s} \big)+
 \dfrac{h^3}{6}\, \mathcal{W}_{\mathrm{shell}}  \big(  \mathcal{E}_{m,s} ,
 ( \mathcal{E}_{m,s}{\rm B}_{y_0}+{\rm C}_{y_0}\, \mathcal{K}_{e,s}){\rm B}_{y_0} \big)\vspace{2.5mm}\notag\\&+ \,\dfrac{h^5}{80}\,\,
 W_{\mathrm{mp}} \big((  \mathcal{E}_{m,s} \, {\rm B}_{y_0} +  {\rm C}_{y_0} \mathcal{K}_{e,s} )   {\rm B}_{y_0} \,\big),  \vspace{2.5mm}\notag\\
 W_{\mathrm{bend,curv}}\big(  \mathcal{K}_{e,s}    \big) = &\,  \,\Big(h-{\rm K}\,\dfrac{h^3}{12}\Big)\,
 W_{\mathrm{curv}}\big(  \mathcal{K}_{e,s} \big)    +  \Big(\dfrac{h^3}{12}\,-{\rm K}\,\dfrac{h^5}{80}\Big)\,
 W_{\mathrm{curv}}\big(  \mathcal{K}_{e,s}   {\rm B}_{y_0} \,  \big)  + \,\dfrac{h^5}{80}\,\,
 W_{\mathrm{curv}}\big(  \mathcal{K}_{e,s}   {\rm B}_{y_0}^2  \big),\notag
 \end{align}
 with
 \begin{align}\label{quadraticforms}
 W_{\mathrm{shell}}( X) & =   \mu\,\lVert  \mathrm{sym}\,X\rVert ^2 +  \mu_{\rm c}\lVert \mathrm{skew}\,X\rVert ^2 +\dfrac{\lambda\,\mu}{\lambda+2\,\mu}\,\big[ \mathrm{tr}   (X)\big]^2\notag\\
 &=  \mu\, \lVert  \mathrm{  dev \,sym} \,X\rVert ^2  +  \mu_{\rm c} \lVert  \mathrm{skew}   \,X\rVert ^2 +\,\dfrac{2\,\mu\,(2\,\lambda+\mu)}{3(\lambda+2\,\mu)}\,[\mathrm{tr}  (X)]^2,\\
 \mathcal{W}_{\mathrm{shell}}(  X,  Y)& =   \mu\,\bigl\langle  \mathrm{sym}\,X,\,\mathrm{sym}\,   \,Y \bigr\rangle   +  \mu_{\rm c}\bigl\langle \mathrm{skew}\,X,\,\mathrm{skew}\,   \,Y \bigr\rangle   +\,\dfrac{\lambda\,\mu}{\lambda+2\,\mu}\,\mathrm{tr}   (X)\,\mathrm{tr}   (Y),  \notag\vspace{2.5mm}\\
 W_{\mathrm{mp}}(  X)&= \mu\,\lVert \mathrm{sym}\,X\rVert ^2+  \mu_{\rm c}\lVert \mathrm{skew}\,X\rVert ^2 +\,\dfrac{\lambda}{2}\,\big[  \tr(X)\,\big]^2=
  {W}_{\mathrm{shell}}(  X)+ \,\dfrac{\lambda^2}{2\,(\lambda+2\,\mu)}\,[\mathrm{tr} (X)]^2,\notag\vspace{2.5mm}\\
 W_{\mathrm{curv}}(  X )&=\mu\, L_{\rm c}^2 \left( b_1\,\lVert  \dev\,\text{sym} \,X\rVert ^2+b_2\,\lVert \text{skew}\,X\rVert ^2+b_3\,
 [\tr (X)]^2\right), \qquad\qquad \forall\, X,Y\in \mathbb{R}^{3\times 3}.\notag
 \end{align}
 
 Contrary to other 6-parameter theory of shells  \cite{Pietraszkiewicz79,Pietraszkiewicz-book04,Eremeyev06,Pietraszkiewicz10} the membrane-bending energy is expressed in terms of the specific  tensor $\mathcal{E}_{m,s} \, {\rm B}_{y_0} +  {\rm C}_{y_0} \mathcal{K}_{e,s}$. This tensor represents a nonlinear change of curvature tensor, since in the linear constrained Cosserat model \cite{GhibaNeffPartIV} it reduces to the change of curvature tensor considered by Anicic and L\'{e}ger \cite{anicic1999formulation} and more recently by {\v{S}}ilhav{\`y} \cite {vsilhavycurvature}.

 The parameters $\mu$ and $\lambda$ are the \textit{Lam\'e constants}
 of classical isotropic elasticity, $\kappa=\frac{2\,\mu+3\,\lambda}{3}$ is the \textit{infinitesimal bulk modulus}, $b_1, b_2, b_3$ are \textit{non-dimensional constitutive curvature coefficients (weights)}, $\mu_{\rm c}\geq 0$ is called the \textit{{Cosserat couple modulus}} and ${L}_{\rm c}>0$ introduces an \textit{{internal length} } which is {characteristic} for the material, e.g., related to the grain size in a polycrystal. The
 internal length ${L}_{\rm c}>0$ is responsible for \textit{size effects} in the sense that thinner samples are relatively stiffer than
 thicker samples. If not stated otherwise, we assume that $\mu>0$, $\kappa>0$, $\mu_{\rm c}>0$, $b_1>0$, $b_2>0$, $b_3> 0$. All the constitutive coefficients  are now deduced from the three-dimensional formulation, without using any a posteriori fitting of some two-dimensional constitutive coefficients.

 The potential of applied external loads $ \overline{\Pi}(m,\overline{Q}_{e,s}) $ appearing in \eqref{e89} is expressed by 
 \begin{align}\label{e4o}
 \overline{\Pi}(m,\overline{Q}_{e,s})\,=\,& \, \Pi_\omega(m,\overline{Q}_{e,s}) + \Pi_{\gamma_t}(m,\overline{Q}_{e,s})\,,\qquad\textrm{with}   \\
 \Pi_\omega(m,\overline{Q}_{e,s}) \,=\,& \dd\int_{\omega}\bigl\langle  {f}, u \bigr\rangle \, da + \Lambda_\omega(\overline{Q}_{e,s})\qquad \text{and}\qquad
 \Pi_{\gamma_t}(m,\overline{Q}_{e,s})\,=\, \dd\int_{\gamma_t}\bigl\langle  {t},  u \bigr\rangle \, ds + \Lambda_{\gamma_t}(\overline{Q}_{e,s})\,,\notag
 \end{align}
 where $ u(x_1,x_2) \,=\, m(x_1,x_2)-y_0(x_1,x_2) $ is the displacement vector of the midsurface,  $\Pi_\omega(m,\overline{Q}_{e,s})$ is the potential of the external surface loads $f$, while  $\Pi_{\gamma_t}(m,\overline{Q}_{e,s})$ is the potential of the external boundary loads $t$.  The functions $\Lambda_\omega\,, \Lambda_{\gamma_t} : {\rm L}^2 (\omega, \textrm{SO}(3))\rightarrow\mathbb{R} $ are expressed in terms of the loads from the three-dimensional parental variational problem \cite{GhibaNeffPartI} and they are assumed to be continuous and bounded operators. Here, $ \gamma_t $ and $ \gamma_d $ are nonempty subsets of the boundary of $ \omega $ such that $   \gamma_t \cup \gamma_d= \partial\omega $ and $ \gamma_t \cap \gamma_d= \emptyset $\,. On $ \gamma_t $ we have considered traction boundary conditions, while on $ \gamma_d $ we have the Dirichlet-type boundary conditions: \begin{align}\label{boundary}
 m\big|_{\gamma_d}&=m^* \qquad \text{simply supported (fixed, welded)}, \qquad \qquad \qquad 
 \overline{Q}_{e,s}\big|_{\gamma_d}=\overline{Q}_{e,s}^*,\ \  \ \ \ \ \ \ \text{clamped},
 \end{align}
 where the boundary conditions are to be understood in the sense of traces.

    In our model the total energy is not simply the sum of  energies coupling the pure membrane and the pure bending effect, respectively.  Two mixed coupling energies are still present after the dimensional reduction of the variational problem from the geometrically nonlinear three-dimensional Cosserat elasticity.  
 	
Considering   materials for  which $\,\mu_{\rm c}>0$ and the Poisson ratio  $\nu=\frac{\lambda}{2(\lambda+\mu)}$ and Young's modulus $ E=\frac{\mu(3\,\lambda+2\,\mu)}{\lambda+\mu}$ are such that\footnote{These conditions are equivalent to $\mu>0$ and $2\,\lambda+\mu> 0$, which assure the positive definiteness of the quadratic form $ W_{\mathrm{shell}}( X)  =   \mu\,\lVert  \mathrm{sym}\,X\rVert ^2 +  \mu_{\rm c}\lVert \mathrm{skew}\,X\rVert ^2 +\dfrac{\lambda\,\mu}{\lambda+2\,\mu}\,\big[ \mathrm{tr}   (X)\big]^2
	=  \mu\, \lVert  \mathrm{  dev \,sym} \,X\rVert ^2  +  \mu_{\rm c} \lVert  \mathrm{skew}   \,X\rVert ^2 +\,\dfrac{2\,\mu\,(2\,\lambda+\mu)}{3(\lambda+2\,\mu)}\,[\mathrm{tr}  (X)]^2$.}
$
-\frac{1}{2}<\nu<\frac{1}{2}$ \text{and} $E>0\, 
$ the Cosserat shell model admits global minimizers \cite{GhibaNeffPartII}. 
Under these assumptions on the constitutive coefficients, together with the positivity of $\mu$, $\mu_{\rm c}$, $b_1$, $b_2$ and $b_3$, and the orthogonal Cartan-decomposition  of the Lie-algebra
$
\mathfrak{gl}(3)$ and with the definition
\begin{align}\label{e78}
{W}_{\mathrm{shell}}( X)
\coloneqq &\, {W}_{\mathrm{shell}}^{\infty}( \sym \,X) +  \mu_{\rm c} \lVert  \mathrm{skew}   \,X\rVert ^2 \quad \ \forall \, X\in\mathbb{R}^{3\times 3},\\  {W}_{\mathrm{shell}}^{\infty}( S)\coloneqq &\,  \mu\, \lVert  S\rVert ^2   +\,\dfrac{\lambda\,\mu}{\lambda+2\,\mu}\,\big[ \mathrm{tr}   (S)\big]^2\qquad \quad \, \forall \, S\,\in{\rm Sym}(3),\notag
\end{align}
it follows that there exists positive constants  $c_1^+, c_2^+, C_1^+$ and $C_2^+$  such that for all $X\in \mathbb{R}^{3\times 3}$ the following inequalities hold
\begin{align}\label{pozitivdef}
C_1^+ \lVert S\rVert ^2&\geq\, {W}_{\mathrm{shell}}^{\infty}( S) \geq\, c_1^+ \lVert  S\rVert ^2 \qquad \qquad \qquad \qquad \qquad \qquad\ \forall \, S\,\in{\rm Sym}(3),\notag\\
C_1^+ \lVert \sym\,X\rVert ^2+\mu_{\rm c}\,\lVert \skw\,X\rVert ^2&\geq\, W_{\mathrm{shell}}(  X) \geq\, c_1^+ \lVert  \sym\,X\rVert ^2+\mu_{\rm c}\,\lVert \skw\,X\rVert ^2 \quad \qquad\forall \, X\in\mathbb{R}^{3\times 3},\\
C_2^+ \lVert X \rVert ^2 &\geq\, W_{\mathrm{curv}}(  X )
\geq\,  c_2^+ \lVert X \rVert ^2\qquad \qquad \qquad \qquad \qquad \qquad  \forall \, X\in\mathbb{R}^{3\times 3}.\notag
\end{align}
Here,  $c_1^+$ and $C_1^+$ denote the  smallest and the largest eigenvalues, respectively, of the quadratic form ${W}_{\mathrm{shell}}^{\infty}( X)$. Hence, they are independent of $\mu_{\rm c}$.

\section{The limit problem for infinite Cosserat couple  modulus $\mu_{\rm c}\to \infty$}\label{fin.1}\setcounter{equation}{0}

In this section we consider the case $\mu_c\rightarrow\infty$, since then the constraint $\overline{R}_\xi={\rm polar}(F_\xi)$ is enforced in the starting three-dimensional variational problem. In that case, the parental three-dimensional model turns into the Toupin couple stress model \cite[Eq. 11.8]{Toupin62}.
\subsection{Constrained elastic Cosserat shell models}\label{constcon}
 Let us see what is happening to the dimensionally reduced model under the same circumstances. The following lemma give{s} us information on the constitutive restrictions under which the energy density remains bounded. \begin{lemma}\label{propcoerh5}  For sufficiently small values of the thickness $h$ such that  	\begin{align}\label{rcondh5new}
 	h\max\{\sup_{x\in\omega}|\kappa_1|, \sup_{x\in\omega}|\kappa_2|\}<\alpha \qquad \text{with}\qquad  \alpha<\sqrt{\frac{2}{3}(29-\sqrt{761})}\simeq 0.97083
 	\end{align}  
	and for constitutive coefficients  satisfying  $\mu>0, \,\mu_{\rm c}>0$, $2\,\lambda+\mu> 0$, $b_1>0$, $b_2>0$ and $b_3>0$,   the  energy density 
	\begin{align}W(\mathcal{E}_{m,s}, \mathcal{K}_{e,s})=W_{\mathrm{memb}}\big(  \mathcal{E}_{m,s} \big)+W_{\mathrm{memb,bend}}\big(  \mathcal{E}_{m,s} ,\,  \mathcal{K}_{e,s} \big)+W_{\mathrm{bend,curv}}\big(  \mathcal{K}_{e,s}    \big)
	\end{align}
	satisfies the estimate
	\begin{align}\label{26bis}
	W(\mathcal{E}_{m,s}, \mathcal{K}_{e,s})
	\geq&\,h\,\dfrac{7}{48}\,  {W}_{\mathrm{shell}}  \big(  \mathcal{E}_{m,s}\big)+\dfrac{h^3}{12}\,\dfrac{37}{80}\, {W}_{\mathrm{shell}}  \big(
	\mathcal{E}_{m,s}{\rm B}_{y_0}+{\rm C}_{y_0}\, \mathcal{K}_{e,s}  \big)\\&+\dfrac{h^5}{80}\,\frac{1}{6}
	W_{\mathrm{shell}} \big((  \mathcal{E}_{m,s} \, {\rm B}_{y_0} +  {\rm C}_{y_0} \mathcal{K}_{e,s} )   {\rm B}_{y_0} \,\big)+h\,\dfrac{47}{48}\,
	W_{\mathrm{curv}}\big(  \mathcal{K}_{e,s} \big).\notag
	\end{align} 
\end{lemma}
\begin{proof}
	The proof is  similar to the proof of Theorem 4.1 from \cite{GhibaNeffPartII} and the proof given in Appendix \ref{coerh5Appendix}. 
\end{proof}

In the remainder of this paper, the strain tensors corresponding to the constrained elastic Cosserat shell models will be denoted with the subscript  $\cdot_{ \infty }$, i.e., ${Q}_{ \infty }$, $\mathcal{E}_{ \infty }$, $\mathcal{K}_{ \infty }$ etc.

Therefore, since the membrane energy and the membrane-bending energy have to remain finite when \linebreak $\mu_c\rightarrow\infty$, in view of Lemma \ref{propcoerh5}  and \eqref{pozitivdef} we  have to impose \cite{Toupin62,Toupin64} that \begin{align}
\label{restsk}\lVert{\rm skew}( \mathcal{E}_{ \infty })\rVert=0
,
\qquad\quad 
\lVert{\rm skew}(\mathcal{E}_{ \infty } \, {\rm B}_{y_0} +  {\rm C}_{y_0} \mathcal{K}_{ \infty } )\rVert=0,\quad \qquad 
\lVert{\rm skew}((\mathcal{E}_{ \infty } \, {\rm B}_{y_0} +  {\rm C}_{y_0} \mathcal{K}_{ \infty } )   {\rm B}_{y_0})\rVert=0.
\end{align}
The
assumption $\mathcal{E}_{ \infty }\in {\rm Sym}(3)$ implies that 
$
{Q}_{ \infty }Q_0={\rm polar}(\nabla m|{n})\in\textrm{SO}(3),
$ 
where $n=\frac{\partial_{x_1} m\times \partial_{x_2} m}{\lVert\partial_{x_1} m\times \partial_{x_2} m\rVert}$ is the unit normal vector to the deformed midsurface, and ${\rm polar}(X)\in\textrm{SO}(3)$  denotes the orthogonal part of the invertible matrix  $X\in {\rm GL}^+(3)$ in the polar decomposition \cite{neff2013grioli}, i.e., $X={\rm polar}(X)\, U$ with $U\in {\rm Sym}^+(3)$. 
Indeed, we have
\begin{align}\label{lcmc}
{\rm skew}(\mathcal{E}_{ \infty }) =&\,\,0\notag\\
\Leftrightarrow  \quad  &{Q}_{ \infty }^T(\nabla  m| \,{Q}_{ \infty } \nabla\Theta .e_3)[\nabla\Theta ]^{-1}\in {\rm Sym}(3)\notag \\
\Leftrightarrow  \quad  &{Q}_{ \infty }^T(\nabla  m| \,{Q}_{ \infty } \nabla\Theta .e_3)[\nabla\Theta ]^{-1} \, = \,[\nabla\Theta ]^{-T}(\nabla  m| \,{Q}_{ \infty } \nabla\Theta .e_3)^T{Q}_{ \infty }\notag \\
\Leftrightarrow   \quad  & Q_0^T{Q}_{ \infty }^T(\nabla  m| \,{Q}_{ \infty } \nabla\Theta .e_3)U_0^{-1} \, =  \,U_0^{-1}(\nabla  m| \,{Q}_{ \infty } \nabla\Theta .e_3)^T{Q}_{ \infty }Q_0\notag \\
\Leftrightarrow  \quad  & U_0 Q_0^T{Q}_{ \infty }^T(\nabla  m| \,{Q}_{ \infty } \nabla\Theta .e_3)U_0^{-1} \, = \, (\nabla  m| \,{Q}_{ \infty } \nabla\Theta .e_3)^T{Q}_{ \infty }Q_0\\
\Leftrightarrow  \quad  & \begin{footnotesize}\begin{pmatrix}
*&*&0 \\
*&*&0 \\
0&0&1 \\
\end{pmatrix}\end{footnotesize} Q_0^T{Q}_{ \infty }^T(\nabla  m| \,{Q}_{ \infty } \nabla\Theta .e_3)\begin{footnotesize}\begin{pmatrix}
*&*&0 \\
*&*&0 \\
0&0&1 \\
\end{pmatrix}\end{footnotesize}=(\nabla  m| \,{Q}_{ \infty } \nabla\Theta .e_3)^T{Q}_{ \infty }Q_0\notag \\
\Leftrightarrow  \quad  & \begin{footnotesize}\begin{pmatrix}
*&*&0 \\
*&*&0 \\
0&0&1 \\
\end{pmatrix}\end{footnotesize} \begin{footnotesize}\begin{pmatrix}
*&*&0 \\
*&*&0 \\
*&*&1 \\
\end{pmatrix}\end{footnotesize}\begin{footnotesize}\begin{pmatrix}
*&*&0 \\
*&*&0 \\
0&0&1 \\
\end{pmatrix}\end{footnotesize}=(\nabla  m| \,{Q}_{ \infty } \nabla\Theta .e_3)^T{Q}_{ \infty }Q_0\notag \\
\Leftrightarrow  \quad  & \begin{footnotesize}\begin{pmatrix}
*&*&0 \\
*&*&0 \\
*&*&1 \\
\end{pmatrix}\end{footnotesize}\begin{footnotesize}\begin{pmatrix}
*&*&0 \\
*&*&0 \\
0&0&1 \\
\end{pmatrix}\end{footnotesize} =(\nabla  m| \,{Q}_{ \infty } \nabla\Theta .e_3)^T{Q}_{ \infty }Q_0\quad 
\Leftrightarrow  \quad   \begin{footnotesize}\begin{pmatrix}
*&*&0 \\
*&*&0 \\
*&*&1 \\
\end{pmatrix}\end{footnotesize} =\begin{footnotesize}\begin{pmatrix}
*&*&\bigl\langle   \partial_{x_1}m,{Q}_{ \infty }Q_0.e_3\bigr\rangle   \\
*&*&\bigl\langle   \partial_{x_2}m,{Q}_{ \infty }Q_0.e_3\bigr\rangle   \\
0&0&1
\end{pmatrix}\end{footnotesize}\, .\notag
\end{align}
Thus, the vector ${Q}_{ \infty }Q_0.e_3$ is collinear with $n$\,. Moreover, since  ${Q}_{ \infty }Q_0\in {\rm SO}(3)$, we have ${Q}_{ \infty }Q_0.e_3=n$.

Hence, $
\mathcal{E}_{ \infty } ={Q}_{ \infty }^T (\nabla  m|{Q}_{ \infty }\nabla\Theta .e_3)[\nabla\Theta ]^{-1}-\id_3\in{\rm Sym}(3) $ implies
$
{Q}_{ \infty }^T(\nabla  m|n)[\nabla\Theta ]^{-1}\in {\rm Sym}(3),
$
and therefore \begin{align}
{Q}_{ \infty }={\rm polar}\big((\nabla  m|n) [\nabla\Theta ]^{-1}\big)=(\nabla m|n)[\nabla\Theta ]^{-1}\,\widetilde{U}^{-1},
\end{align}
where $\widetilde{U}\in {\rm Sym}^+(3)$ is defined by
\begin{align}
(\nabla m|n)[\nabla\Theta ]^{-1}=[{\rm polar}\big((\nabla  m|n) [\nabla\Theta ]^{-1}\big)] \, \widetilde{U}.
\end{align}

The constraint ${Q}_{ \infty }={\rm polar}\big((\nabla  m|n) [\nabla\Theta ]^{-1}\big)$ not only adjusts the ``trièdre mobile'' \cite{Cosserat09} to be tangential to the surface (equivalent to ${Q}_{ \infty } Q_0.e_3=n$, Kirchhoff-Love normality assumption), see Figure \ref{Pfig5}
\begin{figure}[h!]
	\centering\includegraphics[scale=1]{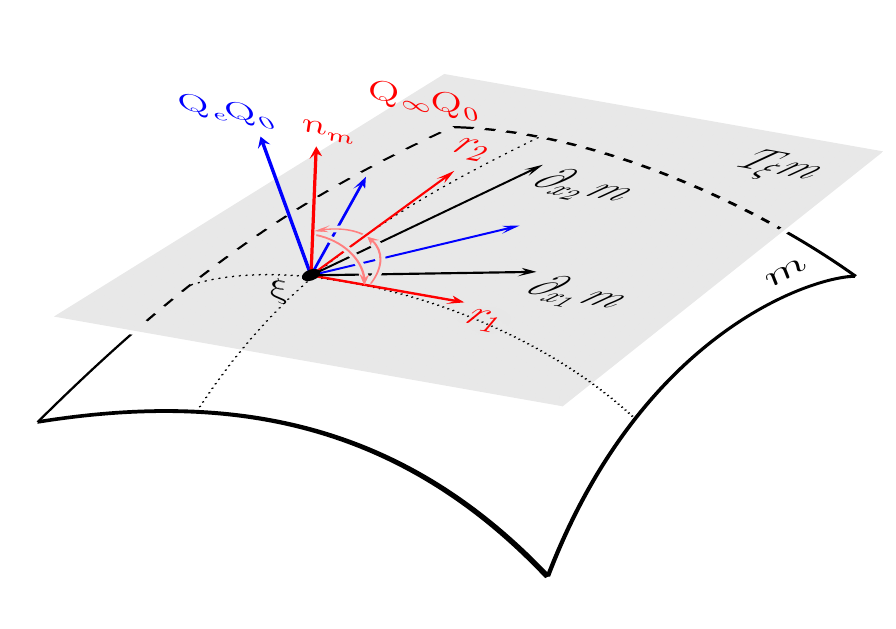}
	\caption{\footnotesize  In the limit $\boldsymbol{\mu_{\rm c}\to 
		\infty}$, the blue $3$-frame (trièdre mobile) is replaced by the red $3$-frame (trièdre caché) being tangent to the surface and prescribing a specific in-plane drill rotation.}
\label{Pfig5}\end{figure}
such that the third column coincides with the normal to the surface, but also chooses a specific {\bf in-plane drill rotation}, coming from the planar polar decomposition.

Using 
$
\widetilde{U}^2=\widetilde{U}^T\, \widetilde{U}=\big((\nabla m|n)[\nabla\Theta ]^{-1}\big)^T (\nabla m|n)[\nabla\Theta ]^{-1}
=[\nabla\Theta ]^{-T}\,\widehat {\rm I}_{m}\,[\nabla\Theta ]^{-1},
$
we deduce
\begin{align}
{Q}_{ \infty }={\rm polar}\big((\nabla  m|n) [\nabla\Theta ]^{-1}\big)=(\nabla m|n)[\nabla\Theta ]^{-1}\,\sqrt{[\nabla\Theta ]\,\widehat {\rm I}_{m}^{-1}\,[\nabla\Theta ]^{T}},
\end{align}
with the lifted quantity $\widehat{\rm I}_{m} \in \mathbb{R}^{3\times 3}$   given by
$
\widehat{\rm I}_{m}\coloneqq ({\nabla  m}|n)^T({\nabla  m}|n).
$
 Moreover,    since 
$
{Q}_{ \infty }\nabla\Theta .e_3={Q}_{ \infty }Q_0.e_3=n,
$
the latter identity leads to
\begin{align}\label{miu}
\mathcal{E}_{ \infty }&= [{\rm polar}\big((\nabla  m|n) [\nabla\Theta ]^{-1}\big)]^T(\nabla m|n)[\nabla\Theta ]^{-1}-\id_3= \widetilde{U}-\id_3= \sqrt{[\nabla\Theta ]^{-T}\,\widehat{\rm I}_m[\nabla\Theta ]^{-1}}-\id_3.
\end{align}
An alternative expression of $ \mathcal{E}_{ \infty }$ containing only classical quantities can be obtained using Remark \ref{propAB} from Appendix \ref{ApropAB}, i.e., 
$
{\rm A}_{y_0}^2={\rm A}_{y_0}=(\nabla y_0|0) \,[\nabla\Theta ]^{-1}=
[\nabla\Theta ]^{-T}\,\widehat{\rm I}_{y_0}\,\id_2^{\flat }\,[\nabla\Theta ]^{-1}=
[\nabla\Theta ]^{-T}\,{\rm I}_{y_0}^{\flat }\,[\nabla\Theta ]^{-1}.
$

In the following, we proceed to compute $\mathcal{E}_{ \infty } \, {\rm B}_{y_0} +  {\rm C}_{y_0} \mathcal{K}_{ \infty }$ which is appearing in the expression of the membrane-bending energy.  
From Remark \ref{propAB} we note the identity   \begin{align}
{Q}_{ \infty }^T\,(\nabla [{Q}_{ \infty }\nabla \Theta (0).e_3]\,|\,0)\,[\nabla\Theta ]^{-1}={\rm C}_{y_0} \mathcal{K}_{ \infty }-{\rm B}_{y_0},
\end{align}
 which, together with the other identities from   Remark \ref{propAB}, leads to
\begin{align}\label{CK}
{\rm C}_{y_0} \mathcal{K}_{ \infty } = &\,{Q}_{ \infty }^T\,(\nabla [{Q}_{ \infty }\nabla \Theta (0).e_3]\,|\,0)\,[\nabla\Theta ]^{-1}+{\rm B}_{y_0}
={Q}_{ \infty }^T(\nabla[ {Q}_{ \infty }Q_0.e_3]|0)[\nabla\Theta ]^{-1}+{\rm B}_{y_0}\notag\vspace{2.5mm}\\
=&\,{Q}_{ \infty }^T(\nabla n|0)[\nabla\Theta ]^{-1}+{\rm B}_{y_0}=[{\rm polar}\big((\nabla  m|n) [\nabla\Theta ]^{-1}\big)]^T (\nabla {n}|0)[\nabla\Theta ]^{-1}+{\rm B}_{y_0}\notag\vspace{2.5mm}\\
=&\,\Big[\,(\nabla m|n)[\nabla\Theta ]^{-1} \sqrt{[\nabla\Theta ]\,\widehat{\rm I}_m^{-1}[\nabla\Theta ]^{T}}\Big]^T (\nabla {n}|0)[\nabla\Theta ]^{-1}+{\rm B}_{y_0}\vspace{2.5mm}\\
=&\,-\sqrt{[\nabla\Theta ]\,\widehat{\rm I}_m^{-1}[\nabla\Theta ]^{T}}\,[\nabla\Theta ]^{-T} {\rm II}_m^\flat[\nabla\Theta ]^{-1} +\sqrt{[\nabla\Theta ]\,\widehat{\rm I}_{y_0}^{-1}[\nabla\Theta ]^{T}}[\nabla\Theta ]^{-T}{\rm II}_{y_0}^\flat [\nabla\Theta ]^{-1}\notag\notag
.\notag
\end{align}

Because ~ $
\sqrt{[\nabla\Theta ]^{-T}\,\widehat{\rm I}_m\,[\nabla\Theta ]^{-1}\,[\nabla\Theta ]\,\widehat{\rm I}_m^{-1}[\nabla\Theta ]^{T}}=\,\sqrt{[\nabla\Theta ]\,\widehat{\rm I}_m^{-1}[\nabla\Theta ]^{T}\,[\nabla\Theta ]^{-T}\,\widehat{\rm I}_m\,[\nabla\Theta ]^{-1}}$, ~
it follows \\ $\sqrt{[\nabla\Theta ]^{-T}\,\widehat{\rm I}_m\,[\nabla\Theta ]^{-1}}\,[\nabla\Theta ]\,\widehat{\rm I}_m^{-1}[\nabla\Theta ]^{T}=\sqrt{[\nabla\Theta ]\,\widehat{\rm I}_m^{-1}[\nabla\Theta ]^{T}}$.
Using also the two formulae ~
$\widehat{\rm I}_{y_0}^{-1}{\rm II}_{y_0}^\flat ={\rm L}_{y_0}^\flat$,\\ $  \widehat{\rm I}_{m}^{-1}{\rm II}_{m}^\flat ={\rm L}_{m}^\flat$, we may write
\begingroup
\allowdisplaybreaks
\begin{align}
\mathcal{E}_{ \infty } \, {\rm B}_{y_0} +  {\rm C}_{y_0} \mathcal{K}_{ \infty }=&\,\Big(\sqrt{[\nabla\Theta ]^{-T}\,\widehat{\rm I}_m\,[\nabla\Theta ]^{-1}}-\id_3\Big)  [\nabla\Theta ]^{-T}{\rm II}_{y_0}^\flat [\nabla\Theta ]^{-1}\notag\\
&\,-\sqrt{[\nabla\Theta ]\,\widehat{\rm I}_m^{-1}[\nabla\Theta ]^{T}}\,[\nabla\Theta ]^{-T} {\rm II}_m^\flat[\nabla\Theta ]^{-1}+[\nabla\Theta ]^{-T}{\rm II}_{y_0}^\flat [\nabla\Theta ]^{-1}\notag\\
=&\,\sqrt{[\nabla\Theta ]^{-T}\,\widehat{\rm I}_m\,[\nabla\Theta ]^{-1}}\,  [\nabla\Theta ]^{-T}{\rm II}_{y_0}^\flat [\nabla\Theta ]^{-1}-\sqrt{[\nabla\Theta ]\,\widehat{\rm I}_m^{-1}[\nabla\Theta ]^{T}}\,[\nabla\Theta ]^{-T} {\rm II}_m^\flat[\nabla\Theta ]^{-1}\notag
\\
=&\,\sqrt{[\nabla\Theta ]^{-T}\,\widehat{\rm I}_m\,[\nabla\Theta ]^{-1}}\,  [\nabla\Theta ]^{-T}{\rm II}_{y_0}^\flat [\nabla\Theta ]^{-1}\\&\,\quad-\sqrt{[\nabla\Theta ]^{-T}\,\widehat{\rm I}_m\,[\nabla\Theta ]^{-1}}[\nabla\Theta ]\,\widehat{\rm I}_m^{-1}\,[\nabla\Theta ]^{T}\,[\nabla\Theta ]^{-T} {\rm II}_m^\flat[\nabla\Theta ]^{-1}
\notag
\\
=&\,\sqrt{[\nabla\Theta ]^{-T}\,\widehat{\rm I}_m\,[\nabla\Theta ]^{-1}}\,  [\nabla\Theta ]\widehat{\rm I}_{y_0}^{-1}{\rm II}_{y_0}^\flat [\nabla\Theta ]^{-1}\notag-\sqrt{[\nabla\Theta ]^{-T}\,\widehat{\rm I}_m\,[\nabla\Theta ]^{-1}}[\nabla\Theta ]\,\widehat{\rm I}_m^{-1}\, {\rm II}_m^\flat[\nabla\Theta ]^{-1}
\notag
\\
=&\,\sqrt{[\nabla\Theta ]^{-T}\,\widehat{\rm I}_m\,[\nabla\Theta ]^{-1}}\,  [\nabla\Theta ]\Big({\rm L}_{y_0}^\flat - {\rm L}_m^\flat\Big)[\nabla\Theta ]^{-1}.\notag
\end{align}
\endgroup
In view of the identity
$
\id_3=\sqrt{[\nabla\Theta ]^{-T}\,\widehat{\rm I}_{y_0}[\nabla\Theta ]^{-1}},
$
the quantity $
\mathcal{E}_{ \infty } \, {\rm B}_{y_0} +  {\rm C}_{y_0} \mathcal{K}_{ \infty }$
can  equivalently be written in the alternative form
\begin{align}\label{symetrimy}
\mathcal{E}_{ \infty } \,& {\rm B}_{y_0} +  {\rm C}_{y_0} \mathcal{K}_{ \infty }=\sqrt{[\nabla\Theta ]^{-T}\,\widehat{\rm I}_m\,[\nabla\Theta ]^{-1}}\,  [\nabla\Theta ]\Big({\rm L}_{y_0}^\flat - {\rm L}_m^\flat\Big)[\nabla\Theta ]^{-1}\sqrt{[\nabla\Theta ]^{-T}\,\widehat{\rm I}_{y_0}[\nabla\Theta ]^{-1}},
\end{align}
and therefore we see how both pairs $\widehat{\rm I}_{y_0}, {\rm L}_{y_0}^\flat$ and $\widehat{\rm I}_{m}, {\rm L}_{m}^\flat$, together, influence the constitutive quantity \linebreak $\mathcal{E}_{ \infty } \, {\rm B}_{y_0} +  {\rm C}_{y_0} \mathcal{K}_{ \infty }$.
Since $(\mathcal{E}_{ \infty } \, {\rm B}_{y_0} +  {\rm C}_{y_0} \mathcal{K}_{ \infty }){\rm B}_{y_0}$ is also an argument of an energy term defining the membrane-bending energy, i.e., 	$W_{\mathrm{shell}} \big((  \mathcal{E}_{ \infty } \, {\rm B}_{y_0} +  {\rm C}_{y_0} \mathcal{K}_{ \infty } )   {\rm B}_{y_0} \,\big)$, we  have to compute as well
\begin{align}
(\mathcal{E}_{ \infty } \, {\rm B}_{y_0} +  {\rm C}_{y_0} \mathcal{K}_{ \infty }){\rm B}_{y_0}&=\sqrt{[\nabla\Theta ]^{-T}\,\widehat{\rm I}_m\,[\nabla\Theta ]^{-1}}\,  [\nabla\Theta ]{\rm L}_{y_0}^\flat [\nabla\Theta ]^{-1}[\nabla\Theta ]^{-T} {\rm II}_{y_0}^\flat[\nabla\Theta ]^{-1}\notag\\&\quad-\sqrt{[\nabla\Theta ]^{-T}\,\widehat{\rm I}_m\,[\nabla\Theta ]^{-1}}[\nabla\Theta ]\, {\rm L}_m^\flat[\nabla\Theta ]^{-1}[\nabla\Theta ]^{-T} {\rm II}_{y_0}^\flat[\nabla\Theta ]^{-1}\notag
\\
&=\sqrt{[\nabla\Theta ]^{-T}\,\widehat{\rm I}_m\,[\nabla\Theta ]^{-1}}\,  [\nabla\Theta ]{\rm L}_{y_0}^\flat \widehat{\rm I}_{y_0}^{-1} {\rm II}_{y_0}^\flat[\nabla\Theta ]^{-1}\\&\quad-\sqrt{[\nabla\Theta ]^{-T}\,\widehat{\rm I}_m\,[\nabla\Theta ]^{-1}}[\nabla\Theta ]\, {\rm L}_m^\flat\widehat{\rm I}_{y_0}^{-1} {\rm II}_{y_0}^\flat[\nabla\Theta ]^{-1}
\notag\\
&=\sqrt{[\nabla\Theta ]^{-T}\,\widehat{\rm I}_m\,[\nabla\Theta ]^{-1}}\,  [\nabla\Theta ]\Big({\rm L}_{y_0}^\flat - {\rm L}_m^\flat\Big){\rm L}_{y_0}^\flat[\nabla\Theta ]^{-1}\notag.
\end{align}

Note that   for $
{Q}_{ \infty }={\rm polar}\big((\nabla  m|n) [\nabla\Theta ]^{-1}\big)$, which is a consequence of the limit $\mu_{\rm c}\to \infty$, it follows that $\mathcal{E}_{ \infty }\in{\rm Sym}(3)$, while  the other two requirements \eqref{restsk}$_{2,3}$,  which are also implied by the limit  $\mu_{\rm c}\to \infty$ are not automatically satisfied. The remaining additional constraints in the minimization problem are
\begin{align}\label{conds1}
\mathcal{E}_{ \infty } \, {\rm B}_{y_0} +  {\rm C}_{y_0} \mathcal{K}_{ \infty }\stackrel{!}{\in}{\rm Sym}(3) & \Leftrightarrow\ \ \sqrt{[\nabla\Theta \,]^{-T}\,\widehat{\rm I}_m\,[\nabla\Theta \,]^{-1}}\,  [\nabla\Theta \,]\Big({\rm L}_{y_0}^\flat -{\rm L}_m^\flat\Big)[\nabla\Theta \,]^{-1}\notag\\&\qquad \stackrel{!}{=}[\nabla\Theta \,]^{-T}\Big(({\rm L}_{y_0}^\flat)^T-({\rm L}_{m}^\flat)^T\Big)[\nabla\Theta \,]^{T}\sqrt{[\nabla\Theta \,]^{-T}\,\widehat{\rm I}_m\,[\nabla\Theta \,]^{-1}},
\end{align}
and
\begin{align}\label{conds2}
(  \mathcal{E}_{ \infty } \, {\rm B}_{y_0} +  {\rm C}_{y_0} \mathcal{K}_{ \infty } )   {\rm B}_{y_0}\stackrel{!}{\in}{\rm Sym}(3) &\Leftrightarrow\ \ 
\sqrt{[\nabla\Theta \,]^{-T}\,\widehat{\rm I}_m\,[\nabla\Theta \,]^{-1}}\,  [\nabla\Theta \,]\Big({\rm L}_{y_0}^\flat - {\rm L}_m^\flat\Big){\rm L}_{y_0}^\flat[\nabla\Theta \,]^{-1}\vspace{1.5mm}\\
&\qquad \stackrel{!}{=} [\nabla\Theta \,]^{-T} ({\rm L}_{y_0}^\flat)^T\Big(({\rm L}_{y_0}^\flat)^T-({\rm L}_{m}^\flat)^T\Big))[\nabla\Theta \,]^{T}\sqrt{[\nabla\Theta \,]^{-T}\,\widehat{\rm I}_m\,[\nabla\Theta \,]^{-1}}\,.\notag
\end{align}

If $L_c>0$ is finite and non-vanishing,  the bending-curvature energy is still present in the minimization problem. The latter energy is expressed in terms of  $\mathcal{K}_{ \infty }$, which for $\mu_{\rm c}\to \infty$ turns into
\begin{align}
&\mathcal{K}_{ \infty }  =  \bigg(\mathrm{axl}(\, \sqrt{[\nabla\Theta ]\,\widehat{\rm I}_m^{-T}[\nabla\Theta ]^{T}}[\nabla\Theta ]^{-T}(\nabla m|n)^T\,\partial_{x_1} \Big((\nabla m|n)[\nabla\Theta ]^{-1} \sqrt{[\nabla\Theta ]\,\widehat{\rm I}_m^{-1}[\nabla\Theta ]^{T}}\Big)\big)\notag\\
&\qquad \quad  \ \,|\, \mathrm{axl}(\, \sqrt{[\nabla\Theta ]\,\widehat{\rm I}_m^{-T}[\nabla\Theta ]^{T}}[\nabla\Theta ]^{-T}(\nabla m|n)^T\,\partial_{x_2} \Big((\nabla m|n)[\nabla\Theta ]^{-1} \sqrt{[\nabla\Theta ]\,\widehat{\rm I}_m^{-1}[\nabla\Theta ]^{T}}\Big)\big)  \,|0\bigg)[\nabla\Theta ]^{-1}.
\end{align}

In view of the constitutive restrictions imposed by the limit case $\mu_{\rm c}\to \infty$,
the  variational problem for the constrained Cosserat $O(h^5)$-shell model  is  to find a deformation of the midsurface
$m:\omega\subset\mathbb{R}^2\to\mathbb{R}^3$  minimizing on $\omega$: 
\begin{align}\label{minvarmc}
I= \int_{\omega}   \,\, \Big[ & \,\Big(h+{\rm K}\,\dfrac{h^3}{12}\Big)\,
W_{{\rm shell}}^{\infty}\big(    \sqrt{[\nabla\Theta ]^{-T}\,\widehat{\rm I}_m\,\id_2^{\flat }\,[\nabla\Theta ]^{-1}}-
\sqrt{[\nabla\Theta ]^{-T}\,\widehat{\rm I}_{y_0}\,\id_2^{\flat }\,[\nabla\Theta ]^{-1}} \big)\vspace{2.5mm}\notag\\    
& +   \Big(\dfrac{h^3}{12}\,-{\rm K}\,\dfrac{h^5}{80}\Big)\,
W_{{\rm shell}}^{\infty}  \big(   \sqrt{[\nabla\Theta ]^{-T}\,\widehat{\rm I}_m\,[\nabla\Theta ]^{-1}}\,  [\nabla\Theta ]\Big({\rm L}_{y_0}^\flat - {\rm L}_m^\flat\Big)[\nabla\Theta ]^{-1}\big) \notag \\\hspace*{-2.5cm}&
-\dfrac{h^3}{3} \mathrm{ H}\,\mathcal{W}_{{\rm shell}}^{\infty}  \big(  \sqrt{[\nabla\Theta ]^{-T}\,\widehat{\rm I}_m\,\id_2^{\flat }\,[\nabla\Theta ]^{-1}}-
\sqrt{[\nabla\Theta ]^{-T}\,\widehat{\rm I}_{y_0}\,\id_2^{\flat }\,[\nabla\Theta ]^{-1}} ,\notag\\
&\qquad \qquad \qquad\quad\   \sqrt{[\nabla\Theta ]^{-T}\,\widehat{\rm I}_m\,[\nabla\Theta ]^{-1}}\,  [\nabla\Theta ]\Big({\rm L}_{y_0}^\flat - {\rm L}_m^\flat\Big)[\nabla\Theta ]^{-1} \big)\\&+
\dfrac{h^3}{6}\, \mathcal{W}_{{\rm shell}}^{\infty}  \big(  \sqrt{[\nabla\Theta ]^{-T}\,\widehat{\rm I}_m\,\id_2^{\flat }\,[\nabla\Theta ]^{-1}}-
\sqrt{[\nabla\Theta ]^{-T}\,\widehat{\rm I}_{y_0}\,\id_2^{\flat }\,[\nabla\Theta ]^{-1}} ,\notag
\\
&\qquad \qquad \qquad\ \ \sqrt{[\nabla\Theta ]^{-T}\,\widehat{\rm I}_m\,[\nabla\Theta ]^{-1}}\,  [\nabla\Theta ]\Big({\rm L}_{y_0}^\flat - {\rm L}_m^\flat\Big){\rm L}_{y_0}^\flat[\nabla\Theta ]^{-1}\big)\vspace{2.5mm}\notag\notag\\&+ \,\dfrac{h^5}{80}\,\,
W_{\mathrm{mp}}^{\infty} \big(\sqrt{[\nabla\Theta ]^{-T}\,\widehat{\rm I}_m\,[\nabla\Theta ]^{-1}}\,  [\nabla\Theta ]\Big({\rm L}_{y_0}^\flat - {\rm L}_m^\flat\Big){\rm L}_{y_0}^\flat[\nabla\Theta ]^{-1}\,\big)\notag\vspace{2.5mm}\notag
\\
&+ \Big(h-{\rm K}\,\dfrac{h^3}{12}\Big)\,
W_{\mathrm{curv}}\big(  \mathcal{K}_{ \infty } \big)    +  \Big(\dfrac{h^3}{12}\,-{\rm K}\,\dfrac{h^5}{80}\Big)\,
W_{\mathrm{curv}}\big(  \mathcal{K}_{ \infty }   {\rm B}_{y_0} \,  \big)  + \,\dfrac{h^5}{80}\,\,
W_{\mathrm{curv}}\big(  \mathcal{K}_{ \infty }   {\rm B}_{y_0}^2  \big)
\Big] \,{\rm det}\nabla \Theta        \,\mathrm d a\notag\\&\qquad - \overline{\Pi}(m,{Q}_{ \infty }),\notag
\end{align}
such that
\begin{align}\label{contsym}
\mathcal{E}_{ \infty } \, {\rm B}_{y_0}& +  {\rm C}_{y_0} \mathcal{K}_{ \infty }\stackrel{!}{\in}{\rm Sym}(3)\ \vspace{1.5mm}\notag\\ & \Leftrightarrow\ \ \sqrt{[\nabla\Theta \,]^{-T}\,\widehat{\rm I}_m\,[\nabla\Theta \,]^{-1}}\,  [\nabla\Theta \,]\Big({\rm L}_{y_0}^\flat -{\rm L}_m^\flat\Big)[\nabla\Theta \,]^{-1}\notag\\&\qquad \stackrel{!}{=}[\nabla\Theta \,]^{-T}\Big(({\rm L}_{y_0}^\flat)^T-({\rm L}_{m}^\flat)^T\Big)[\nabla\Theta \,]^{T}\sqrt{[\nabla\Theta \,]^{-T}\,\widehat{\rm I}_m\,[\nabla\Theta \,]^{-1}},\notag\vspace{1.5mm}\notag\\
(  \mathcal{E}_{ \infty } \,& {\rm B}_{y_0} +  {\rm C}_{y_0} \mathcal{K}_{ \infty } )   {\rm B}_{y_0}\stackrel{!}{\in}{\rm Sym}(3)\vspace{1.5mm}\\ &\Leftrightarrow\ \ 
\sqrt{[\nabla\Theta \,]^{-T}\,\widehat{\rm I}_m\,[\nabla\Theta \,]^{-1}}\,  [\nabla\Theta \,]\Big({\rm L}_{y_0}^\flat - {\rm L}_m^\flat\Big){\rm L}_{y_0}^\flat[\nabla\Theta \,]^{-1}\notag\vspace{1.5mm}\notag\\
&\qquad \stackrel{!}{=}  [\nabla\Theta \,]^{-T} ({\rm L}_{y_0}^\flat)^T\Big(({\rm L}_{y_0}^\flat)^T-({\rm L}_{m}^\flat)^T\Big))[\nabla\Theta \,]^{T}\sqrt{[\nabla\Theta \,]^{-T}\,\widehat{\rm I}_m\,[\nabla\Theta \,]^{-1}}\,,\notag
\end{align}
where 
\begin{align} 
\mathcal{K}_{ \infty } & = \, \Big(\mathrm{axl}({Q}_{ \infty }^T\,\partial_{x_1} {Q}_{ \infty })\,|\, \mathrm{axl}({Q}_{ \infty }^T\,\partial_{x_2} {Q}_{ \infty })\,|0\Big)[\nabla\Theta ]^{-1}, \vspace{2.5mm}\notag\\
{Q}_{ \infty }&={\rm polar}\big((\nabla  m|n) [\nabla\Theta ]^{-1}\big)=(\nabla m|n)[\nabla\Theta ]^{-1}\,\sqrt{[\nabla\Theta ]^{-T}\,\widehat {\rm I}_{m}\,[\nabla\Theta ]^{-1}},\notag\\
W_{{\rm shell}}^{\infty}(  S)  &=   \mu\,\lVert\,   S\rVert^2  +\,\dfrac{\lambda\,\mu}{\lambda+2\mu}\,\big[ \mathrm{tr}   \, (S)\big]^2,\qquad 
\mathcal{W}_{{\rm shell}}^{\infty}(  S,  T) =   \mu\,\bigl\langle  S,   T\bigr\rangle+\,\dfrac{\lambda\,\mu}{\lambda+2\mu}\,\mathrm{tr}  (S)\,\mathrm{tr}  (T), 
\\
W_{\mathrm{mp}}^{\infty}(  S)&= \mu\,\lVert  S\rVert^2+\,\dfrac{\lambda}{2}\,\big[ \mathrm{tr}\,   (S)\big]^2 \qquad \quad\forall \ S,T\in{\rm Sym}(3), \notag\vspace{2.5mm}\\
W_{\mathrm{curv}}(  X )&=\mu\,L_c^2\left( b_1\,\lVert \dev\,\textrm{sym}\, X\rVert^2+b_2\,\lVert\text{skew} \,X\rVert^2+b_3\,
[\tr(X)]^2\right) \quad\qquad \forall\  X\in\mathbb{R}^{3\times 3}.\notag
\end{align}

\subsection{3D versus 2D symmetry requirements for $\mu_{\rm c}\to\infty$}\label{Msym}

As already mentioned in Subsection \ref{constcon}, in the starting three-dimensional variational problem \eqref{minprob}, the limit $\mu_c\rightarrow\infty$ leads to  the constraint
\begin{align}
\text{skew}(\overline U _\xi-\id_3)=0\qquad \Leftrightarrow\qquad \overline{R}_\xi={\rm polar}(F_\xi),
\end{align}  and the parental three-dimensional model turns into the Toupin couple stress model \cite[Eq. 11.8]{Toupin62}.

Motivated by modelling arguments, after obtaining a suitable  form of the relevant coefficients $\varrho_m,\,\varrho_b$,  in \cite{GhibaNeffPartI} we have  based the
expansion of the three-dimensional elastic Cosserat energy on a further simplified expression for the (reconstructed) deformation gradient  $F_\xi$, namely 
\begin{align}\label{red2}
&F_{s,\xi}\,=\,\nabla_x\varphi_s(x_1,x_2,x_3)[\nabla_x \Theta(x_1,x_2,x_3)]^{-1}  \\\nonumber
&\cong \widetilde{F}_{e,s}\coloneqq\,(\nabla  m|\, \varrho_m^e\,\overline{Q}_{e,s}(x_1,x_2)\nabla_x\Theta(x_1,x_2,0)\, e_3)[\nabla_x \Theta(x_1,x_2,x_3)]^{-1}
\\\nonumber  &\qquad \qquad\qquad +x_3 (\nabla \left[\,\overline{Q}_{e,s}(x_1,x_2)\nabla_x\Theta(x_1,x_2,0)\, e_3\right]|\varrho_b^e\,
\overline{Q}_{e,s}(x_1,x_2)\nabla_x\Theta(x_1,x_2,0)\, e_3)[\nabla_x \Theta(x_1,x_2,x_3)]^{-1}.
\end{align}
Corresponding to this ansatz of the deformation gradient, the tensor $\widetilde{\mathcal{E}}_\xi=\overline U _\xi-\id_3$ is approximated by 
\begin{align}\label{defEes}
\widetilde{\mathcal{E}}_{s}&\coloneqq\,\overline{U}_{e,s}-\id
{_3}
\,=\,\overline{Q}_{e,s}^T\, \widetilde{F}_{e,s}-\id_3,
\end{align}
which admits the following expression  with the help of  the usual strain measures  in the nonlinear 6-parameter shell theory
\begin{align}
\widetilde{\mathcal{E}}_{s}   =&
\dfrac{1}{\det \nabla \Theta(x_3)}\,\Big\{  \mathcal{E}_{m,s}
+x_3\Big[\mathcal{E}_{m,s} ( {\rm B}_{y_0}-2\,{\rm H}\, {\rm A}_{y_0})\ +\,{\rm C}_{y_0} \, \mathcal{K}_{e,s} \Big]
+x_3^2\Big[\,{\rm C}_{y_0} \, \mathcal{K}_{e,s} ({\rm B}_{y_0}-2\,{\rm H}\,{\rm A}_{y_0}) \Big]\Big\}\notag\vspace{2.5mm}\\
& - \frac{\lambda}{(\lambda+2\mu)}\Big[\tr( \mathcal{E}_{m,s} ) + x_3\, {\rm tr} (\mathcal{E}_{m,s} {\rm B}_{y_0} + {\rm C}_{y_0}\mathcal{K}_{e,s} )
\Big] (0|0|n_0)\,(0|0|n_0)^T.
\end{align}
 We  substitute the expansion of the factor $ \dfrac{1}{\det \nabla \Theta(x_3)}\, $ in the form
\begin{equation}\label{e76}
\begin{array}{c}
\dfrac{1}{\det \nabla \Theta(x_3)} \, = \,  \dfrac{1}{1-2\,{\rm H}\,x_3+{\rm K}\,x_3^2}  =  1+ 2\,{\rm H}\,x_3+ (4\,{\rm H}^2\,-{\rm K})\,x_3^2+ O(x_3^3)
\end{array}
\end{equation}
and get

\begin{align}
\widetilde{\mathcal{E}}_{s}   =&
\Big[1+ 2\,{\rm H}\,x_3+ (4\,{\rm H}^2\,-{\rm K})\,x_3^2+ O(x_3^3) \Big]
\notag\vspace{2.5mm}\\
&\quad  \times
\,\Big\{  \mathcal{E}_{m,s}
+x_3\Big[\mathcal{E}_{m,s} ( {\rm B}_{y_0}-2\,{\rm H}\, {\rm A}_{y_0})\ +\,{\rm C}_{y_0} \, \mathcal{K}_{e,s} \Big]
+x_3^2\Big[\,{\rm C}_{y_0} \, \mathcal{K}_{e,s} ({\rm B}_{y_0}-2\,{\rm H}\,{\rm A}_{y_0}) \Big]\Big\}\notag\vspace{2.5mm}\\
& - \frac{\lambda}{\lambda+2\mu}\Big[\tr( \mathcal{E}_{m,s} ) + x_3\, {\rm tr}  (\mathcal{E}_{m,s} {\rm B}_{y_0} + {\rm C}_{y_0}\mathcal{K}_{e,s} )
\Big] (0|0|n_0)\,(0|0|n_0)^T.
\end{align}
If we multiply out all  terms and use the relation $ {\rm B}_{y_0}^2 = 2\,{\rm H}\,{\rm B}_{y_0} - {\rm K}\, {\rm A}_{y_0} $, then we obtain
\begin{align}\label{extE}
\widetilde{\mathcal{E}}_{s}  \; =\; &\,\quad \,\,
1\,\Big[  \mathcal{E}_{m,s} - \frac{\lambda}{\lambda+2\mu}\,\tr( \mathcal{E}_{m,s} )\; (0|0|n_0)\, (0|0|n_0)^T  \Big]
\notag\vspace{2.5mm}\\
& 
+x_3\Big[ (\mathcal{E}_{m,s} \, {\rm B}_{y_0} +  {\rm C}_{y_0} \mathcal{K}_{e,s}) -
\frac{\lambda}{(\lambda+2\mu)}\, {\rm tr}  (\mathcal{E}_{m,s} {\rm B}_{y_0} + {\rm C}_{y_0}\mathcal{K}_{e,s} )\; (0|0|n_0)\,  (0|0|n_0)^T  \Big]
\vspace{2.5mm}\\
& 
+x_3^2\Big[\,(\mathcal{E}_{m,s} \, {\rm B}_{y_0} +  {\rm C}_{y_0} \mathcal{K}_{e,s}) {\rm B}_{y_0} \Big]\;+\; O(x_3^3).\notag
\end{align}
Since $\{1,x_3,x_3^2\}$ are linear independent, from the last relation we see that the symmetry constraint $\widetilde{\mathcal{E}}_\xi\approx\mathcal{E}_{s}\in {\rm Sym}(3)$ imposed by the limit case $\mu_{\rm c}\to \infty$ implies the symmetry constraints on  the tensors $ \mathcal{E}_{m,s} $ , $ (\mathcal{E}_{m,s} \, {\rm B}_{y_0} +  {\rm C}_{y_0} \mathcal{K}_{e,s})  $ and 
$  (  \mathcal{E}_{m,s} \, {\rm B}_{y_0} +  {\rm C}_{y_0} \mathcal{K}_{e,s} )   {\rm B}_{y_0}  $, i.e., 
\begin{align}
\mathcal{E}_{m,s} \in {\rm Sym}(3), \qquad  (\mathcal{E}_{m,s} \, {\rm B}_{y_0} +  {\rm C}_{y_0} \mathcal{K}_{e,s})\in {\rm Sym}(3) \qquad \text{and}\qquad  
(  \mathcal{E}_{m,s} \, {\rm B}_{y_0} +  {\rm C}_{y_0} \mathcal{K}_{e,s} )   {\rm B}_{y_0}\in {\rm Sym}(3).
\end{align}

This fact  has in essence the  following physical significance:
\begin{center} \fbox{\begin{minipage}{16cm}\textit{For $\boldsymbol{\mu_{\rm c}\to \infty}$, the (through the thickness reconstructed) strain tensor  $ \widetilde{\mathcal{E}}_{s} $ for the 3D shell is symmetric.}\end{minipage}} \end{center}

The terms of order $  O(x_3^3) $ are not relevant here, since we have taken a quadratic ansatz for the deformation.

Since $ \widetilde{\mathcal{E}}_{s}=\overline{Q}_{e,s}^T\, \widetilde{F}_{e,s}-\id_3$, this condition is equivalent to: \textit{the Biot-type stretch tensor $ \overline{{U}}_{e,s}=\overline{Q}_{e,s}^T\, \widetilde{F}_{e,s}\; $ for the 3D shell is symmetric.}
In conclusion, from \eqref{extE} we see that the tensor $ (\mathcal{E}_{m,s} \, {\rm B}_{y_0} +  {\rm C}_{y_0} \mathcal{K}_{e,s}) $ can be viewed as a bending tensor for shells in the sense of Anicic and L\'eger  \cite{anicic1999formulation} and {\v{S}}ilhav{\`y} \cite {vsilhavycurvature}.

\subsection{Conditional existence  for the $O(h^5)$-constrained elastic Cosserat shell model}\label{4.1}

In this subsection we give a conditional  existence result regarding the $O(h^5)$-constrained elastic Cosserat shell model. First, it is important to define the appropriate admissible set of solutions, since ${Q}_{ \infty }$ and $m$ are not any more independent:  ${Q}_{ \infty }$ is determined by ${Q}_{ \infty }={\rm polar}(\nabla  m|n)\, Q_0^T$, once $m\in{\rm H}^1(\omega, \mathbb{R}^3)$ is known. However, ${Q}_{ \infty }={\rm polar}(\nabla  m|n)\, Q_0^T$ has to belong to ${\rm H}^1(\omega, {\rm SO}(3))$, which does not directly  follow  from $m\in{\rm H}^1(\omega, \mathbb{R}^3)$. Hence, since the existence result of the solution for the unconstrained elastic Cosserat shell model assures the desired regularity for ${Q}_{ \infty }$, we follow the same method as in \cite{GhibaNeffPartII}. We notice that the existence result presented in \cite{GhibaNeffPartII} is valid in the case of no boundary conditions upon ${Q}_{ \infty }$, too\footnote{Since ${\rm SO}(3)$ is compact, a property that does not transfer to the linearized problem.}. In the definition of the admissible set of solution, we need to take into account all the compatibility conditions between ${Q}_{ \infty }\Big|_{\gamma_d}$ and the values of $m$ on ${\gamma_d}$ imposed by the limit case $\mu_{\rm c}\to \infty$.
The  set $\mathcal{A}$ of admissible functions is therefore defined by
\begin{align}\label{21}
\mathcal{A}=\Bigg\{(m,&{Q}_{ \infty })\in{\rm H}^1(\omega, \mathbb{R}^3)\times{\rm H}^1(\omega, {\rm SO}(3))\ \bigg| \  m\big|_{ \gamma_d}=m^*, \qquad{Q}_{ \infty }Q_0.e_3\big|_{ \gamma_d}=\,\dd\frac{\partial_{x_1}m^*\times \partial_{x_2}m^*}{\lVert \partial_{x_1}m^*\times \partial_{x_2}m^*\rVert }\notag\\ 
&\ U\coloneqq {Q}_{ \infty }^T (\nabla  m|{Q}_{ \infty }Q_0.e_3)[\nabla\Theta ]^{-1} \in {\rm L^2}(\omega, {\rm Sym}^+(3))\\
&\mathfrak{K}_{\rm 1}\coloneqq {Q}_{ \infty }^T (\nabla  m|{Q}_{ \infty }Q_0.e_3) [\nabla\Theta ]^{-1}{\rm B}_{y_0}+{Q}_{ \infty }^T(\nabla ({Q}_{ \infty }Q_0.e_3)|0)[\nabla\Theta ]^{-1}\in {\rm L^2}(\omega, {\rm Sym}(3))
\notag\\[-1.5ex]
&\mathfrak{K}_{\rm 2}\coloneqq {Q}_{ \infty }^T (\nabla  m|{Q}_{ \infty }Q_0.e_3) [\nabla\Theta ]^{-1}{\rm B}_{y_0}^2+{Q}_{ \infty }^T(\nabla ({Q}_{ \infty }Q_0.e_3)|0)[\nabla\Theta ]^{-1}{\rm B}_{y_0}\in {\rm L^2}(\omega, {\rm Sym}(3))
\Bigg\},\notag
\end{align}
which incorporates a weak reformulation of the symmetry constraint in \eqref{contsym}, in the sense that all the derivatives are considered now in the sense of distributions, and the boundary conditions are to be understood in the sense of traces.
\textit{However, a priori it is not clear if the  set $\mathcal{A}$ is non-empty}. In view of \eqref{contsym},  we note that if  $m\in {\rm H}^2(\omega,\mathbb{R}^3)$ is such that ${\rm L}_{y_0}={\rm L}_m$ and $m\big|_{ \gamma_d}=m^*$, then by choosing  ${Q}_{ \infty }={\rm polar}[(\nabla  m|n)[\nabla\Theta ]^{-1}]\in {\rm SO}(3)$ we obtain $ (m,{Q}_{ \infty })\in \mathcal{A}$. In general, $\mathcal{A}$ may be empty. Therefore we propose:

\begin{theorem}\label{th1}{\rm [Conditional existence result for the theory including terms up to order $O(h^5)$]}\\
	Assume that  the admissible set $\mathcal{A}$ is non-empty and the external loads satisfy the conditions 
${f}\in\textrm{\rm L}^2(\omega,\mathbb{R}^3)$, \linebreak $t\in \textrm{\rm L}^2(\gamma_t,\mathbb{R}^3)$,
and the boundary data satisfy the conditions ${m}^*\in{\rm H}^1(\omega ,\mathbb{R}^3)$ and ${\rm polar}(\nabla {m}^*\,|\,n^*) \in{\rm H}^1(\omega, {\rm SO}(3))$.
	Assume that the following conditions concerning the initial configuration are satisfied\footnote{For shells with little initial regularity. Classical shell models typically need to assume that $y_0\in {\rm C}^3(\overline{\omega},\mathbb{R}^3)$.}: $\,y_0:\omega\subset\mathbb{R}^2\rightarrow\mathbb{R}^3$ is a continuous injective mapping and
	\begin{align}\label{26}
	{y}_0&\in{\rm H}^1(\omega ,\mathbb{R}^3),\quad   {Q}_{0}={\rm polar}(\nabla\Theta)\in{\rm H}^1(\omega, {\rm SO}(3)),\quad
	\nabla\Theta \in {\rm L}^\infty(\omega ,\mathbb{R}^{3\times 3}),\quad {\rm det}\nabla \Theta  \geq\, a_0 >0\,,
	\end{align}
	where $a_0$ is a positive constant.
	Then, for sufficiently small values of the thickness $h$ satisfying 	\begin{align}\label{rcondh5c}
	h\max\{\sup_{x\in\omega}|\kappa_1|, \sup_{x\in\omega}|\kappa_2|\}<\alpha \qquad \text{with}\qquad  \alpha<\sqrt{\frac{2}{3}(29-\sqrt{761})}\simeq 0.97083
	\end{align} 
	and for constitutive coefficients  such that $\mu>0$, $2\,\lambda+\mu> 0$, $b_1>0$, $b_2>0$, $b_3>0$ and $L_{\rm c}>0$, the minimization problem \eqref{minvarmc} admits at least one minimizing solution $m$ such that 
	$(m,{Q}_{ \infty })\in  \mathcal{A}$.
\end{theorem}
\begin{proof}
	The proof follows similar steps as  in \cite{GhibaNeffPartII} for the proof of the existence of solution for the unconstrained  $O(h^5)$ elastic Cosserat shell model and is similar to  the proof of Proposition \ref{propcoerh5} from Appendix \ref{Appendixrelaxh}. Here, we provide only certain milestones, to guide the reader and explain those details  that differ in comparison to \cite{GhibaNeffPartII}. First, it follows that for sufficiently small values of the thickness $h$ such that   \eqref{rcondh5c} is satisfied 
	and for constitutive coefficients  satisfying  $\mu>0$, $2\,\lambda+\mu> 0$, $b_1>0$, $b_2>0$ and $b_3>0$,   the  energy density
	\begin{align}W^{\infty}(\mathcal{E}_{ \infty }, \mathcal{K}_{ \infty })=W_{\mathrm{memb}}^{\infty}\big(  \mathcal{E}_{ \infty } \big)+W_{\mathrm{memb,bend}}^{\infty}\big(  \mathcal{E}_{ \infty } ,\,  \mathcal{K}_{ \infty } \big)+W_{\mathrm{bend,curv}}\big(  \mathcal{K}_{ \infty }    \big)
	\end{align}
	is coercive, where  $W_{\mathrm{memb}}^{\infty}$ and  $W_{\mathrm{memb,bend}}^{\infty}$ are the membrane energy and the membrane-bending energy corresponding to the limit case  $\mu_{\rm c}\to \infty$. This means  that  there exists a constant   $a_1^+>0$  such that
	\begin{equation}\label{26bis}
	W^{\infty}(\mathcal{E}_{ \infty }, \mathcal{K}_{ \infty })\,\geq\, a_1^+\, \big( \lVert \mathcal{E}_{ \infty }\rVert ^2 + \lVert \mathcal{K}_{ \infty }\rVert ^2\,\big)  \quad \forall\, (m,{Q}_{ \infty })\in \mathcal{A},
	\end{equation}
	where
	$a_1^+$ depends on the constitutive coefficients (but not on the Cosserat couple modulus $\mu_{\rm c}$). Indeed, all the steps used in the proof of  Proposition \ref{propcoerh5} from Appendix \ref{Appendixrelaxh} remain valid by considering $ \mu_{\rm c}\to \infty$ since for all $(m,{Q}_{ \infty })\in \mathcal{A}$ it follows that $\mathcal{E}_{ \infty }\in {\rm Sym}(3)$ and $ 	\mathcal{E}_{ \infty }{\rm B}_{y_0}+{\rm C}_{y_0}\, \mathcal{K}_{ \infty }\in {\rm Sym}(3)$.

	Moreover,   under  the same hypothesis and using the same arguments as above, it follows that  for $\mu_{\rm c}\to \infty$, too, the  energy density
	$W^{\infty}(\mathcal{E}_{ \infty }, \mathcal{K}_{ \infty })
	$
	is uniformly convex in $(\mathcal{E}_{ \infty }, \mathcal{K}_{ \infty })$ for all $(m,{Q}_{ \infty })\in \mathcal{A}$, i.e., there exists a constant $a_1^+>0$ such that
	\begin{align}
	D^2W^{\infty}(\mathcal{E}_{ \infty },\mathcal{K}_{ \infty }).\,[(H_1,H_2),(H_1,H_2)]\geq a_1^+(\lVert H_1\rVert^2+\lVert H_2\rVert^2) \quad \forall\, H_1,H_2\in \mathbb{R}^{3\times 3}.
	\end{align}
	
The assumptions on the  external loads and the boundedness of $\,\,\Pi_{S^0}\,$ and $ \,\Pi_{\partial S^0_f}\,$ imply that there exists a constant $c>0$ such that\footnote{ {By $c$ and $c_i$, $i\in\mathbb{N}$, we will denote (positive) constants that may vary from estimate to estimate but will remain independent of $m$, $\nabla m$ and ${Q}_{ \infty }$.}}
\begin{equation}\label{27}
|\,\overline{\Pi}(m,{Q}_{ \infty })\,|\,\leq\, \,\,C\,\big(\,\lVert m\rVert_{{\rm H}^1(\omega)}+1\big),\quad\forall\,(m,{Q}_{ \infty })\in  {\rm H}^1(\omega, \mathbb{R}^3)\times{\rm H}^1(\omega, {\rm SO}(3)).
\end{equation}
Considering
\begin{equation}
\overline{R}_{ \infty }(x_1,x_2) ={Q}_{ \infty }(x_1,x_2)\,Q_0(x_1,x_2,0)\in{\rm SO}(3),
\end{equation}
we write
\begin{align}
\mathcal{E}_{ \infty }  \,=& \,   Q_0  [\overline{R}_{ \infty }^T (\nabla  m|{Q}_{ \infty }\nabla\Theta .e_3)- Q_0^T (\nabla  y_0|n_0)][\nabla\Theta ]^{-1}=\,Q_0( \overline{R}_{ \infty }^T\,\nabla m -Q_0^T\nabla y_0 |0)[\nabla\Theta ]^{-1}.
\end{align}
Using this expression for $\mathcal{E}_{ \infty }$ we obtain
\begin{align}\label{Elambda}
\lVert \mathcal{E}_{ \infty }\rVert ^2& 
\geq\, \lambda_0^2\,\lVert ( \overline{R}_{ \infty }^T\,\nabla m -Q_0^T\nabla y_0 |0)\rVert ^2\geq\,
\lambda_0^2\lVert \nabla m\rVert ^2_{{\rm L}^2(\omega)}- c_1\,\lVert \nabla m \rVert_{{\rm L}^2(\omega)}+ c_2,
\end{align}
where $\lambda_0$ is the smallest eigenvalue of the positive definite matrix $\widehat{\rm I}^{-1}_{y_0}$ and  ${c}_1>0$, ${c}_2>0$  are some positive constants. 
Similarly, we deduce that
\begin{align}\lVert \mathcal{K}_{ \infty } \rVert ^2
&\geq \, \lambda_0^2\, \lVert (\mathrm{axl}({Q}_{ \infty }^T\,\partial_{x_1} {Q}_{ \infty })\,|\, \mathrm{axl}({Q}_{ \infty }^T\,\partial_{x_2} {Q}_{ \infty }))\rVert ^2.
\end{align}

In view  of the coercivity of the internal energy and \eqref{26}, \eqref{27} and \eqref{Elambda}, and after using the Poincar\'e--inequality we deduce that  the functional $I(m,{Q}_{ \infty })$   is bounded from below on $\mathcal{A}$, i.e., 
\begin{align}\label{36}
I(m,{Q}_{ \infty })&\geq\, c_1\lVert  m-m^*\rVert_{{\rm H}^1(\omega)}^2 +c_2
\, \quad\quad \quad \forall \, (m,{Q}_{ \infty })\in  {\rm H}^1(\omega, \mathbb{R}^3)\times{\rm H}^1(\omega, {\rm SO}(3)),
\end{align}
where $c_1>0$ and $c {_2}\in\mathbb{R}$.
Hence, there exists an infimizing sequence $\big\{(m_k,\overline{Q}_{k})\big\}_{k=1}^\infty$ in $  \mathcal{A}$, such that
\begin{equation}\label{37}
\lim_{k\rightarrow \infty} I(m_k,\overline{Q}_{k}) = \,\inf\, \big\{I(m,{Q}_{ \infty })\, \big|\,  (m,{Q}_{ \infty })\in  \mathcal{A}\big\}.
\end{equation}
Since we have $I(m^*,{Q}_{ \infty }^*)<\infty$,  in view of the conditions on the boundary data, the infimizing sequence $\big\{(m_k,\overline{Q}_{k})\big\}_{k=1}^\infty$ can be chosen such that
\begin{equation}\label{38}
I(m_k,\overline{Q}_{k})\,\leq \,I(m^*,{Q}_{ \infty }^*)\,< \infty\,, \qquad \forall\,k\geq\, 1.
\end{equation}
Using  \eqref{36} and \eqref{38} we remark that the sequence $\big\{m_k \big\}_{k=1}^\infty$ is bounded in ${\rm H}^1(\omega,\mathbb{R}^3)$. Hence, we can consider a subsequence of $\big\{m_k \big\}_{k=1}^\infty$  (not relabeled) which converges weakly in  ${\rm H}^1(\omega,\mathbb{R}^3)$. According to Rellich's selection principle, it converges strongly in ${\rm L}^2(\omega,\mathbb{R}^{3})$, i.e., ~there exists an element $\widehat{m}\in{\rm H}^1(\omega,\mathbb{R}^3)$ such that
\begin{equation}\label{39}
m_k  \rightharpoonup \widehat{m} \quad\mathrm{in}\quad {\rm H}^1(\omega, \mathbb{R}^3),\qquad \mathrm{and}\qquad m_k \rightarrow\widehat{ m} \quad\mathrm{in}\quad {\rm L}^2(\omega, \mathbb{R}^3).
\end{equation}

We skip further details, since they mimic the proof of existence theorem  from \cite{GhibaNeffPartII}, and we use only the fact that there exists a subsequence of $ \big\{ \overline{Q}_{k}\big\}_{k=1}^\infty\in {\rm H}^1(\omega,{\rm SO}(3))$ (not relabeled) and an element $\widehat{\overline{Q}}_{ \infty }\in {\rm H}^1(\omega,{\rm SO}(3))$ with
\begin{equation}\label{40}
\overline{Q}_{k} \rightharpoonup       \widehat{\overline{Q}}_{ \infty }    \quad\mathrm{in}\quad {\rm H}^1(\omega, \mathbb{R}^{3\times3}) , \qquad\mathrm{and}\qquad         \overline{Q}_{k}  \rightarrow   \widehat{\overline{Q}}_{ \infty } \quad\mathrm{in}\quad {\rm L}^2(\omega, \mathbb{R}^{3\times3}).
\end{equation}
Let us next construct the limit strain and curvature measures
\begin{align}\label{e552}
\widehat{\mathcal{E}}_{ \infty }  :&=\,    \widehat{\overline{Q}}_{ \infty }^T  (\nabla   {\widehat{m}}|\widehat{\overline{Q}}_{ \infty }\nabla\Theta .e_3)[\nabla\Theta ]^{-1}-\id_3=Q_0( \widehat{\overline{R}}_{ \infty }^T\,\nabla  {\widehat{m}} -Q_0^T\nabla y_0 |0)[\nabla\Theta ]^{-1} ,  \\
\widehat{\mathcal{K}}_{ \infty }  :&=\,  \Big(\mathrm{axl}(\widehat{\overline{Q}}_{ \infty }^T\,\partial_{x_1} \widehat{\overline{Q}}_{ \infty })\,|\, \mathrm{axl}(\widehat{\overline{Q}}_{ \infty }^T\,\partial_{x_2} \widehat{\overline{Q}}_{ \infty })\,|0\Big)[\nabla\Theta ]^{-1}.\notag
\end{align}

Note that using \
\begin{align}\label{curvfuraxl}
{Q}_{ \infty }^T\,\partial_{x_i} {Q}_{ \infty }\,=\,Q_0\,\overline{R}_{ \infty }^T\,\partial_{x_i} (\overline{R}_{ \infty }\,Q_0^T)\,=\,Q_0 (\overline{R}^T\,\partial_{x_i} \overline{R}_{ \infty })\,Q_0^T-Q_0(Q_0^T\partial_{x_i} Q_0)\,Q_0^T,\quad i\,=\,1,2,3
\end{align}
and the identity ~ $\mathrm{axl}(Q\, A\, Q^T)\,=\,Q\,\mathrm{axl}( A)$ ~ for all $ Q\in {\rm SO}(3)$ and for all $A\in \mathfrak{so}(3)$,
we obtain 
\begin{align}
\widehat{\mathcal{K}}_{ \infty } &= Q_0\Big(\,\text{axl}(\widehat{\overline{R}}_{ \infty }^T\partial_{x_1} \widehat{\overline{R}}_{ \infty })\!-\!\text{axl}(Q^{T}_0\partial_{x_1} Q_0) \,\,\big|\,\, \text{axl}(\widehat{\overline{R}}_{ \infty }^T\partial_{x_2} \widehat{\overline{R}}_{ \infty })\! -\!\text{axl}(Q^{T}_0\partial_{x_2} Q_0)\,\,\big|\, \,0\,\Big)[\nabla\Theta ]^{-1}.
\end{align}
In the next steps of the proof, all the arguments from  \cite{GhibaNeffPartII}  hold true and we  conclude that \begin{align}\label{convE}
\mathcal{E}_{ \infty }^{(k)}  :&=Q_0( \overline{R}_k^T\,\nabla m_k -Q_0^T\nabla y_0 |0)[\nabla\Theta ]^{-1}\rightharpoonup\widehat{\mathcal{E}}_{ \infty }
\end{align} in ${\rm L}^2(\omega,\mathbb{R}^{3\times3})$, where
~ $\widehat{\overline{R}}_k(x_1,x_2) \coloneqq \widehat{\overline{Q}}_k(x_1,x_2)\,Q_0(x_1,x_2,0)\in{\rm SO}(3)$, as well as
\begin{align}\label{convK}
\mathcal{K}_{ \infty }^{(k)}  :&= Q_0\Big(\,\text{axl}({\overline{R}}_k^T\partial_{x_1} {\overline{R}}_k)\!-\!\text{axl}(Q^{T}_0\partial_{x_1} Q_0) \,\,\big|\,\, \text{axl}({\overline{R}}_k^T\partial_{x_2}{\overline{R}}_k)\! -\!\text{axl}(Q^{T}_0\partial_{x_2} Q_0)\,\,\big|\, \,0\,\Big)[\nabla\Theta ]^{-1} \rightharpoonup\widehat{\mathcal{K}}_{ \infty }
\end{align}
in ${\rm L}^2(\omega,\mathbb{R}^{3\times3})$. The admissible set $\mathcal{A}$ defined by \eqref{21} is closed under weak convergence. Indeed, since the set of symmetric matrices is closed under weak convergence, we find 
\begin{equation}\label{Uk}
\begin{array}{rll}
U_k\coloneqq &\mathcal{E}_{ \infty }^{(k)}  +\id_3 \rightharpoonup U\coloneqq \widehat{\mathcal{E}}_{ \infty }+\id_3&\quad\text{in}\quad  {\rm L^2}(\omega, {\rm Sym}(3)),\vspace{2mm}\\
\mathfrak{K}_{\rm 1}^{(k)}\coloneqq &\mathcal{E}_{ \infty }^{(k)}\, {\rm B}_{y_0} +  {\rm C}_{y_0} \mathcal{K}_{ \infty }^{(k)} \rightharpoonup \mathfrak{K}_{\rm 1}\coloneqq \widehat{\mathcal{E}}_{ \infty } \, {\rm B}_{y_0} +  {\rm C}_{y_0} \widehat{\mathcal{K}}_{ \infty }& \quad\text{in}\quad {\rm L^1}(\omega, {\rm Sym}(3)),\vspace{2mm}\\ 
\mathfrak{K}_{\rm 2}^{(k)}\coloneqq &\mathcal{E}_{ \infty }^{(k)} \, {\rm B}_{y_0}^2 +  {\rm C}_{y_0} \mathcal{K}_{ \infty }^{(k)}  {\rm B}_{y_0}\rightharpoonup \mathfrak{K}_{\rm 2}\coloneqq \widehat{\mathcal{E}}_{ \infty } \, {\rm B}_{y_0}^2 +  {\rm C}_{y_0} \widehat{\mathcal{K}}_{ \infty }{\rm B}_{y_0}&\quad\text{in}\quad\,  {\rm L^1}(\omega, {\rm Sym}(3)),
\end{array}
\end{equation} whenever  the sequence $\big\{(m_k,\overline{Q}_{k})\big\}_{k=1}^\infty\subseteq \mathcal{A}$ is weakly convergent to  $(\widehat{{m}},\widehat{\overline{Q}}_{ \infty })\in{\rm H}^1(\omega, \mathbb{R}^3)\times{\rm H}^1(\omega, {\rm SO}(3))$. Moreover, by virtue of the relations  $(m_k,\overline{Q}_{k})\in \mathcal{A}$ and \eqref{39}, \eqref{40},   we derive that $\widehat{m}={m}^*$  and $\widehat{\overline{Q}}_{ \infty }Q_0.e_3=\,\dd\frac{\partial_{x_1}m^*\times \partial_{x_2}m^*}{\lVert \partial_{x_1}m^*\times \partial_{x_2}m^*\rVert }$ on $\gamma_d\,$ in the sense of traces. The second boundary condition  is satisfied, since from \eqref{Uk} it follows that $\widehat{\overline{Q}}_{ \infty }Q_0.e_3=\widehat{n}\coloneqq \dd\frac{\partial_{x_1}\widehat{m}\times \partial_{x_2}\widehat{m}}{\lVert \partial_{x_1}\widehat{m}\times \partial_{x_2}\widehat{m}\rVert }$ and $\widehat{\overline{Q}}_{ \infty }={\rm polar}[(\nabla  \widehat{m}|\widehat{n})[\nabla\Theta ]^{-1}]$.
Hence, we obtain that the limit pair satisfies $(\widehat{m},\widehat{\overline{Q}}_{ \infty })\in \mathcal{A}$ and, 
in view of convexity in the chosen strain and curvature measures \eqref{convE}, \eqref{convK} and \eqref{Uk}, we also have
\begin{equation}\label{47}
\int_\omega W^\infty(\widehat{\mathcal{E}}_{ \infty },\widehat{\mathcal{K}}_{ \infty })\,{\rm det}\nabla \Theta  \, {{\rm d}a}\,\leq \, \liminf_{n\to\infty}  \int_\omega W^\infty(\mathcal{E}_{ \infty }^{(k)},\mathcal{K}_{ \infty }^{(k)})\,{\rm det}\nabla \Theta  \, {{\rm d}a}.
\end{equation}
 Taking into account the assumptions on the  external loads, the continuity of the load potential functions,    and the convergence relations \eqref{39}$_2$ and \eqref{40}$_2\,$, we deduce
\begin{equation}\label{48}
{\overline{\Pi}}(\widehat{{m}}, \widehat{\overline{Q}}_{ \infty })=  \lim_{n\to\infty}  {\overline{\Pi}}(m_k, \overline{Q}_k).
\end{equation}
From \eqref{47} and \eqref{48} we get
\begin{equation}\label{49}
I(\widehat{{m}},\widehat{\overline{Q}}_{ \infty })\,\leq\,  \liminf_{n\to\infty} \, I(m_k,\overline{Q}_k)\,.
\end{equation}
Finally, the relations \eqref{37} and \eqref{49} show that ~
$ I(\widehat{{m}},\widehat{\overline{Q}}_{ \infty })\,=\,
\,\inf\, \big\{I({m},{Q}_{ \infty })\, \big|\,  ({m},{Q}_{ \infty })\in \mathcal{A}\big\}.
$
Since $(\widehat{{m}},\widehat{\overline{Q}}_{ \infty })\in\mathcal{A}$, we conclude that $(\widehat{{m}},\widehat{\overline{Q}}_{ \infty })$ is a minimizing solution pair of our minimization problem, in which $\widehat{\overline{Q}}_{ \infty }={\rm polar}\big((\nabla  \widehat{{m}}|\widehat{{n}}) [\nabla\Theta ]^{-1}\big)$, $\widehat{\mathcal{E}}_{ \infty } \, {\rm B}_{y_0} +  {\rm C}_{y_0} \widehat{\mathcal{K}}_{ \infty }\in {\rm Sym}(3)$ and $  (\widehat{\mathcal{E}}_{ \infty } \, {\rm B}_{y_0} +  {\rm C}_{y_0} \widehat{\mathcal{K}}_{ \infty }){\rm B}_{y_0} \in {\rm Sym}(3).$
\end{proof}
\subsection{Conditional existence   for the $O(h^3)$-constrained elastic Cosserat shell model }\label{4.2}

In the $O(h^5)$-shell model and for $\mu_{\rm c}\to \infty$ the assumptions
\begin{align}
\label{restskh3}\lVert{\rm skew}( \mathcal{E}_{m,s})\rVert=0
,
\qquad\qquad\qquad
\lVert{\rm skew}(\mathcal{E}_{m,s} \, {\rm B}_{y_0} +  {\rm C}_{y_0} \mathcal{K}_{e,s} )\rVert=0,
\end{align}
are not sufficient to ensure that the energy density  is finite, we need to require additionally 
\begin{align}\label{extracond}
\lVert{\rm skew}[(\mathcal{E}_{m,s} \, {\rm B}_{y_0} +  {\rm C}_{y_0} \mathcal{K}_{e,s} ){\rm B}_{y_0}] \rVert=0.
\end{align}
This condition  is necessary since the energy term $\frac{h^5}{80}\,\frac{1}{6}
W_{\mathrm{shell}} \big((  \mathcal{E}_{m,s} \, {\rm B}_{y_0} +  {\rm C}_{y_0} \mathcal{K}_{e,s} )   {\rm B}_{y_0} \,\big)$  must remain finite for  $\mu_{\rm c}\to \infty$. 
However, in the $O(h^3)$-shell model the latter  energy term is not anymore present and there is no need to impose  the extra symmetry condition \eqref{extracond}. Indeed, if we restrict our model to order $O(h^3)$, then the internal energy density reads \begin{align}\label{h3energy} W^{(h^3)}(\mathcal{E}_{m,s}, \mathcal{K}_{e,s})=&\,  \Big(h+{\rm K}\,\dfrac{h^3}{12}\Big)\,
W_{\mathrm{shell}}\big(    \mathcal{E}_{m,s} \big)+  \dfrac{h^3}{12}\,
W_{\mathrm{shell}}  \big(   \mathcal{E}_{m,s} \, {\rm B}_{y_0} +   {\rm C}_{y_0} \mathcal{K}_{e,s} \big) \notag \\&
-\dfrac{h^3}{3} \mathrm{ H}\,\mathcal{W}_{\mathrm{shell}}  \big(  \mathcal{E}_{m,s} ,
\mathcal{E}_{m,s}{\rm B}_{y_0}+{\rm C}_{y_0}\, \mathcal{K}_{e,s} \big)+
\dfrac{h^3}{6}\, \mathcal{W}_{\mathrm{shell}}  \big(  \mathcal{E}_{m,s} ,
( \mathcal{E}_{m,s}{\rm B}_{y_0}+{\rm C}_{y_0}\, \mathcal{K}_{e,s}){\rm B}_{y_0} \big)\notag\vspace{2.5mm}\\
&+  \Big(h-{\rm K}\,\dfrac{h^3}{12}\Big)\,
W_{\mathrm{curv}}\big(  \mathcal{K}_{e,s} \big)    +  \dfrac{h^3}{12}
W_{\mathrm{curv}}\big(  \mathcal{K}_{e,s}   {\rm B}_{y_0} \,  \big),
\end{align}
which, under the restrictions \eqref{restskh3}, can be expressed as
\begin{align}\label{h3energyconstr} W^{(h^3)}_{\infty}(\mathcal{E}_{m,s}, \mathcal{K}_{e,s})=&\,  \Big(h+{\rm K}\,\dfrac{h^3}{12}\Big)\,
W_{\mathrm{shell}}\big(   {\rm sym} \,\mathcal{E}_{m,s} \big)+  \dfrac{h^3}{12}\,
W_{\mathrm{shell}}  \big(   {\rm sym}(\mathcal{E}_{m,s} \, {\rm B}_{y_0} +   {\rm C}_{y_0} \mathcal{K}_{e,s}) \big) \notag \\&
-\dfrac{h^3}{3} \mathrm{ H}\,\mathcal{W}_{\mathrm{shell}}  \big( {\rm sym}\, \mathcal{E}_{m,s} ,
{\rm sym}(\mathcal{E}_{m,s}{\rm B}_{y_0}+{\rm C}_{y_0}\, \mathcal{K}_{e,s})\big) \\&+
\dfrac{h^3}{6}\, \mathcal{W}_{\mathrm{shell}}  \big(  {\rm sym}\,\mathcal{E}_{m,s} ,
{\rm sym}[( \mathcal{E}_{m,s}{\rm B}_{y_0}+{\rm C}_{y_0}\, \mathcal{K}_{e,s}){\rm B}_{y_0} ]\big)\notag\vspace{2.5mm}\\
&+  \Big(h-{\rm K}\,\dfrac{h^3}{12}\Big)\,
W_{\mathrm{curv}}\big(  \mathcal{K}_{e,s} \big)    +  \dfrac{h^3}{12}
W_{\mathrm{curv}}\big(  \mathcal{K}_{e,s}   {\rm B}_{y_0} \,  \big)\notag
\end{align}
and which remains finite for $\mu_{\rm c}\to \infty$.

In view of the above constitutive restrictions \eqref{restskh3} imposed by the limit case,
the  variational problem for the constrained Cosserat $O(h^3)$-shell model  is  to find a deformation of the midsurface
$m:\omega\subset\mathbb{R}^2\to\mathbb{R}^3$  minimizing on $\omega$:
\begin{align}\label{minvarmch3} 
I= \int_{\omega}   \,\, \Big[  &\Big(h+{\rm K}\,\dfrac{h^3}{12}\Big)\,
W_{{\rm shell}}^{\infty}\big(    \sqrt{\![\nabla\Theta ]^{-T}\,\widehat{\rm I}_m\,\id_2^{\flat }\,[\nabla\Theta ]^{-1}}\!-\!
\sqrt{\![\nabla\Theta ]^{-T}\,\widehat{\rm I}_{y_0}\,\id_2^{\flat }\,[\nabla\Theta ]^{-1}} \big),\vspace{2.5mm}\notag\\    
&+   \dfrac{h^3}{12}\,
W_{{\rm shell}}^{\infty}  \big(   \sqrt{[\nabla\Theta ]^{-T}\,\widehat{\rm I}_m\,[\nabla\Theta ]^{-1}}\,  [\nabla\Theta ]\Big({\rm L}_{y_0}^\flat - {\rm L}_m^\flat\Big)[\nabla\Theta ]^{-1}\big) \notag \\&
-\dfrac{h^3}{3} \mathrm{ H}\,\mathcal{W}_{{\rm shell}}^{\infty}  \big(  \sqrt{[\nabla\Theta ]^{-T}\,\widehat{\rm I}_m\,\id_2^{\flat }\,[\nabla\Theta ]^{-1}}-
\sqrt{[\nabla\Theta ]^{-T}\,\widehat{\rm I}_{y_0}\,\id_2^{\flat }\,[\nabla\Theta ]^{-1}} ,\notag\\
&\qquad \qquad \qquad\ \ \sqrt{[\nabla\Theta ]^{-T}\,\widehat{\rm I}_m\,[\nabla\Theta ]^{-1}}\,  [\nabla\Theta ]\Big({\rm L}_{y_0}^\flat - {\rm L}_m^\flat\Big)[\nabla\Theta ]^{-1} \big)\\&+
\dfrac{h^3}{6}\, \mathcal{W}_{{\rm shell}}^{\infty}  \big(  \sqrt{[\nabla\Theta ]^{-T}\,\widehat{\rm I}_m\,\id_2^{\flat }\,[\nabla\Theta ]^{-1}}-
\sqrt{[\nabla\Theta ]^{-T}\,\widehat{\rm I}_{y_0}\,\id_2^{\flat }\,[\nabla\Theta ]^{-1}}, 
\notag\\
&\qquad \qquad \qquad\ \ \sqrt{[\nabla\Theta ]^{-T}\,\widehat{\rm I}_m\,[\nabla\Theta ]^{-1}}\,  [\nabla\Theta ]\Big({\rm L}_{y_0}^\flat - {\rm L}_m^\flat\Big){\rm L}_{y_0}^\flat[\nabla\Theta ]^{-1}\big)\vspace{2.5mm}\notag\\& + \Big(h-{\rm K}\,\dfrac{h^3}{12}\Big)\,
W_{\mathrm{curv}}\big(  \mathcal{K}_{ \infty } \big)    +  \dfrac{h^3}{12}\,
W_{\mathrm{curv}}\big(  \mathcal{K}_{ \infty }   {\rm B}_{y_0} \,  \big)
\Big] \,{\rm det}\nabla \Theta        \,\mathrm d a - \overline{\Pi}(m,{Q}_{ \infty }),\notag
\end{align}
such that
\begin{align}\label{contsymh3}
&\mathcal{E}_{ \infty } \, {\rm B}_{y_0} +  {\rm C}_{y_0} \mathcal{K}_{ \infty }\stackrel{!}{\in}{\rm Sym}(3)\ \ \Leftrightarrow\vspace{1.5mm}\notag\\ &  \sqrt{[\nabla\Theta ]^{-T}\,\widehat{\rm I}_m\,[\nabla\Theta ]^{-1}}\,  [\nabla\Theta ]\Big({\rm L}_{y_0}^\flat \!\!-\!{\rm L}_m^\flat\Big)[\nabla\Theta ]^{-1}\stackrel{!}{=}[\nabla\Theta ]^{-T}\Big(({\rm L}_{y_0}^\flat)^T\!\!-({\rm L}_{m}^\flat)^T\Big)[\nabla\Theta ]^{T}\sqrt{[\nabla\Theta ]^{-T}\,\widehat{\rm I}_m\,[\nabla\Theta ]^{-1}},\notag
\end{align}
where 
\begin{align}
\mathcal{K}_{ \infty } & = \, (\mathrm{axl}({Q}_{ \infty }^T\,\partial_{x_1} {Q}_{ \infty })\,|\, \mathrm{axl}({Q}_{ \infty }^T\,\partial_{x_2} {Q}_{ \infty })\,|0)[\nabla\Theta ]^{-1}, \vspace{2.5mm}\notag\\
{Q}_{ \infty }&={\rm polar}\big((\nabla  m|n) [\nabla\Theta ]^{-1}\big)=(\nabla m|n)[\nabla\Theta ]^{-1}\,\sqrt{[\nabla\Theta ]^{-T}\,\widehat {\rm I}_{m}\,[\nabla\Theta ]^{-1}},\\ 
W_{{\rm shell}}^{\infty}(  S)  &=   \mu\,\lVert\,   S\rVert^2  +\,\dfrac{\lambda\,\mu}{\lambda+2\mu}\,\big[ \mathrm{tr}   \, (S)\big]^2,\qquad 
\mathcal{W}_{{\rm shell}}^{\infty}(  S,  T) =   \mu\,\bigl\langle  S, T\bigr\rangle+\,\dfrac{\lambda\,\mu}{\lambda+2\mu}\,\mathrm{tr}  (S)\,\mathrm{tr}  (T), \notag\vspace{2.5mm}\\
W_{\mathrm{mp}}^{\infty}(  S)&= \mu\,\lVert  S\rVert^2+\,\dfrac{\lambda}{2}\,\big[ \mathrm{tr}\,   (S)\big]^2 \qquad \ \ \forall \ S,T\in{\rm Sym}(3), \notag\vspace{2.5mm}\\
W_{\mathrm{curv}}(  X )&=\mu\,L_c^2\left( b_1\,\lVert \dev\,\textrm{sym}\, X\rVert^2+b_2\,\lVert\text{skew} \,X\rVert^2+b_3\,
[\tr(X)]^2\right) \qquad \ \ \forall\  X\in\mathbb{R}^{3\times 3}.\notag
\end{align}

We consider the  admissible set $\mathcal{A}_{(h^3)}$ of solutions to be defined by
\begin{align}\label{21h3}
\mathcal{A}_{(h^3)}=\Bigg\{(m,&{Q}_{ \infty })\in{\rm H}^1(\omega, \mathbb{R}^3)\times{\rm H}^1(\omega, {\rm SO}(3))\ \bigg| \  m\big|_{ \gamma_d}=m^*, \qquad {Q}_{ \infty }Q_0.e_3\big|_{ \gamma_d}=\,\dd\frac{\partial_{x_1}m^*\times \partial_{x_2}m^*}{\lVert \partial_{x_1}m^*\times \partial_{x_2}m^*\rVert }\notag\\ 
&\,\,U\coloneqq {Q}_{ \infty }^T (\nabla  m|n)[\nabla\Theta ]^{-1} \in {\rm L^2}(\omega, {\rm Sym}(3))\\[-1ex]
&\mathfrak{K}_{\rm 1}\coloneqq {Q}_{ \infty }^T (\nabla  m|n) [\nabla\Theta ]^{-1}{\rm B}_{y_0}+{Q}_{ \infty }^T(\nabla n|0)[\nabla\Theta ]^{-1}\in {\rm L^2}(\omega, {\rm Sym}(3))
\Bigg\},\notag
\end{align}
 where the boundary conditions are to be understood in the sense of traces. As in the case of the constrained Cosserat shell model up to $O(h^5)$, the set $\mathcal{A}_{(h^3)}$ may be empty. 
In \cite{GhibaNeffPartII}  and Appendix \ref{coerh3Appendix} we have shown that if the constitutive coefficients are  such that $\mu>0, \,\mu_{\rm c}>0$, $2\,\lambda+\mu> 0$, $b_1>0$, $b_2>0$, $b_3>0$ and $L_{\rm c}>0$ and if the thickness $h$ satisfies at least one of the following conditions:
\begin{enumerate}
\item[i)] $	h\max\{\sup_{x\in\omega}|\kappa_1|, \sup_{x\in\omega}|\kappa_2|\}<\alpha$ \quad {\bf and} \quad  $	h^2<\frac{(5-2\sqrt{6})(\alpha^2-12)^2}{4\, \alpha^2}\frac{ {c_2^+}}{\max\{C_1^+,\mu_c\}}$  \quad  \text{with}  \quad  $\quad 0<\alpha<2\sqrt{3}$;
\item[ii)] $h\max\{\sup_{x\in\omega}|\kappa_1|, \sup_{x\in\omega}|\kappa_2|\}<\frac{1}{a}$  \quad {\bf and} \quad  $a>\max\Big\{1 + \frac{\sqrt{2}}{2},\frac{1+\sqrt{1+3\frac{\max\{C_1^+,\mu_c\}}{\min\{c_1^+,\mu_c\}}}}{2}\Big\}$,
\end{enumerate}
where  $c_1^+$ and $C_1^+$ denote the  smallest and the largest eigenvalues, respectively, of the quadratic form ${W}_{\mathrm{shell}}^{\infty}( X)$, then the energy density
\begin{align}\label{h3energy} W_{(h^3)}(\mathcal{E}_{ \infty }, \mathcal{K}_{ \infty })=&\,  \Big(h+{\rm K}\,\dfrac{h^3}{12}\Big)\,
W_{\mathrm{shell}}\big(    \mathcal{E}_{ \infty } \big)+  \dfrac{h^3}{12}\,
W_{\mathrm{shell}}  \big(   \mathcal{E}_{ \infty } \, {\rm B}_{y_0} +   {\rm C}_{y_0} \mathcal{K}_{ \infty } \big) \notag \\&
-\dfrac{h^3}{3} \mathrm{ H}\,\mathcal{W}_{\mathrm{shell}}  \big(  \mathcal{E}_{ \infty } ,
\mathcal{E}_{ \infty }{\rm B}_{y_0}+{\rm C}_{y_0}\, \mathcal{K}_{ \infty } \big)+
\dfrac{h^3}{6}\, \mathcal{W}_{\mathrm{shell}}  \big(  \mathcal{E}_{ \infty } ,
( \mathcal{E}_{ \infty }{\rm B}_{y_0}+{\rm C}_{y_0}\, \mathcal{K}_{ \infty }){\rm B}_{y_0} \big)\vspace{2.5mm}\notag\\
&+  \Big(h-{\rm K}\,\dfrac{h^3}{12}\Big)\,
W_{\mathrm{curv}}\big(  \mathcal{K}_{ \infty } \big)    +  \dfrac{h^3}{12}
W_{\mathrm{curv}}\big(  \mathcal{K}_{ \infty }   {\rm B}_{y_0} \,  \big)
\end{align}
is coercive. However, this result is not suitable for the limit case $\mu_{\rm c}\to \infty$ since the above conditions imposed upon the   thickness would imply $h\to 0$ for $\mu_{\rm c}\to \infty$. We circumvent the problem by proving  the following result.

\begin{proposition}\label{coerh3}{\rm [Coercivity in the theory including terms up to order $O(h^3)$]} Assume that the constitutive coefficients are  such that $\mu>0$, $2\,\lambda+\mu> 0$, $b_1>0$, $b_2>0$, $b_3>0$ and $L_{\rm c}>0$ and let $c_2^+$  denote the smallest eigenvalue  of
	$
	W_{\mathrm{curv}}(  S ),
	$
	and $c_1^+$ and $ C_1^+>0$ denote the smallest and the largest eigenvalues of the quadratic form $W_{\mathrm{shell}}^\infty(  S)$.
	If the thickness $h$ satisfies  one of the following conditions:
	\begin{enumerate}
	\item[i)] $	h\max\{\sup_{x\in\omega}|\kappa_1|, \sup_{x\in\omega}|\kappa_2|\}<\alpha$ \quad {\bf and} \quad  $	h^2<\frac{(5-2\sqrt{6})(\alpha^2-12)^2}{4\, \alpha^2}\frac{ {c_2^+}}{C_1^+}$  \quad  \text{with}  \quad  $\quad 0<\alpha<2\sqrt{3}$;
	\item[ii)] $h\max\{\sup_{x\in\omega}|\kappa_1|, \sup_{x\in\omega}|\kappa_2|\}<\frac{1}{a}$  \quad {\bf and} \quad  $a>\max\Big\{1 + \frac{\sqrt{2}}{2},\frac{1+\sqrt{1+3\frac{C_1^+}{c_1^+}}}{2}\Big\}$,
	\end{enumerate}
	then the internal energy density
\begin{align}\label{h3energy} W_{(h^3)}^{\infty}(\mathcal{E}_{ \infty }, \mathcal{K}_{ \infty })=&\,  \Big(h+{\rm K}\,\dfrac{h^3}{12}\Big)\,
W_{\mathrm{shell}}^{\infty}\big(    \mathcal{E}_{ \infty } \big)+  \dfrac{h^3}{12}\,
W_{\mathrm{shell}}^{\infty}  \big(   \mathcal{E}_{ \infty } \, {\rm B}_{y_0} +   {\rm C}_{y_0} \mathcal{K}_{ \infty } \big) \notag \\&
-\dfrac{h^3}{3} \mathrm{ H}\,\mathcal{W}_{\mathrm{shell}}^{\infty}  \big(  \mathcal{E}_{ \infty } ,
\mathcal{E}_{ \infty }{\rm B}_{y_0}+{\rm C}_{y_0}\, \mathcal{K}_{ \infty } \big)+
\dfrac{h^3}{6}\, \mathcal{W}_{\mathrm{shell}}^{\infty}  \big(  \mathcal{E}_{ \infty } ,
( \mathcal{E}_{ \infty }{\rm B}_{y_0}+{\rm C}_{y_0}\, \mathcal{K}_{ \infty }){\rm B}_{y_0} \big)\vspace{2.5mm}\notag\\
&+  \Big(h-{\rm K}\,\dfrac{h^3}{12}\Big)\,
W_{\mathrm{curv}}\big(  \mathcal{K}_{ \infty } \big)    +  \dfrac{h^3}{12}
W_{\mathrm{curv}}\big(  \mathcal{K}_{ \infty }   {\rm B}_{y_0} \,  \big) 
\end{align}
	is coercive on $\mathcal{A}_{(h^3)}$, in the sense that  there exists   a constant $a_1^+>0$ such that
	\begin{equation}\label{26bis}
	W_{(h^3)}^{\infty}(\mathcal{E}_{ \infty }, \mathcal{K}_{ \infty })\,\geq\, a_1^+\, \big(  \lVert \mathcal{E}_{ \infty }\rVert ^2 + \lVert \mathcal{K}_{ \infty }\rVert ^2\,\big)\qquad   \forall  (m,{Q}_{ \infty })\in \mathcal{A}_{(h^3)},
	\end{equation}
	where
	$a_1^+$ depends on the constitutive coefficients but is independent of the Cosserat couple modulus $\mu_{\rm c}$.
\end{proposition}

\begin{proof}The proof is similar with the proof of Proposition 4.1 from
  \cite{GhibaNeffPartII} and Proposition \ref{coerh3r} from the appendix, the only difference consists in using  $\lVert X\rVert\geq \lVert \sym \,X\rVert$, $\forall X\in \mathbb{R}^{3\times 3}$ and the positive definiteness conditions \eqref{pozitivdef}.
\end{proof}

Once the coercivity is proven, the following existence result follows easily:
\begin{theorem}\label{th11}{\rm [Conditional existence result for the constrained theory including terms up to order $O(h^3)$]}\\
	Assume that   the admissible set $\mathcal{A}_{(h^3)}$ is non-empty and 	 the external loads satisfy the conditions
${f}\in\textrm{\rm L}^2(\omega,\mathbb{R}^3)$, $t\in \textrm{\rm L}^2(\gamma_t,\mathbb{R}^3)$,
the boundary data satisfy the conditions ${m}^*\in{\rm H}^1(\omega ,\mathbb{R}^3)$  and ${\rm polar}(\nabla {m}^*\,|\,n^*)\in{\rm H}^1(\omega, {\rm SO}(3))$,
	and that the following conditions concerning the initial configuration are fulfilled: $\,y_0:\omega\subset\mathbb{R}^2\rightarrow\mathbb{R}^3$ is a continuous injective mapping and
	\begin{align}
	{y}_0&\in{\rm H}^1(\omega ,\mathbb{R}^3),\quad   {Q}_{0}={\rm polar}(\nabla \Theta)\in{\rm H}^1(\omega, {\rm SO}(3)),\quad
	\nabla\Theta \in {\rm L}^\infty(\omega ,\mathbb{R}^{3\times 3}),\quad {\rm det}\nabla \Theta  \geq\, a_0 >0\,,
	\end{align}
	where $a_0$ is a positive constant.
	Assume that the constitutive coefficients are  such that $\mu>0$, $2\,\lambda+\mu> 0$, $b_1>0$, $b_2>0$, $b_3>0$ and $L_{\rm c}>0$.
	Then, if the thickness $h$ satisfies at least one of the following conditions:
	\begin{enumerate}
		\item[i)] $	h\max\{\sup_{x\in\omega}|\kappa_1|, \sup_{x\in\omega}|\kappa_2|\}<\alpha$ \quad {\bf and} \quad  $	h^2<\frac{(5-2\sqrt{6})(\alpha^2-12)^2}{4\, \alpha^2}\frac{ {c_2^+}}{C_1^+}$  \quad  \text{with}  \quad  $\quad 0<\alpha<2\sqrt{3}$;
		\item[ii)] $h\max\{\sup_{x\in\omega}|\kappa_1|, \sup_{x\in\omega}|\kappa_2|\}<\frac{1}{a}$  \quad {\bf and} \quad  $a>\max\Big\{1 + \frac{\sqrt{2}}{2},\frac{1+\sqrt{1+3\frac{C_1^+}{c_1^+}}}{2}\Big\}$,
	\end{enumerate}
	where  $c_2^+$  denotes the smallest eigenvalue  of
	$
	W_{\mathrm{curv}}(  S ),
	$
	and $c_1^+$ and $ C_1^+>0$ denote the smallest and the biggest eigenvalues of the quadratic form $W_{\mathrm{shell}}^{\infty}(  S)$,
	the minimization problem corresponding to the energy density defined by \eqref{minvarmch3} admits at least one minimizing solution $m$ such that the pair
	$(m,{Q}_{ \infty })\in  \mathcal{A}_{(h^3)}$.
\end{theorem}

\subsection{The constrained elastic Cosserat plate model}\label{3.4}
In the case of Cosserat plates (planar shell) there is no initial curvature and  we have $\Theta(x_1,x_2,x_3)\,=\,(x_1,x_2,x_3)$ together with
\begin{align}
\nabla\Theta(x_3)&\,=\,\id_3, \qquad y_0(x_1,x_2)\,=\,(x_1,x_2,0), \qquad Q_0\,=\,\id_3,\qquad n_0\,=\,e_3, \qquad d_i^0\,=\,e_i,\\
\qquad \qquad {\rm B}_{\rm id}&\,=\,0_3, \qquad {\rm C}_{\rm id}\,=\,\begin{footnotesize}\begin{footnotesize}\begin{pmatrix}
0&1&0 \\
-1&0&0 \\
0&0&0
\end{pmatrix}\end{footnotesize}\end{footnotesize}\in \mathfrak{so}(3),\quad {\rm I}_{\rm id} \,=\,\id_2,\quad \ {\rm II}_{\rm id} \,=\,0_2\quad
{\rm L}_{\rm id} \,=\,0_2,\quad \ {\rm K} \,=\,0\, ,\quad 
{\rm H}\, \,=\,0.\notag
\end{align}

Therefore, all  $O(h^5)$-terms of our constrained elastic Cosserat shell model given in Subsection \ref{4.1} as well as all the mixed terms are automatically vanishing and we obtain that 	the variational problem of the constrained Cosserat plate model is  to find a deformation of the midsurface
$m:\omega\subset\mathbb{R}^2\to\mathbb{R}^3$  minimizing on $\omega$:
\begin{align}\label{minplate}
I= \int_{\omega}   \,\, \Big[ & h\,
{W}_{{\rm shell}}^{\infty}\big(    \sqrt{\widehat{\rm I}_m^{\flat }}-
\id_2^{\flat } \big)+ \dfrac{h^3}{12}\,
{W}_{{\rm shell}}^{\infty}  \big(   \sqrt{\widehat{\rm I}_m^{-1}} {\rm II}_m^\flat\big)  + h\,
W_{\mathrm{curv}}\big(  \mathcal{K}_{ \infty } \big) 
\Big] \,      \,\mathrm d a - \overline{\Pi}(m,{Q}_{ \infty }),
\end{align}
such that
\begin{align}
{\rm C}_{\rm id} \mathcal{K}_{ \infty }&\stackrel{!}{\in}{\rm Sym}(3)\ \ \Leftrightarrow\ \ \,\sqrt{\widehat{\rm I}_m}\,{\rm L}_m^\flat\stackrel{!}{\in}{\rm Sym}(3) \ \ \Leftrightarrow\ \ \,\sqrt{\widehat{\rm I}_m}\,\widehat{\rm I}_m^{-1} {\rm II}_m^\flat\stackrel{!}{\in}{\rm Sym}(3)\ \ \Leftrightarrow\ \ \,\sqrt{\widehat{\rm I}_m^{-1}} {\rm II}_m^\flat\stackrel{!}{\in}{\rm Sym}(3),
\shortintertext{where}
{Q}_{ \infty }&={\rm polar}(\nabla  m|n)=(\nabla m|n) \sqrt{\widehat{\rm I}_m^{-1}}, \\ 
\mathcal{K}_{ \infty } &=  \Big(\mathrm{axl}(\, \sqrt{\widehat{\rm I}_m^{-T}}(\nabla m|n)^T\,\partial_{x_1} \big((\nabla m|n) \sqrt{\widehat{\rm I}_m^{-1}}\big)\big) \,|\, \mathrm{axl}(\, \sqrt{\widehat{\rm I}_m^{-T}}(\nabla m|n)^T\,\partial_{x_2} \big((\nabla m|n) \sqrt{\widehat{\rm I}_m^{-1}}\big)\big) \,\big|0\Big)\,.\notag
\end{align}
We already observe that the bending tensor $\dd\sqrt{\widehat{\rm I}_m^{-1}} {\rm II}_m^\flat$ is invariant under $m\to \alpha \, m$, $\alpha>0$ (cf.~the discussion in Section \ref{subsec:Acharya} ). Note that the scaling $m\to \alpha \, m$, $\alpha>0$ is certainly not connected to a bending type deformation. Therefore, it does make sense that $\sqrt{\widehat{\rm I}_m^{-1}} {\rm II}_m^\flat$  remains invariant. 

The existence of minimizers  for this problem was already discussed in \cite{Neff_plate04_cmt}.  The considered 
admissible set $\mathcal{A}_{\rm plate}$ of solutions is \cite{Neff_plate04_cmt}
\begin{align}\label{21h3plate}
\mathcal{A}_{\rm plate}=\Bigg\{(m,&{Q}_{ \infty })\in{\rm H}^1(\omega, \mathbb{R}^3)\times{\rm H}^1(\omega, {\rm SO}(3))\ \bigg|\  m\big|_{ \gamma_d}=m^*, \qquad {Q}_{ \infty }Q_0.e_3\big|_{ \gamma_d}=\,\dd\frac{\partial_{x_1}m^*\times \partial_{x_2}m^*}{\lVert \partial_{x_1}m^*\times \partial_{x_2}m^*\rVert }\notag\\ 
&\,\,U\coloneqq {Q}_{ \infty }^T (\nabla  m|n) \in {\rm L^2}(\omega, {\rm Sym}(3)),\quad 
\mathfrak{K}_{\rm 1}\coloneqq {Q}_{ \infty }^T(\nabla n|0)\in {\rm L^2}(\omega, {\rm Sym}(3))
\Bigg\},\notag
\end{align}
where the boundary conditions are to be understood in the sense of traces.  The restriction $\mathfrak{K}_{\rm 1}\coloneqq {Q}_{ \infty }^T(\nabla n|0)\stackrel{!}{\in} {\rm L^2}(\omega, {\rm Sym}(3))\ \Leftrightarrow \ \sqrt{\widehat{\rm I}_m^{-1}} {\rm II}_m^\flat\stackrel{!}{\in}{\rm Sym}(3)$ is not automatically satisfied, in general. However, it is satisfied for pure bending, i.e., when no change of metric is present $(\nabla m|n)\in {\rm SO}(3)\ \Leftrightarrow \  {\rm I}_{m}={\rm I}_{y_0}$, situation when  $\mathcal{G}_\infty^\flat=0$,  $\mathcal{R}_\infty^\flat={\rm II}_m^\flat\in{\rm Sym}(3)$. 
Therefore, if there exists a solution of the pure bending problem, i.e.,  the situation $(\nabla m|n)\in {\rm SO}(3)$, when $U \in {\rm diag}$
and ${\rm II}_m\in {\rm diag}$, such that $ {\rm I}_{m}={\rm I}_{y_0}$, then the set $\mathcal{A}_{\rm plate}$ is not empty. 

In the constrained planar case, the  change of metric tensor is given by
$
\mathcal{G}_{\infty}^\flat =\sqrt{{\rm I}_m^{\flat }}-\id_2^\flat\in {\rm Sym}(3),
$
the  bending strain tensor becomes
$
\mathcal{R}_{\infty}^\flat=\sqrt{\widehat{\rm I}_m^{-1}} {\rm II}_m^\flat\stackrel{!}{\in} {\rm Sym}(3)$ (which is not automatically satisfied if ${Q}_{ \infty }={\rm polar}(\nabla  m|n)$),
the transverse shear deformation vector 
vanishes
$
\mathcal{T}_{\infty}=(0,0)
$
and the vector of drilling bending reads
$
\mathcal{N}_{\infty} \coloneqq  e_3^T\, \big(\mbox{axl}({Q}_{ \infty }^T\partial_{x_1}{Q}_{ \infty })\,|\, \mbox{axl}({Q}_{ \infty }^T\partial_{x_2}{Q}_{ \infty }) \big).
$

\section{Modified constrained Cosserat shell models}\label{sec:modified}

\subsection{A modified $O(h^5)$-constrained Cosserat shell model. Unconditional existence} \label{mcm}\setcounter{equation}{0}

As we have seen in  Section \ref{Msym}, the symmetry  restrictions 
\begin{align}
\mathcal{E}_{m,s} \in {\rm Sym}(3), \qquad  (\mathcal{E}_{m,s} \, {\rm B}_{y_0} +  {\rm C}_{y_0} \mathcal{K}_{e,s})\in {\rm Sym}(3) \qquad \text{and}\qquad  
(  \mathcal{E}_{m,s} \, {\rm B}_{y_0} +  {\rm C}_{y_0} \mathcal{K}_{e,s} )   {\rm B}_{y_0}\in {\rm Sym}(3)
\end{align}
assure that the  (through the thickness reconstructed) strain tensor  $ \widetilde{\mathcal{E}}_{s} $ for the 3D-shell is symmetric. However,   the (reconstructed) strain tensor  $ \widetilde{\mathcal{E}}_{s} $ for the 3D shell represents in itself only an approximation of the  true tensor $\widetilde{\mathcal{E}}_\xi=\overline U _\xi-\id_3$ in the limit case $\mu_{\rm c}\to \infty$, too. 

The considered ansatz for the (reconstructed) deformation gradient  $F_\xi$ given by $\widetilde{F}_{e,s}$ in \eqref{red2} does not assure, in the absence of the other two  symmetry conditions, that 
$
\mathcal{E}_{m,s} \in {\rm Sym}(3)
$ 
implies that the (reconstructed) strain tensor  $ \widetilde{\mathcal{E}}_{s} $ for the 3D-shell is symmetric. This is not surprising because in the general Cosserat model the strain  tensor  $\widetilde{\mathcal{E}}_\xi=\overline U _\xi-\id_3$ is not symmetric anyway, so the choice of an ansatz for the (reconstructed) deformation gradient  $\widetilde{F}_{e,s}$ which would lead to a symmetric strain tensor was not our purpose.

In the existence proof we have seen that the admissible set may be  empty for general boundary conditions. Therefore, the existence Theorem \ref{th1} is only conditional. With this motivation, we want to modify the resulting constrained Cosserat shell model such that the new constrained model allows for an unconditional existence proof. In order to do so, we relax the symmetry conditions and further on impose only the first symmetry condition
\begin{align}
\mathcal{E}_{m,s} \in {\rm Sym}(3).
\end{align}
In addition, looking back at \eqref{conds1}, \eqref{conds2} and \eqref{minvarmc}, we assume that the energy density depends only on the (symmetric parts)
\begin{align}
\mathcal{E}_{m,s}, \qquad \sym (\mathcal{E}_{m,s} \, {\rm B}_{y_0} +  {\rm C}_{y_0} \mathcal{K}_{e,s}) \qquad \text{and}\qquad  
\sym\big[(  \mathcal{E}_{m,s} \, {\rm B}_{y_0} +  {\rm C}_{y_0} \mathcal{K}_{e,s} )   {\rm B}_{y_0}\big],
\end{align}
instead of
\begin{align}
\mathcal{E}_{m,s}, \qquad  (\mathcal{E}_{m,s} \, {\rm B}_{y_0} +  {\rm C}_{y_0} \mathcal{K}_{e,s}) \qquad \text{and}\qquad  
(  \mathcal{E}_{m,s} \, {\rm B}_{y_0} +  {\rm C}_{y_0} \mathcal{K}_{e,s} )   {\rm B}_{y_0}.
\end{align}
For the  (reconstructed) strain tensor  $ \widetilde{\mathcal{E}}_{s} $ for the 3D shell we  would then propose the modified ansatz\footnote{It does add nothing to the shell model itself.}
\begin{align}\label{mextE}
\widetilde{\mathcal{E}}_{s}  \; =\; &\quad \,1\,\,\,
\Big[  \mathcal{E}_{m,s} - \frac{\lambda}{\lambda+2\mu}\,\tr( \mathcal{E}_{m,s} )\; (0|0|n_0)\, (0|0|n_0)^T  \Big]
\notag\vspace{2.5mm}\\
& 
+x_3\Big[ \sym(\mathcal{E}_{m,s} \, {\rm B}_{y_0} +  {\rm C}_{y_0} \mathcal{K}_{e,s}) -
\frac{\lambda}{(\lambda+2\mu)}\, {\rm tr}  (\mathcal{E}_{m,s} {\rm B}_{y_0} + {\rm C}_{y_0}\mathcal{K}_{e,s} )\; (0|0|n_0)\,  (0|0|n_0)^T  \Big]
\vspace{2.5mm}\\
& 
+x_3^2\,\,\sym\Big[\,(\mathcal{E}_{m,s} \, {\rm B}_{y_0} +  {\rm C}_{y_0} \mathcal{K}_{e,s}) {\rm B}_{y_0} \Big]\;+\; O(x_3^3),\notag
\end{align}
which is symmetric once $\mathcal{E}_{m,s}$ is symmetric.

Therefore, taking into account the modified constitutive restrictions imposed by the limit  $\mu_{\rm c}\to \infty$,
the  variational problem for the constrained Cosserat $O(h^5)$-shell model  is now  to find a deformation of the midsurface
$m:\omega\subset\mathbb{R}^2\to\mathbb{R}^3$  minimizing on $\omega$:
\begingroup
\allowdisplaybreaks
\begin{align}\label{modminvarmc}
\hspace*{-2.5cm}I= \int_{\omega}   \,\, \Big[ & \,\Big(h+{\rm K}\,\dfrac{h^3}{12}\Big)\,
W_{{\rm shell}}^{\infty}\big(    \sqrt{[\nabla\Theta ]^{-T}\,\widehat{\rm I}_m\,\id_2^{\flat }\,[\nabla\Theta ]^{-1}}-
\sqrt{[\nabla\Theta ]^{-T}\,\widehat{\rm I}_{y_0}\,\id_2^{\flat }\,[\nabla\Theta ]^{-1}} \big),\vspace{2.5mm}\notag\\ \hspace*{-2.5cm}   
& +   \Big(\dfrac{h^3}{12}\,-{\rm K}\,\dfrac{h^5}{80}\Big)\,
W_{{\rm shell}}^{\infty}  \Big(\sym \big[   \sqrt{[\nabla\Theta ]^{-T}\,\widehat{\rm I}_m\,[\nabla\Theta ]^{-1}}\,  [\nabla\Theta ]\Big({\rm L}_{y_0}^\flat - {\rm L}_m^\flat\Big)[\nabla\Theta ]^{-1}\big]\Big) \notag \\\hspace*{-2.5cm}&
-\dfrac{h^3}{3} \mathrm{ H}\,\mathcal{W}_{{\rm shell}}^{\infty}  \Big(  \sqrt{[\nabla\Theta ]^{-T}\,\widehat{\rm I}_m\,\id_2^{\flat }\,[\nabla\Theta ]^{-1}}-
\sqrt{[\nabla\Theta ]^{-T}\,\widehat{\rm I}_{y_0}\,\id_2^{\flat }\,[\nabla\Theta ]^{-1}} ,\notag\\\hspace*{-2.5cm}
&\qquad \qquad \qquad\quad\   \sym \big[   \sqrt{[\nabla\Theta ]^{-T}\,\widehat{\rm I}_m\,[\nabla\Theta ]^{-1}}\,  [\nabla\Theta ]\Big({\rm L}_{y_0}^\flat - {\rm L}_m^\flat\Big)[\nabla\Theta ]^{-1}\big]\Big)\\\hspace*{-2.5cm}&+
\dfrac{h^3}{6}\, \mathcal{W}_{{\rm shell}}^{\infty}  \Big(  \sqrt{[\nabla\Theta ]^{-T}\,\widehat{\rm I}_m\,\id_2^{\flat }\,[\nabla\Theta ]^{-1}}-
\sqrt{[\nabla\Theta ]^{-T}\,\widehat{\rm I}_{y_0}\,\id_2^{\flat }\,[\nabla\Theta ]^{-1}} ,
\notag\\\hspace*{-2.5cm}
&\qquad \qquad \qquad\ \ \sym\big[\sqrt{[\nabla\Theta ]^{-T}\,\widehat{\rm I}_m\,[\nabla\Theta ]^{-1}}\,  [\nabla\Theta ]\Big({\rm L}_{y_0}^\flat - {\rm L}_m^\flat\Big){\rm L}_{y_0}^\flat[\nabla\Theta ]^{-1}\big]\Big)\vspace{2.5mm}\notag\\
&+ \,\dfrac{h^5}{80}\,\,
W_{\mathrm{mp}}^{\infty} \Big(\sym \big[\sqrt{[\nabla\Theta ]^{-T}\,\widehat{\rm I}_m\,[\nabla\Theta ]^{-1}}\,  [\nabla\Theta ]\Big({\rm L}_{y_0}^\flat - {\rm L}_m^\flat\Big){\rm L}_{y_0}^\flat[\nabla\Theta ]^{-1}\,\big]\Big)\notag\vspace{2.5mm}\notag\\
&+ \Big(h-{\rm K}\,\dfrac{h^3}{12}\Big)\,
W_{\mathrm{curv}}\big(  \mathcal{K}_{ \infty } \big)    +  \Big(\dfrac{h^3}{12}\,-{\rm K}\,\dfrac{h^5}{80}\Big)\,
W_{\mathrm{curv}}\big(  \mathcal{K}_{ \infty }   {\rm B}_{y_0} \,  \big)  + \,\dfrac{h^5}{80}\,\,
W_{\mathrm{curv}}\big(  \mathcal{K}_{ \infty }   {\rm B}_{y_0}^2  \big)
\Big] \,{\rm det}\nabla \Theta        \,\mathrm d a\notag\\&\quad - \overline{\Pi}(m,{Q}_{ \infty }),\notag
\end{align}
\endgroup
where 
\begingroup
\allowdisplaybreaks
\begin{align} 
\mathcal{K}_{ \infty } & = \, \Big(\mathrm{axl}({Q}_{ \infty }^T\,\partial_{x_1} {Q}_{ \infty })\,|\, \mathrm{axl}({Q}_{ \infty }^T\,\partial_{x_2} {Q}_{ \infty })\,|0\Big)[\nabla\Theta ]^{-1}, \vspace{2.5mm}\notag\\
{Q}_{ \infty }&={\rm polar}\big((\nabla  m|n) [\nabla\Theta ]^{-1}\big)=(\nabla m|n)[\nabla\Theta ]^{-1}\,\sqrt{[\nabla\Theta ]^{-T}\,\widehat {\rm I}_{m}\,[\nabla\Theta ]^{-1}},\notag\\
W_{{\rm shell}}^{\infty}(  S)  &=   \mu\,\lVert\,   S\rVert^2  +\,\dfrac{\lambda\,\mu}{\lambda+2\mu}\,\big[ \mathrm{tr}   \, (S)\big]^2,\qquad 
\mathcal{W}_{{\rm shell}}^{\infty}(  S,  T) =   \mu\,\bigl\langle  S,   T\bigr\rangle+\,\dfrac{\lambda\,\mu}{\lambda+2\mu}\,\mathrm{tr}  (S)\,\mathrm{tr}  (T), \vspace{2.5mm}\\
W_{\mathrm{mp}}^{\infty}(  S)&= \mu\,\lVert  S\rVert^2+\,\dfrac{\lambda}{2}\,\big[ \mathrm{tr}\,   (S)\big]^2 \qquad \quad\forall \ S,T\in{\rm Sym}(3), \notag\vspace{2.5mm}\\
W_{\mathrm{curv}}(  X )&=\mu\,L_c^2\left( b_1\,\lVert \dev\,\textrm{sym}\, X\rVert^2+b_2\,\lVert\text{skew} \,X\rVert^2+b_3\,
[\tr(X)]^2\right) \quad\qquad \forall\  X\in\mathbb{R}^{3\times 3}.\notag
\end{align}
\endgroup

The  set $\mathcal{A}^{\rm mod}$ of admissible functions is accordingly  defined by
\begin{align}
\mathcal{A}^{\rm mod}=\Bigg\{(m,&{Q}_{ \infty })\in{\rm H}^1(\omega, \mathbb{R}^3)\times{\rm H}^1(\omega, {\rm SO}(3))\ \bigg| \  m\big|_{ \gamma_d}=m^*, \qquad{Q}_{ \infty }Q_0.e_3\big|_{ \gamma_d}=\,\dd\frac{\partial_{x_1}m^*\times \partial_{x_2}m^*}{\lVert \partial_{x_1}m^*\times \partial_{x_2}m^*\rVert }\notag\\ 
&\ U\coloneqq {Q}_{ \infty }^T (\nabla  m|{Q}_{ \infty }Q_0.e_3)[\nabla\Theta ]^{-1} \in {\rm L^2}(\omega, {\rm Sym}^+(3))
\Bigg\},\notag
\end{align}
which incorporates a weak reformulation of the imposed symmetry constraint $\mathcal{E}_{m,s} \in {\rm Sym}(3)$. 

\begin{theorem}\label{th1mod}{\rm [Unconditional existence result for the modified theory including terms up to order $O(h^5)$]}\\
	Assume that  the external loads satisfy the conditions
${f}\in\textrm{\rm L}^2(\omega,\mathbb{R}^3)$, $t\in \textrm{\rm L}^2(\gamma_t,\mathbb{R}^3)$,
and the boundary data satisfy the conditions $ {m}^*\in{\rm H}^1(\omega ,\mathbb{R}^3)$ and ${\rm polar}(\nabla {m}^*\,|\,n^*) \in{\rm H}^1(\omega, {\rm SO}(3))$.
Assume that the following conditions concerning the initial configuration are satisfied\footnote{For shells with little initial regularity. Classical shell models typically need to assume that $y_0\in {\rm C}^3(\overline{\omega},\mathbb{R}^3)$.}: $\,y_0:\omega\subset\mathbb{R}^2\rightarrow\mathbb{R}^3$ is a continuous injective mapping and
	\begin{align}\label{26}
	{y}_0&\in{\rm H}^1(\omega ,\mathbb{R}^3),\quad   {Q}_{0}={\rm polar}(\nabla\Theta)\in{\rm H}^1(\omega, {\rm SO}(3)),\quad
	\nabla\Theta \in {\rm L}^\infty(\omega ,\mathbb{R}^{3\times 3}),\quad {\rm det}\nabla \Theta  \geq\, a_0 >0\,,
	\end{align}
	where $a_0$ is a positive constant.
	Then, for sufficiently small values of the thickness $h$ such that  \begin{align}\label{rcondh5c2}
	h\max\{\sup_{x\in\omega}|\kappa_1|, \sup_{x\in\omega}|\kappa_2|\}<\alpha \qquad \text{with}\qquad  \alpha<\sqrt{\frac{2}{3}(29-\sqrt{761})}\simeq 0.97083
	\end{align} 
	and for constitutive coefficients  such that $\mu>0$, $2\,\lambda+\mu> 0$, $b_1>0$, $b_2>0$, $b_3>0$ and $L_{\rm c}>0$, the minimization problem \eqref{modminvarmc} admits at least one minimizing solution $m$ such that 
	$(m,{Q}_{ \infty })\in  \mathcal{A}^{\rm mod}$.
\end{theorem}
\begin{proof}
	The proof is similar to the proof of Theorem \ref{th1}, excluding the discussion of the other two symmetry constraints appearing in the definition of the admissible set in Theorem \ref{th1}. The essential difference is that the admissible set $\mathcal{A}^{\rm mod}$ is not empty since, e.g., for $m\in {\rm H}^2(\omega,\mathbb{R}^3)$ and ${Q}_{ \infty }={\rm polar} [(\nabla  m|n)[\nabla\Theta ]^{-1}]$ we have $(m,{Q}_{ \infty })\in \mathcal{A}^{\rm mod}$.
\end{proof}

\subsection{A modified $O(h^3)$-constrained Cosserat shell model. Unconditional existence}

Another particularity of the new modified constrained Cosserat shell model is that the admissible set is the same  for both the modified constrained theory including terms up to order $O(h^5)$ and  the modified constrained theory including terms up to order $O(h^3)$.

Indeed, by considering the new
 variational problem for the constrained Cosserat $O(h^3)$-shell model, i.e.,    to find a deformation of the midsurface
$m:\omega\subset\mathbb{R}^2\to\mathbb{R}^3$  minimizing on $\omega$: 
\begingroup
\allowdisplaybreaks
\begin{align}\label{modminvarmch3}
\hspace*{-2.5cm}I= \int_{\omega}   \,\, \Big[ & \,\Big(h+{\rm K}\,\dfrac{h^3}{12}\Big)\,
W_{{\rm shell}}^{\infty}\big(    \sqrt{[\nabla\Theta ]^{-T}\,\widehat{\rm I}_m\,\id_2^{\flat }\,[\nabla\Theta ]^{-1}}-
\sqrt{[\nabla\Theta ]^{-T}\,\widehat{\rm I}_{y_0}\,\id_2^{\flat }\,[\nabla\Theta ]^{-1}} \big),\vspace{2.5mm}\notag\\ \hspace*{-2.5cm}   
& +   \Big(\dfrac{h^3}{12}\,-{\rm K}\,\dfrac{h^5}{80}\Big)\,
W_{{\rm shell}}^{\infty}  \Big(\sym \big[   \sqrt{[\nabla\Theta ]^{-T}\,\widehat{\rm I}_m\,[\nabla\Theta ]^{-1}}\,  [\nabla\Theta ]\Big({\rm L}_{y_0}^\flat - {\rm L}_m^\flat\Big)[\nabla\Theta ]^{-1}\big]\Big) \notag \\\hspace*{-2.5cm}&
-\dfrac{h^3}{3} \mathrm{ H}\,\mathcal{W}_{{\rm shell}}^{\infty}  \Big(  \sqrt{[\nabla\Theta ]^{-T}\,\widehat{\rm I}_m\,\id_2^{\flat }\,[\nabla\Theta ]^{-1}}-
\sqrt{[\nabla\Theta ]^{-T}\,\widehat{\rm I}_{y_0}\,\id_2^{\flat }\,[\nabla\Theta ]^{-1}} ,\notag\\\hspace*{-2.5cm}
&\qquad \qquad \qquad\quad\   \sym \big[   \sqrt{[\nabla\Theta ]^{-T}\,\widehat{\rm I}_m\,[\nabla\Theta ]^{-1}}\,  [\nabla\Theta ]\Big({\rm L}_{y_0}^\flat - {\rm L}_m^\flat\Big)[\nabla\Theta ]^{-1}\big]\Big)\\\hspace*{-2.5cm}&+
\dfrac{h^3}{6}\, \mathcal{W}_{{\rm shell}}^{\infty}  \Big(  \sqrt{[\nabla\Theta ]^{-T}\,\widehat{\rm I}_m\,\id_2^{\flat }\,[\nabla\Theta ]^{-1}}-
\sqrt{[\nabla\Theta ]^{-T}\,\widehat{\rm I}_{y_0}\,\id_2^{\flat }\,[\nabla\Theta ]^{-1}} ,
\notag\\\hspace*{-2.5cm}
&\qquad \qquad \qquad\ \ \sym\big[\sqrt{[\nabla\Theta ]^{-T}\,\widehat{\rm I}_m\,[\nabla\Theta ]^{-1}}\,  [\nabla\Theta ]\Big({\rm L}_{y_0}^\flat - {\rm L}_m^\flat\Big){\rm L}_{y_0}^\flat[\nabla\Theta ]^{-1}\big]\Big)\vspace{2.5mm}\notag\\
&+ \,\dfrac{h^5}{80}\,\,
W_{\mathrm{mp}}^{\infty} \Big(\sym \big[\sqrt{[\nabla\Theta ]^{-T}\,\widehat{\rm I}_m\,[\nabla\Theta ]^{-1}}\,  [\nabla\Theta ]\Big({\rm L}_{y_0}^\flat - {\rm L}_m^\flat\Big){\rm L}_{y_0}^\flat[\nabla\Theta ]^{-1}\,\big]\Big)\notag\vspace{2.5mm}\notag\\
&+ \Big(h-{\rm K}\,\dfrac{h^3}{12}\Big)\,
W_{\mathrm{curv}}\big(  \mathcal{K}_{ \infty } \big)    +  \Big(\dfrac{h^3}{12}\,-{\rm K}\,\dfrac{h^5}{80}\Big)\,
W_{\mathrm{curv}}\big(  \mathcal{K}_{ \infty }   {\rm B}_{y_0} \,  \big)  + \,\dfrac{h^5}{80}\,\,
W_{\mathrm{curv}}\big(  \mathcal{K}_{ \infty }   {\rm B}_{y_0}^2  \big)
\Big] \,{\rm det}\nabla \Theta        \,\mathrm d a\notag\\&\quad - \overline{\Pi}(m,{Q}_{ \infty }),\notag
\end{align}
\endgroup
where 
\begin{align} 
\mathcal{K}_{ \infty } & = \, \Big(\mathrm{axl}({Q}_{ \infty }^T\,\partial_{x_1} {Q}_{ \infty })\,|\, \mathrm{axl}({Q}_{ \infty }^T\,\partial_{x_2} {Q}_{ \infty })\,|0\Big)[\nabla\Theta ]^{-1}, \vspace{2.5mm}\notag\\
{Q}_{ \infty }&={\rm polar}\big((\nabla  m|n) [\nabla\Theta ]^{-1}\big)=(\nabla m|n)[\nabla\Theta ]^{-1}\,\sqrt{[\nabla\Theta ]^{-T}\,\widehat {\rm I}_{m}\,[\nabla\Theta ]^{-1}},\notag\\
W_{{\rm shell}}^{\infty}(  S)  &=   \mu\,\lVert\,   S\rVert^2  +\,\dfrac{\lambda\,\mu}{\lambda+2\mu}\,\big[ \mathrm{tr}   \, (S)\big]^2,\qquad 
\mathcal{W}_{{\rm shell}}^{\infty}(  S,  T) =   \mu\,\bigl\langle  S,   T\bigr\rangle+\,\dfrac{\lambda\,\mu}{\lambda+2\mu}\,\mathrm{tr}  (S)\,\mathrm{tr}  (T), \notag\vspace{2.5mm}\\
W_{\mathrm{mp}}^{\infty}(  S)&= \mu\,\lVert  S\rVert^2+\,\dfrac{\lambda}{2}\,\big[ \mathrm{tr}\,   (S)\big]^2 \qquad \quad\forall \ S,T\in{\rm Sym}(3), \vspace{2.5mm}\\
W_{\mathrm{curv}}(  X )&=\mu\,L_c^2\left( b_1\,\lVert \dev\,\textrm{sym}\, X\rVert^2+b_2\,\lVert\text{skew} \,X\rVert^2+b_3\,
[\tr(X)]^2\right) \quad\qquad \forall\  X\in\mathbb{R}^{3\times 3},\notag
\end{align}
the following existence results holds:

\begin{theorem}\label{th2mod}{\rm [Unconditional existence result for the modified theory including terms up to order $O(h^3)$]}\\
	Assume that 	 the external loads satisfy the conditions
${f}\in\textrm{\rm L}^2(\omega,\mathbb{R}^3)$, $t\in \textrm{\rm L}^2(\gamma_t,\mathbb{R}^3)$,
the boundary data satisfy the conditions 
${m}^*\in{\rm H}^1(\omega ,\mathbb{R}^3)$  and ${\rm polar}(\nabla {m}^*\,|\,n^*)\in{\rm H}^1(\omega, {\rm SO}(3))$,
and that the following conditions concerning the initial configuration are fulfilled: $\,y_0:\omega\subset\mathbb{R}^2\rightarrow\mathbb{R}^3$ is a continuous injective mapping and
\begin{align}
{y}_0&\in{\rm H}^1(\omega ,\mathbb{R}^3),\quad   {Q}_{0}={\rm polar}(\nabla \Theta)\in{\rm H}^1(\omega, {\rm SO}(3)),\quad
\nabla\Theta \in {\rm L}^\infty(\omega ,\mathbb{R}^{3\times 3}),\quad {\rm det}\nabla \Theta  \geq\, a_0 >0\,,
\end{align}
where $a_0$ is a positive constant.
Assume that the constitutive coefficients are  such that $\mu>0$, $2\,\lambda+\mu> 0$, $b_1>0$, $b_2>0$, $b_3>0$ and $L_{\rm c}>0$.
Then, if the thickness $h$ satisfies at least one of the following conditions:
\begin{enumerate}
	\item[i)] $	h\max\{\sup_{x\in\omega}|\kappa_1|, \sup_{x\in\omega}|\kappa_2|\}<\alpha$ \quad {\bf and} \quad  $	h^2<\frac{(5-2\sqrt{6})(\alpha^2-12)^2}{4\, \alpha^2}\frac{ {c_2^+}}{C_1^+}$  \quad  \text{with}  \quad  $\quad 0<\alpha<2\sqrt{3}$;
	\item[ii)] $h\max\{\sup_{x\in\omega}|\kappa_1|, \sup_{x\in\omega}|\kappa_2|\}<\frac{1}{a}$  \quad {\bf and} \quad  $a>\max\Big\{1 + \frac{\sqrt{2}}{2},\frac{1+\sqrt{1+3\frac{C_1^+}{c_1^+}}}{2}\Big\}$,
\end{enumerate}
where  $c_2^+$  denotes the smallest eigenvalue  of
$
W_{\mathrm{curv}}(  S ),
$
and $c_1^+$ and $ C_1^+>0$ denote the smallest and the biggest eigenvalues of the quadratic form $W_{\mathrm{shell}}^{\infty}(  S)$,
the minimization problem corresponding to the energy density defined by \eqref{modminvarmch3}admits at least one minimizing solution $m$ such that the pair
$(m,{Q}_{ \infty })\in  \mathcal{A}^{\rm mod}$.
\end{theorem}

\subsection{A modified constrained Cosserat plate model. Unconditional existence}

In the case of the Cosserat plate theory, the modified constrained model is   to find a deformation of the midsurface
$m:\omega\subset\mathbb{R}^2\to\mathbb{R}^3$  minimizing on $\omega$: 
\begin{align}\label{modminvarmch3plate}
I= \int_{\omega}   \,\, \Big[ & h\,
{W}_{{\rm shell}}^{\infty}\big(    \sqrt{\widehat{\rm I}_m^{\flat }}-
\id_2^{\flat } \big)+ \dfrac{h^3}{12}\,
{W}_{{\rm shell}}^{\infty}  \big( \boldsymbol{\sym}(\sqrt{\widehat{\rm I}_m^{-1}} {\rm II}_m^\flat)\big)  + h\,
W_{\mathrm{curv}}\big(  \mathcal{K}_{ \infty } \big) 
\Big] \,      \,\mathrm d a - \overline{\Pi}(m,{Q}_{ \infty }),
\end{align}
where 
\begin{align}
{Q}_{ \infty }&={\rm polar}(\nabla  m|n)=(\nabla m|n) \sqrt{\widehat{\rm I}_m^{-1}}, \\ 
\mathcal{K}_{ \infty } &=  \Big(\mathrm{axl}(\, \sqrt{\widehat{\rm I}_m^{-T}}(\nabla m|n)^T\,\partial_{x_1} \big((\nabla m|n) \sqrt{\widehat{\rm I}_m^{-1}}\big)\big) \,|\, \mathrm{axl}(\, \sqrt{\widehat{\rm I}_m^{-T}}(\nabla m|n)^T\,\partial_{x_2} \big((\nabla m|n) \sqrt{\widehat{\rm I}_m^{-1}}\big)\big) \,\big|0\Big)\,.\notag
\end{align}
The admissible set is now  
\begin{align}
\mathcal{A}_{\rm plate}^{\rm mod}=\Bigg\{(m,&{Q}_{ \infty })\in{\rm H}^1(\omega, \mathbb{R}^3)\times{\rm H}^1(\omega, {\rm SO}(3))\ \bigg| \  m\big|_{ \gamma_d}=m^*, \qquad {Q}_{ \infty }Q_0.e_3\big|_{ \gamma_d}=\,\dd\frac{\partial_{x_1}m^*\times \partial_{x_2}m^*}{\lVert \partial_{x_1}m^*\times \partial_{x_2}m^*\rVert }\notag\\ 
&\,\,U\coloneqq {Q}_{ \infty }^T (\nabla  m|n) \in {\rm L^2}(\omega, {\rm Sym}(3))
\Bigg\},\notag
\end{align}
 which is non-empty.

\begin{theorem}\label{th2modplate}{\rm [Unconditional existence result for the modified constrained Cosserat plate theory]}\\
	Assume that 	 the external loads satisfy the conditions
${f}\in\textrm{\rm L}^2(\omega,\mathbb{R}^3)$, $t\in \textrm{\rm L}^2(\gamma_t,\mathbb{R}^3)$,
	the boundary data satisfy the conditions 
${m}^*\in{\rm H}^1(\omega ,\mathbb{R}^3)$ and ${\rm polar}(\nabla {m}^*\,|\,n^*)\in{\rm H}^1(\omega, {\rm SO}(3))$.
Assume that the constitutive coefficients are  such that $\mu>0$, $2\,\lambda+\mu> 0$, $b_1>0$, $b_2>0$, $b_3>0$ and $L_{\rm c}>0$.
	Then, 
	the minimization problem corresponding to the energy density defined by \eqref{modminvarmch3plate} admits at least one minimizing solution $m$ such that the pair
	$(m,{Q}_{ \infty })\in  \mathcal{A}^{\rm mod}_{\rm plate}$.
\end{theorem}

\section{Strain measures in the Cosserat shell model}\setcounter{equation}{0}\label{SmC}

Using the  Remark \eqref{eq2}, we can express the strain tensors of the unconstrained Cosserat shell model using the (referential) fundamental forms $ {\rm I}_{y_0} $, $ {\rm II}_{y_0}$, $ {\rm III}_{y_0} $ and $ {\rm L}_{y_0} $  (instead of using the matrices $ {\rm A}_{y_0}$, $ {\rm B}_{y_0}$ and  $ {\rm C}_{y_0}$), i.e., 
\begingroup
\allowdisplaybreaks
\begin{align}\label{eq5}
\mathcal{E}_{m,s}=&\quad\ \, [\nabla\Theta ]^{-T}
\begin{footnotesize}\left( \begin{array}{c|c}
(\overline{Q}_{e,s} \nabla y_0)^{T} \nabla m- {\rm I}_{y_0} & 0 \vspace{4pt}\\
(\overline{Q}_{e,s}  n_0)^{T} \nabla m & 0
\end{array} \right)\end{footnotesize} [\nabla\Theta ]^{-1} =
[\nabla\Theta ]^{-T}
\begin{footnotesize}\left( \begin{array}{c|c}
\mathcal{G} & 0 \vspace{4pt}\\
\mathcal{T}  & 0
\end{array} \right)\end{footnotesize} [\nabla\Theta ]^{-1},
\vspace{6pt}\notag\\
\mathrm{C}_{y_0} \mathcal{K}_{e,s} = &\quad\ \, [\nabla\Theta ]^{-T}
\begin{footnotesize}\left( \begin{array}{c|c}
(\overline{Q}_{e,s} \nabla y_0)^{T} \nabla (\overline{Q}_{e,s} n_0)+ {\rm II}_{y_0} & 0 \vspace{4pt}\notag\\
0 & 0
\end{array} \right)\end{footnotesize} [\nabla\Theta ]^{-1}= -[\nabla\Theta ]^{-T}
\begin{footnotesize}\left( \begin{array}{c|c}
\mathcal{R} & 0 \vspace{4pt}\notag\\
0 & 0
\end{array} \right)\end{footnotesize} [\nabla\Theta ]^{-1},\notag\\
\mathcal{E}_{m,s} {\rm B}_{y_0} = &\quad\ \,  [\nabla\Theta ]^{-T}
\begin{footnotesize}\left( \begin{array}{c|c}
\mathcal{G} \,{\rm L}_{y_0} & 0 \vspace{4pt}\notag\\
\mathcal{T} \,{\rm L}_{y_0} & 0
\end{array} \right)\end{footnotesize} [\nabla\Theta ]^{-1} ,
\vspace{10pt}\\
\mathcal{E}_{m,s} {\rm B}^2_{y_0} = &\quad\ \, [\nabla\Theta ]^{-T}
\begin{footnotesize}\left( \begin{array}{c|c}
\mathcal{G}\, {\rm L}^2_{y_0} & 0 \vspace{4pt}\\
\mathcal{T} \,{\rm L}^2_{y_0} & 0
\end{array} \right)\end{footnotesize} [\nabla\Theta ]^{-1} ,
\vspace{10pt}\\
\mathrm{C}_{y_0} \mathcal{K}_{e,s} {\rm B}_{y_0} = & \,- [\nabla\Theta ]^{-T}
\begin{footnotesize}\left( 
\mathcal{R}\, {\rm L}_{y_0}\right)^\flat\end{footnotesize} [\nabla\Theta ]^{-1},
\vspace{10pt}\notag\\\notag
\mathrm{C}_{y_0} \mathcal{K}_{e,s} {\rm B}^2_{y_0} = &\, - [\nabla\Theta ]^{-T}
\begin{footnotesize}\left( 
\mathcal{R}\, {\rm L}^2_{y_0}\right)^\flat \end{footnotesize}[\nabla\Theta ]^{-1},
\vspace{10pt}\notag\\
\mathcal{E}_{m,s} {\rm B}_{y_0}  + \mathrm{C}_{y_0} \mathcal{K}_{e,s} 
= &\, -[\nabla\Theta ]^{-T}
\begin{footnotesize}\left( \begin{array}{c|c}
\mathcal{R}-\mathcal{G} \,{\rm L}_{y_0} & 0 \vspace{4pt}\\
\mathcal{T} \,{\rm L}_{y_0} & 0
\end{array} \right)\end{footnotesize} [\nabla\Theta ]^{-1}\notag,\\
\mathcal{E}_{m,s} {\rm B}_{y_0}^2  + \mathrm{C}_{y_0} \mathcal{K}_{e,s} {\rm B}_{y_0}
= &\, -[\nabla\Theta ]^{-T}
\begin{footnotesize}\left( \begin{array}{c|c}
(\mathcal{R} -\mathcal{G} \,{\rm L}_{y_0})\,{\rm L}_{y_0}& 0 \vspace{4pt}\\
\mathcal{T} \,{\rm L}_{y_0}^2 & 0
\end{array} \right)\end{footnotesize} [\nabla\Theta ]^{-1}\notag
,
\end{align}
\endgroup
where
\begin{align}\label{eq4}
\mathcal{G} \coloneqq &\, (\overline{Q}_{e,s} \nabla y_0)^{T} \nabla m- {\rm I}_{y_0}\not\in {\rm Sym}(2)\qquad\qquad\qquad\qquad\quad\ \ \ \,\  \,\textrm{\it the non-symmetric change of metric tensor},\notag
\\
\mathcal{R} \coloneqq & \, -(\overline{Q}_{e,s} \nabla y_0)^{T} \nabla (\overline{Q}_{e,s} n_0)- {\rm II}_{y_0}\not\in {\rm Sym}(2)
\quad  \qquad\qquad \ \ \, \,\,\textrm{\it the non-symmetric bending  strain tensor},
\\
\mathcal{T}\coloneqq & \, (\overline{Q}_{e,s}  n_0)^{T} \nabla m= \, \left(\bigl\langle\overline{Q}_{e,s}  n_0, \partial_{x_1} m\bigr\rangle,\bigl\langle\overline{Q}_{e,s}  n_0, \partial_{x_2} m\bigr\rangle\right)\qquad  \textrm{\it the transverse shear deformation (row) vector}.\notag
\end{align}
The definition of $\mathcal{G}$ is related to the classical {\it change of metric }  tensor in the Koiter model
\begin{align} \mathcal{G}_{\mathrm{Koiter}} \coloneqq  \dfrac12 \big[ ( \nabla m)^{T} \nabla m- ( \nabla y_0)^{T} \nabla y_0\big]=\dfrac12\,({\rm I}_m-{\rm I}_{y_0})\in {\rm Sym}(2),\end{align}
while the bending strain tensor may be compared  with the classical {\it bending strain} tensor in the Koiter shell model  and the Naghdi-shell model  \cite[p.~11]{mardare2008derivation} with one independent director field $d:\omega\subset\mathbb{R}^2\to\mathbb{R}^3$
\begin{align} 
\mathcal{R}_{\rm{Koiter}} &\coloneqq   -[(\nabla m)^{T} \nabla n-(\nabla y_0)^{T} \nabla n_0]={\rm II}_{m}- {\rm II}_{y_0}\in {\rm Sym}(2),\\
\mathcal{R}_{\rm{Naghdi}} &\coloneqq   -[\sym((\nabla m)^{T} \nabla d)-(\nabla y_0)^{T} \nabla y_0]= -[\sym((\nabla m)^{T} \nabla d)-{\rm II}_{y_0}]\in {\rm Sym}(2).\notag
\end{align}
Since, in our constrained Cosserat shell model
\begin{align}\label{miu41A}
\mathcal{E}_{ \infty }
&=\sqrt{[\nabla\Theta ]^{-T}\,{\rm I}_m^{\flat }\,[\nabla\Theta ]^{-1}}-
\sqrt{[\nabla\Theta ]^{-T}\,{\rm I}_{y_0}^{\flat }\,[\nabla\Theta ]^{-1}}\,\in \, {\rm Sym}(3),
\end{align} using its alternative expression, see \eqref{eq5}, \begin{align}
\mathcal{E}_{ \infty }=[\nabla\Theta ]^{-T}
\begin{footnotesize}\left( \begin{array}{c|c}
\mathcal{G}_{ \infty } \ & \ 0 \vspace{4pt}\\
\mathcal{T}_{ \infty } \  & \ 0
\end{array} \right)\end{footnotesize} [\nabla\Theta ]^{-1}
\end{align}
we deduce that in the constrained Cosserat shell model it follows that
the change of metric tensor (in-plane deformation) must also be symmetric
\begin{align}
\mathcal{G}_\infty \coloneqq &\, ({Q}_{ \infty }\nabla y_0)^{T} \nabla m- {\rm I}_{y_0}{\in} \ {\rm Sym}(2),
\end{align}
the transverse shear deformation vector 
\begin{align} 
\mathcal{T}_\infty\coloneqq & \, ({Q}_{ \infty }   n_0)^{T} \nabla m=\left(\bigl\langle{Q}_{ \infty }   n_0, \partial_{x_1} m\bigr\rangle,\bigl\langle{Q}_{ \infty }   n_0, \partial_{x_2} m\bigr\rangle\right)=(0,0),
\end{align}
is zero, while
 the bending  strain tensor reads
\begin{align}
\mathcal{R}_\infty \coloneqq & \, -({Q}_{ \infty }\nabla y_0)^{T} \nabla n- {\rm II}_{y_0}\not\in {\rm Sym}(2)
\end{align}
and which remains non-symmetric, even if in the minimization problem the extra constraints (coming from $\mu_{\rm c}\to \infty$ and bounded energy)
\begin{equation}
\begin{array}{rcr}
\mathcal{E}_{\infty} {\rm B}_{y_0}  + \mathrm{C}_{y_0} \mathcal{K}_{\infty}\stackrel{!}{\in}{\rm Sym}(2)& \qquad \Leftrightarrow\qquad &
\mathcal{R}_{\infty}-\mathcal{G}_{\infty}\,{\rm L}_{y_0}\stackrel{!}{\in}{\rm Sym}(2),\vspace{2mm}\\
(\mathcal{E}_{\infty} {\rm B}_{y_0}  + \mathrm{C}_{y_0} \mathcal{K}_{\infty}) {\rm B}_{y_0}\stackrel{!}{\in}{\rm Sym}(2)
&\qquad \Leftrightarrow\qquad &
(\mathcal{R}_{\infty} -\mathcal{G}_{\infty} \,{\rm L}_{y_0})\,{\rm L}_{y_0}\stackrel{!}{\in}{\rm Sym}(2)
\end{array}
\end{equation}
are imposed\footnote{However, in the modified constrained Cosserat shell model presented in Section \ref{mcm} the energy density is expressed in terms of 
\begin{align}
\sym(\mathcal{R}_{\infty} -\mathcal{G}_{\infty} \,{\rm L}_{y_0})\in {\rm Sym}(2)\qquad \text{and}\qquad \sym[(\mathcal{R}_{\infty} -\mathcal{G}_{\infty} \,{\rm L}_{y_0})\,{\rm L}_{y_0}]\in {\rm Sym}(2).
\end{align}}.
Hence, in the constrained Cosserat shell model, we have the  following consequences of the imposed assumptions when $\mu_{\rm c}\to \infty$:
\begin{align}\label{equ13}
\mathcal{G}_\infty {\in} \ {\rm Sym}(2) \qquad &\Leftrightarrow\qquad  (\nabla y_0)^T{Q}_{ \infty }^T(\nabla m)=(\nabla m)^T{Q}_{ \infty }\, (\nabla y_0)\notag\\\qquad &\Leftrightarrow\qquad  \bigl\langle {Q}_{ \infty }\partial_{x_1}y_0, \partial_{x_1} m\bigr\rangle =\bigl\langle {Q}_{ \infty }\partial_{x_2}y_0, \partial_{x_1} m\bigr\rangle\qquad \textrm{and}\qquad \\
\mathcal{T}_\infty =(0,0)  \qquad\ \ \   &\Leftrightarrow\qquad ({Q}_{ \infty }  n_0)^{T} \nabla m =(0,0).\notag
\end{align}
It is clear that $\bigl\langle {Q}_{ \infty }  n_0, \partial_{x_\alpha} m \bigr\rangle=({Q}_{ \infty }  n_0)^{T}  \partial_{x_\alpha} m =0$  and $\bigl\langle n, \partial_{x_\alpha} m\bigr\rangle=0$ imply that ${Q}_{ \infty }  n_0$ is collinear with $n$ and since ${Q}_{ \infty } \in {\rm SO}(3)$ that
$
{Q}_{ \infty }  n_0=n.
$
The above restrictions \eqref{equ13} are natural in the constrained nonlinear Cosserat shell model, they are  also the underlying  hypotheses of the classical  Koiter model. 
Moreover,  the conditions $
\mathcal{G}_\infty\in {\rm Sym}(2)$ and $\mathcal{T}_\infty=(0,0)$, i.e., \eqref{equ13},  coincide  with the conditions imposed by Tamba\v ca\footnote{Tamba\v ca  \cite[page 4, Definition of the set $\mathcal{A}^f$]{Tambaca-19} also requires ${Q}_{ \infty } \nabla y_0=\nabla m$, in order to arrive at the pure bending shell model. In this case  $\mathcal{G}_\infty=0$ and $\mathcal{R}_\infty= \, -(\nabla m)^{T} \nabla n- {\rm II}_{y_0}={\rm II}_{m}-{\rm II}_{y_0}\in {\rm Sym}(2)$. } \cite[page 4, Definition of the set $\mathcal{A}^K$]{Tambaca-19}.

Due to the equality
\begin{align}
[\nabla\Theta ]^{-T}
\begin{footnotesize}\left( \begin{array}{c|c}
\mathcal{G}_{ \infty } \ & \ 0 \vspace{4pt}\\
0 \  & \ 0
\end{array} \right)\end{footnotesize} [\nabla\Theta ]^{-1}=\mathcal{E}_{ \infty }=\sqrt{[\nabla\Theta ]^{-T}\,{\rm I}_m^{\flat }\,[\nabla\Theta ]^{-1}}-
\sqrt{[\nabla\Theta ]^{-T}\,{\rm I}_{y_0}^{\flat }\,[\nabla\Theta ]^{-1}},
\end{align}
we find  the expression of the change of metric tensor considered in the constrained Cosserat shell model, in terms of the first fundamental form
\begin{align}
\mathcal{G}_{ \infty }^\flat =[\nabla\Theta ]^{T}\Big(\sqrt{[\nabla\Theta ]^{-T}\,{\rm I}_m^{\flat }\,[\nabla\Theta ]^{-1}}-
\sqrt{[\nabla\Theta ]^{-T}\,{\rm I}_{y_0}^{\flat }\,[\nabla\Theta ]^{-1}}\,\Big)[\nabla\Theta ]^{-1}\in{\rm Sym}(3).
\end{align}

As we may see from \eqref{e90}, \eqref{minvarmc} and \eqref{eq5}, in the constrained Cosserat shell model the energy is expressed in terms of three strain measures: the change of metric tensor $\mathcal{G}_\infty\in {\rm Sym}(2)$, the nonsymmetric quantity $\mathcal{G}_\infty \,{\rm L}_{y_0}- \mathcal{R}_\infty$ which represents {\it the change of curvature} tensor and the  elastic shell bending--curvature tensor
\begin{align}
  \mathcal{K}_{\infty} & \coloneqq\,  \Big(\mathrm{axl}({Q}_{ \infty }^T\,\partial_{x_1} {Q}_{ \infty })\,|\, \mathrm{axl}({Q}_{ \infty }^T\,\partial_{x_2} {Q}_{ \infty })\,|0\Big)[\nabla\Theta ]^{-1}\not\in {\rm Sym}(3),\notag
\end{align} where
\begin{align}
{Q}_{ \infty }={\rm polar}\big((\nabla  m|n) [\nabla\Theta ]^{-1}\big)=(\nabla m|n)[\nabla\Theta ]^{-1}\,\sqrt{[\nabla\Theta ]\,\widehat {\rm I}_{m}^{-1}\,[\nabla\Theta ]^{T}}\in {\rm SO}(3).
\end{align}

 The choice of the name {\it the change of curvature} for  the nonsymmetric quantity $\mathcal{G}_\infty \,{\rm L}_{y_0}- \mathcal{R}_\infty$  will be justified in a forthcoming paper \cite{GhibaNeffPartIV},  in the framework of the linearized theory.  The  bending  strain tensor $\mathcal{R}_\infty$ generalizes the linear Koiter-Sanders-Budiansky bending measure \cite{budiansky1963best,koiter1973foundations} which  vanishes in infinitesimal pure stretch deformation of a quadrant of a cylindrical surface \cite{acharya2000nonlinear}, while the classical {\it bending strain tensor} tensor in the Koiter model  does not have this property (cf. the invariance  discussion in Section \ref{Ach2}).

For the bending-curvature energy density $ W_{\mathrm{bend,curv}} $ we can write its tensor argument $ \mathcal{K}_{\infty}$  in terms of the tensor $ \mathrm{C}_{y_0}\, \mathcal{K}_{\infty}$ and the vector $ \mathcal{K}_{\infty}^T\,n_0$, using Remark \ref{propAB} and according to the decomposition 
\begin{align}  \label{descK}
\mathcal{K}_{\infty} = {\rm A}_{y_0} \, \mathcal{K}_{\infty} +(0|0|n_0) \,(0|0|n_0)^T \, \mathcal{K}_{\infty}= \mathrm{C}_{y_0}( - \mathrm{C}_{y_0} \mathcal{K}_{\infty}) +(0|0|n_0) \,(0|0|\mathcal{K}_{\infty}^T\,n_0)^T\, \, .
\end{align} 

We can express  ${\rm C}_{y_0} \mathcal{K}_{\infty}$  in terms of the {\it bending strain} tensor $ \mathcal{R}_{\infty} $, see  \eqref{eq4} and \eqref{CK}
\begin{align}\label{CKA1}
{\rm C}_{y_0} \mathcal{K}_{\infty} =&\,-[\nabla\Theta \,]^{-T}
\mathcal{R}_{\infty}^\flat[\nabla\Theta]^{-1}\\
=&\,-\sqrt{[\nabla\Theta ]\,\widehat{\rm I}_m^{-1}[\nabla\Theta ]^{T}}\,[\nabla\Theta ]^{-T} {\rm II}_m^\flat[\nabla\Theta ]^{-1} +\sqrt{[\nabla\Theta ]\,\widehat{\rm I}_{y_0}^{-1}[\nabla\Theta ]^{T}}[\nabla\Theta ]^{-T}{\rm II}_{y_0}^\flat [\nabla\Theta ]^{-1}\notag.
\end{align}
From here, we find that in the constrained Cosserat shell model, the non-symmetric bending strain tensor has the following expression in terms of the first and  second fundamental form
\begin{align}\,
\mathcal{R}_{\infty}^\flat=\,[\nabla\Theta \,]^{T}\Big(&\sqrt{[\nabla\Theta ]\,\widehat{\rm I}_m^{-1}[\nabla\Theta ]^{T}}\,[\nabla\Theta ]^{-T} {\rm II}_m^\flat[\nabla\Theta ]^{-1} -\sqrt{[\nabla\Theta ]\,\widehat{\rm I}_{y_0}^{-1}[\nabla\Theta ]^{T}}[\nabla\Theta ]^{-T}{\rm II}_{y_0}^\flat [\nabla\Theta ]^{-1}\Big)\nabla\Theta.
\end{align}

Moreover, for in-extensional deformations
${\rm I}_m={\rm I}_{y_0}$ (pure flexure), the bending strain tensor turns into
\begin{align}\,
\mathcal{R}_{\infty}^\flat=\,[\nabla\Theta \,]^{T}\sqrt{[\nabla\Theta ]\,\widehat{\rm I}_{y_0}^{-1}[\nabla\Theta ]^{T}}\,[\nabla\Theta ]^{-T}\big( {\rm II}_m^\flat -{\rm II}_{y_0}^\flat \big)[\nabla\Theta ]^{-1}\nabla\Theta={\rm II}_m^\flat -{\rm II}_{y_0}^\flat =\mathcal{R}_{\rm Koiter}^\flat\in{\rm Sym}(3)\notag.
\end{align}

Here, the bending strain tensor is incorporated into both the membrane-bending energy 
\begin{align}\label{mbenergy}
W_{\mathrm{memb,bend}}^\infty\big(  \mathcal{E}_\infty ,\,  \mathcal{K}_\infty \big)=& \Big(\dfrac{h^3}{12}\,-{\rm K}\,\dfrac{h^5}{80}\Big)\,
W_{{\rm shell}}^{\infty}  \big( [\nabla\Theta \,]^{-T} [\mathcal{R}_{\infty}-2\,\mathcal{G}_{\infty} \,{\rm L}_{y_0} ]^\flat
[\nabla\Theta \,]^{-1}\big) 
\vspace{2.5mm}\notag\\&+\dfrac{h^3}{3} \mathrm{ H}\,\mathcal{W}_{{\rm shell}}^{\infty}  \big(  [\nabla\Theta \,]^{-T}
(\mathcal{G}_{\infty})^\flat [\nabla\Theta \,]^{-1} , [\nabla\Theta \,]^{-T} [\mathcal{R}_{\infty}-2\,\mathcal{G}_{\infty}\,{\rm L}_{y_0} ]^\flat
[\nabla\Theta \,]^{-1} \big)\notag\\&-
\dfrac{h^3}{6}\, \mathcal{W}_{{\rm shell}}^{\infty}  \big(  [\nabla\Theta \,]^{-T}
(\mathcal{G}_{\infty})^\flat [\nabla\Theta \,]^{-1} , [\nabla\Theta \,]^{-T}
[(\mathcal{R}_{\infty}-2\,\mathcal{G}_{\infty} \,{\rm L}_{y_0})\,{\rm L}_{y_0}]^\flat [\nabla\Theta \,]^{-1}\big)\vspace{2.5mm}\notag\\&+ \,\dfrac{h^5}{80}\,\,
W_{\mathrm{mp}}^{\infty} \big([\nabla\Theta \,]^{-T}
[(\mathcal{R}_{\infty}-2\,\mathcal{G}_{\infty} \,{\rm L}_{y_0})\,{\rm L}_{y_0} ]^\flat [\nabla\Theta \,]^{-1}\big) \notag
\end{align}
and the bending-curvature energy 
\begin{align}
W_{\mathrm{bend,curv}}\big(  \mathcal{K}_\infty    \big) = &\,  \,\Big(h-{\rm K}\,\dfrac{h^3}{12}\Big)\,
W_{\mathrm{curv}}\big(  \mathcal{K}_\infty \big)    +  \Big(\dfrac{h^3}{12}\,-{\rm K}\,\dfrac{h^5}{80}\Big)\,
W_{\mathrm{curv}}\big(  \mathcal{K}_\infty   {\rm B}_{y_0} \,  \big)  + \,\dfrac{h^5}{80}\,\,
W_{\mathrm{curv}}\big(  \mathcal{K}_\infty   {\rm B}_{y_0}^2  \big).
\end{align}

The remaining part $ \mathcal{K}_\infty^T\,n_0$ from \eqref{descK} is completely characterized by the (row) vector
\begin{equation}
\label{e5d}
\mathcal{N}_\infty \coloneqq  n_0^T\, \big(\mbox{axl}({Q}_{ \infty }^T\partial_{x_1}{Q}_{ \infty })\,|\, \mbox{axl}({Q}_{ \infty }^T\partial_{x_2}{Q}_{ \infty }) \big)
\end{equation}
which is   called the vector of {\it drilling bendings} \cite{Pietraszkiewicz14}. 
The vector of drilling bendings $  \mathcal{N}_\infty $ does not vanish in general and it  is present only due to the bending-curvature energy.

The membrane-bending energy is expressed in terms of $\mathcal{G}_{\infty}$,
$(\mathcal{R}_{\infty}-2\,\mathcal{G}_{\infty} \,{\rm L}_{y_0})^\flat$ and\linebreak $[(\mathcal{R}_{\infty}-2\,\mathcal{G}_{\infty} \,{\rm L}_{y_0})\,{\rm L}_{y_0}]^\flat$ and  the non-modified constrained variational problem imposes the following additional symmetry conditions, see \eqref{contsym} and \eqref{eq5}:
\begin{align}
(\mathcal{R}_{\infty}-2\,\mathcal{G}_{\infty} \,{\rm L}_{y_0})^\flat&=[\nabla \Theta]^T \sqrt{[\nabla\Theta ]^{-T}\,\widehat{\rm I}_m\,[\nabla\Theta ]^{-1}}\,  [\nabla\Theta ]\Big({\rm L}_m^\flat-{\rm L}_{y_0}^\flat \Big) \quad\ \stackrel{!}{\in}{\rm Sym}(3)\qquad \quad \textrm{and}\notag
\\
[(\mathcal{R}_{\infty}-2\,\mathcal{G}_{\infty} \,{\rm L}_{y_0})\,{\rm L}_{y_0}]^\flat&=[\nabla \Theta]^T\sqrt{[\nabla\Theta ]^{-T}\,\widehat{\rm I}_m\,[\nabla\Theta ]^{-1}}\,  [\nabla\Theta ]\Big({\rm L}_m^\flat-{\rm L}_{y_0}^\flat \Big){\rm L}_{y_0}^\flat\stackrel{!}{\in}{\rm Sym}(3).
\end{align}
In the modified constrained Cosserat shell model presented in Section \ref{mcm} these  constraints are excluded from the variational formulation, since the energy density is then expressed in terms of 
$
\sym(\mathcal{R}_{\infty} -\mathcal{G}_{\infty} \,{\rm L}_{y_0})\in {\rm Sym}(2)$ and $ \sym[(\mathcal{R}_{\infty} -\mathcal{G}_{\infty} \,{\rm L}_{y_0})\,{\rm L}_{y_0}]\in {\rm Sym}(2).
$
It is clear that using the properties of the occurring quadratic forms, we may rewrite the membrane-bending energy only as quadratic energies having $\mathcal{G}_{\infty}$ and $\mathcal{R}_{\infty} $ as arguments, while the bending-curvature energy may be written in terms of $\mathcal{R}_{\infty} $ and $  \mathcal{N}_\infty $. However, the coercivity estimate of the total energy 	\begin{align}W^{\infty}(\mathcal{E}_{ \infty }, \mathcal{K}_{ \infty })=W_{\mathrm{memb}}^{\infty}\big(  \mathcal{E}_{ \infty } \big)+W_{\mathrm{memb,bend}}^{\infty}\big(  \mathcal{E}_{ \infty } ,\,  \mathcal{K}_{ \infty } \big)+W_{\mathrm{bend,curv}}\big(  \mathcal{K}_{ \infty }    \big)
\end{align}
i.e., 
\begin{align}
W^{\infty}(\mathcal{E}_{ \infty }, \mathcal{K}_{ \infty })\geq\,&
\dfrac{h}{12} a_1^+ \lVert \mathcal{E}_{ \infty }\rVert ^2+\dfrac{h^3}{12}\, a_2^+ \lVert
\mathcal{E}_{ \infty }{\rm B}_{y_0}+{\rm C}_{y_0}\, \mathcal{K}_{ \infty }  \rVert ^2 + a_3^+\frac{ h^3}{6}\lVert  \mathcal{K}_{ \infty }\rVert ^2, \ \ a_i^+>0,
\end{align}
or equivalently 
\begin{align}\label{coecivityl}
W^{\infty}(\mathcal{E}_{ \infty }, \mathcal{K}_{ \infty })\geq\,&
\dfrac{h}{12} a_1^+ \lVert [\nabla\Theta \,]^{-T}
(\mathcal{G}_{ \infty })^\flat [\nabla\Theta \,]^{-1} \rVert ^2\\&+\dfrac{h^3}{12}\, a_2^+ \lVert
[\nabla\Theta \,]^{-T} [\mathcal{R}_{ \infty }-2\,\mathcal{G}_{ \infty } \,{\rm L}_{y_0} ]^\flat [\nabla\Theta \,]^{-1} \rVert ^2 + a_3^+\frac{ h^3}{6}\lVert  \mathcal{K}_{ \infty }\rVert ^2, \ \ a_i^+>0,\notag
\end{align}
indicates that the total energy is controlled by the values of $\mathcal{G}_{ \infty }$, $(\mathcal{R}_{ \infty }-2\,\mathcal{G}_{ \infty } \,{\rm L}_{y_0})$ and $\mathcal{K}_{ \infty }$ individually. We have not been able to identify a similar coercivity inequality showing that the total energy is controlled by the values of $\mathcal{G}_{ \infty }$, $\mathcal{R}_{ \infty }$ and $\mathcal{K}_{ \infty }$ alone.
Moreover, the assumptions which follow from considering $\mu_{\rm c}\to \infty$ are also expressed naturally in terms of $\mathcal{R}_{ \infty }-2\,\mathcal{G}_{ \infty } \,{\rm L}_{y_0}$, i.e., the conditions $\mathcal{R}_{ \infty }-2\,\mathcal{G}_{ \infty } \,{\rm L}_{y_0}\stackrel{!}{\in}{\rm Sym}(3)$ and $(\mathcal{R}_{ \infty }-2\,\mathcal{G}_{ \infty } \,{\rm L}_{y_0})\,{\rm L}_{y_0}\stackrel{!}{\in}{\rm Sym}(3)$, see  estimate \eqref{26bis} and Section \ref{Msym}.

 This line of thought, beside some other arguments presented in the linearised framework by 
 Anicic and L\'eger \cite{anicic1999formulation}, see also \cite{anicic2001modele}, and more recently by {\v{S}}ilhav{\`y} \cite {vsilhavycurvature}, suggest  that the triple $\mathcal{G}_{ \infty }$, $\mathcal{R}_{ \infty }-2\,\mathcal{G}_{ \infty } \,{\rm L}_{y_0}$ and $\mathcal{K}_{ \infty }$ are  appropriate measures  to express the change of metric and of the curvatures ${\rm H}$ and ${\rm K}$, while the bending and drilling effects are both  additionally incorporated in the bending-curvature energy through the elastic shell bending-curvature tensor $\mathcal{K}_{\infty}$.
 
 In order to make connections with existing works in the literature on      6-parameter shell models \cite{Eremeyev06,Pietraszkiewicz-book04,Pietraszkiewicz10}, see also  \cite[Section 6]{GhibaNeffPartI}, we conclude that the membrane-bending energy $W_{\mathrm{memb,bend}}^\infty\big(  \mathcal{E}_\infty ,\,  \mathcal{K}_\infty \big)$ defined by \eqref{mbenergy}, i.e., the influence of the change of curvature tensor $\mathcal{R}_{ \infty }-2\,\mathcal{G}_{ \infty } \,{\rm L}_{y_0}$, is omitted if a constrained Cosserat shell model would be derived from other available simpler 6-parameter shell models \cite{Eremeyev06,Pietraszkiewicz-book04,Pietraszkiewicz10}, even if the bending strain tensor $\mathcal{R}_{ \infty }$ is present (through the presence of the curvature energy).

 \section{Scaling invariance of bending tensors}\label{sec:invariance}\setcounter{equation}{0}
  \subsection{Revisiting Acharya's invariance requirements for a bending strain tensor}\label{subsec:Acharya}
  This brings us to the question of how  to model the physical notion of bending: a clear understanding  of bending measures will certainly lead to a clear definition of  its work conjugate pair when the equilibrium equations are established. Alongside, it helps to  have a proper formulation of the traction boundary conditions, when needed. In this context, Acharya \cite[page 5519]{acharya2000nonlinear} has proposed  a set of modelling requirements for a bending strain tensor in any first order nonlinear shell theory:
  
  \begin{description}[style=multiline,leftmargin=3em]
  \item[AR1] \textit{``Being a strain measure, it should be a tensor that vanishes in rigid deformations".}
  	\item[AR2] \textit{``It should be based on a \textit{proper} tensorial comparison of the deformed and underformed curvature fields }[${\rm II}_{m}$ and ${\rm II}_{y_0}$]".
  	\item[AR3] \textit{``A vanishing bending strain at a point should be associated with any deformation that leaves the orientation of the unit normal field locally unaltered around that point."}
  \end{description}
  
  The first two requirements \textbf{AR1} and \textbf{AR2} are satisfied by all considered nonlinear bending tensors in the literature and both are physically intuitive, while the third requirement \textbf{AR3} suggests that a nonzero bending tensor should only be associated with a change of the orientation of tangent planes and that, for instance, a radial expansion of a cylinder should give a zero bending strain measure, since it produces no further bending deformation of the shell (but changes the curvature). 
  
 It is easy to see that   $\mathcal{R}_{\rm{Koiter}}={\rm II}_{m}-{\rm II}_{y_0}$ satisfies \textbf{AR2} and \textbf{AR1}, since rigid deformations keep the second fundamental form. But the latter is in general not the case for deformations which leave the normal field unaltered, so that $\mathcal{R}_{\rm{Koiter}}$ does not satisfy \textbf{AR3}, for calculations, see \eqref{Kn}.
  
  \begin{figure}[h!]
  	\centering\includegraphics{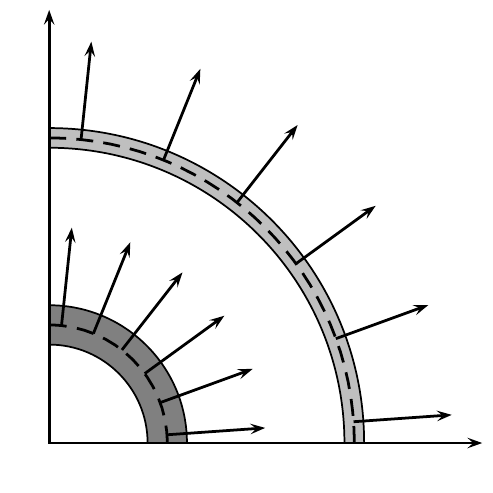}
  	\caption{\footnotesize A radial expansion of a cylinder preserves the tangent planes and therefore it should produce a zero bending strain tensor. However, the metric of the surface is changed, the curvature changes (the principal radius of curvature is changed),  but the normal is conserved. The radial expansion occurs rather because of in-plane stretch.}\label{radial}
  \end{figure}
  
  In \cite[Eq.~(8) and (10)]{acharya2000nonlinear} Acharya has proposed a bending strain tensor $\mathcal{R}_{\rm Acharya}$ for a first-order nonlinear elastic shell theory which in our notation reads (see Appendix \ref{AppAcharya})
  \begin{align}\label{nAch}
  \textrm{the first proposal:} \ \  
  &\widetilde{\mathcal{R}}_{\rm Acharya}
 \! =\!\!- \left([\nabla\Theta \,]^{-T} {\rm II}_{m}^\flat  [\nabla\Theta \,]^{-1}\!\! -\!\!
  \sqrt{[\nabla\Theta \,]^{-T}\; {\rm I}_{m}^\flat \; [\nabla\Theta \,]^{-1}}
  [\nabla\Theta \,]^{-T} {\rm II}_{y_0}^\flat \, [\nabla\Theta \,]^{-1}\right) \!\!\not\in {\rm Sym}(3),\notag \\\textrm{the second proposal:}\ \  & {\mathcal{R}_{\rm Acharya}}
  \!=\sym(\widetilde{\mathcal{R}}_{\rm Acharya})\in {\rm Sym}(3).
  \end{align}
  Acharya's bending tensors are similar to but do not coincide with the bending tensor appearing in our  nonlinear constrained Cosserat-shell model which reads
  \begin{align}\, \label{eq:unserTensor}
  \mathcal{R}_{\infty}^\flat=&\,[\nabla\Theta \,]^{T}\left(\sqrt{[\nabla\Theta ]\,\widehat{\rm I}_m^{-1}[\nabla\Theta ]^{T}}\,[\nabla\Theta ]^{-T} {\rm II}_m^\flat[\nabla\Theta ]^{-1} -[\nabla\Theta ]^{-T}{\rm II}_{y_0}^\flat [\nabla\Theta ]^{-1}\right)\nabla\Theta\\=&
  \,[\nabla\Theta \,]^{T}\sqrt{[\nabla\Theta ]\,\widehat{\rm I}_m^{-1}[\nabla\Theta ]^{T}}\left(\,[\nabla\Theta ]^{-T} {\rm II}_m^\flat[\nabla\Theta ]^{-1} -\sqrt{[\nabla\Theta ]^{-T}\,\widehat{\rm I}_m[\nabla\Theta ]^{-1}}[\nabla\Theta ]^{-T}{\rm II}_{y_0}^\flat [\nabla\Theta ]^{-1}\right)\nabla\Theta\not\in{\rm Sym}(3)\notag.
  \end{align}
  Interestingly, it is possible to express $\widetilde{\mathcal{R}}_{\rm Acharya}$ through  our nonlinear bending tensor $\mathcal{R}_{\infty}^\flat$. It holds (see Appendix \ref{AppAcharya})
  \begin{align}\label{reAca}
  \widetilde{\mathcal{R}}_{\rm Acharya}
  =&- \underbrace{\sqrt{[\nabla\Theta ]^{-T}\,{\rm I}_m^\flat[\nabla\Theta ]^{-1}}}_{\text{not invertible}}[\nabla\Theta \,]^{-T}\mathcal{R}_{\infty}^\flat[\nabla\Theta \,]^{-1}.
  \end{align}

  The nonlinear bending strain tensor $\mathcal{R}_{\rm Acharya}$ would satisfy all three requirements \textbf{AR1 - AR3} ``if locally pure stretch deformations are the only ones that leaves the orientation of tangent planes unaltered locally under deformation." 
  Incidentally, the tensor \eqref{nAch}$_2$ introduced by Acharya reduces, after linearization as well to the Koiter-Sanders-Budiansky ``best'' bending measure, see \cite{GhibaNeffPartIV} for details.
  According to Acharya, his nonlinear bending measure  should only be seen as a mathematical ``better alternative"   for modelling the physical bending process since ``the set of deformations that leave the orientation of tangent planes unaltered locally can be divided into two classes\,-\,deformations that have a pure stretch deformation gradient locally, and those that have their local rotation tensor field consisting of either [in-plane] \textit{drill} rotations or the identity tensor."  Acharya has shown that   his nonlinear  bending strain measure vanishes in pure stretch deformations that leave the normal unaltered, while the other classical bending strain measures fail to do so, but $\mathcal{R}_{\rm Acharya}$  does not necessarily vanish for deformations whose rotation tensor is a ``drill" rotation. Thus, $\mathcal{R}_{\rm Acharya}$ does not satisfy \textbf{AR3} in general. While physically appealing, condition \textbf{AR3} may, therefore, be too strict to be applicable in general. Hence we will introduce and investigate in the following a weaker invariance requirement. To this end let us introduce
  
  \newcommand{\U}{\mathrm{U_e}}
  
  \begin{definition}
  	Let $m$ be a deformation of the midsurface $y_0$. Denoting by $n$ and $n_0$ normal fields on the surface $m$ and $y_0$, respectively, we say that the midsurface deformation $m$ is obtained from a \textbf{pure elastic stretch} provided that ~$\U\coloneqq(\nabla m \,|\, n)\,(\nabla y_0\,|\,n_0)^{-1} = (\nabla m \,|\, n)\,[\nabla \Theta]^{-1}$ is symmetric and positive-definite, i.e., ~belongs to $\operatorname{Sym}^+(3)$.
  \end{definition}
  \begin{remark}
  	$\U$ is invertible by definition.
  \end{remark}
  
  With this definition,  Acharya's essential invariance requirement can be stated as
  \begin{description}[style=multiline,leftmargin=3em]
  	\item[AR3$^*$] \textit{A vanishing bending strain at a point should be associated with any deformation obtained from a pure elastic stretch that leaves the orientation of the unit normal field locally unaltered around that point.}
  \end{description}

  \begin{figure}[h!]
  	\centering\includegraphics[width=0.9\textwidth]{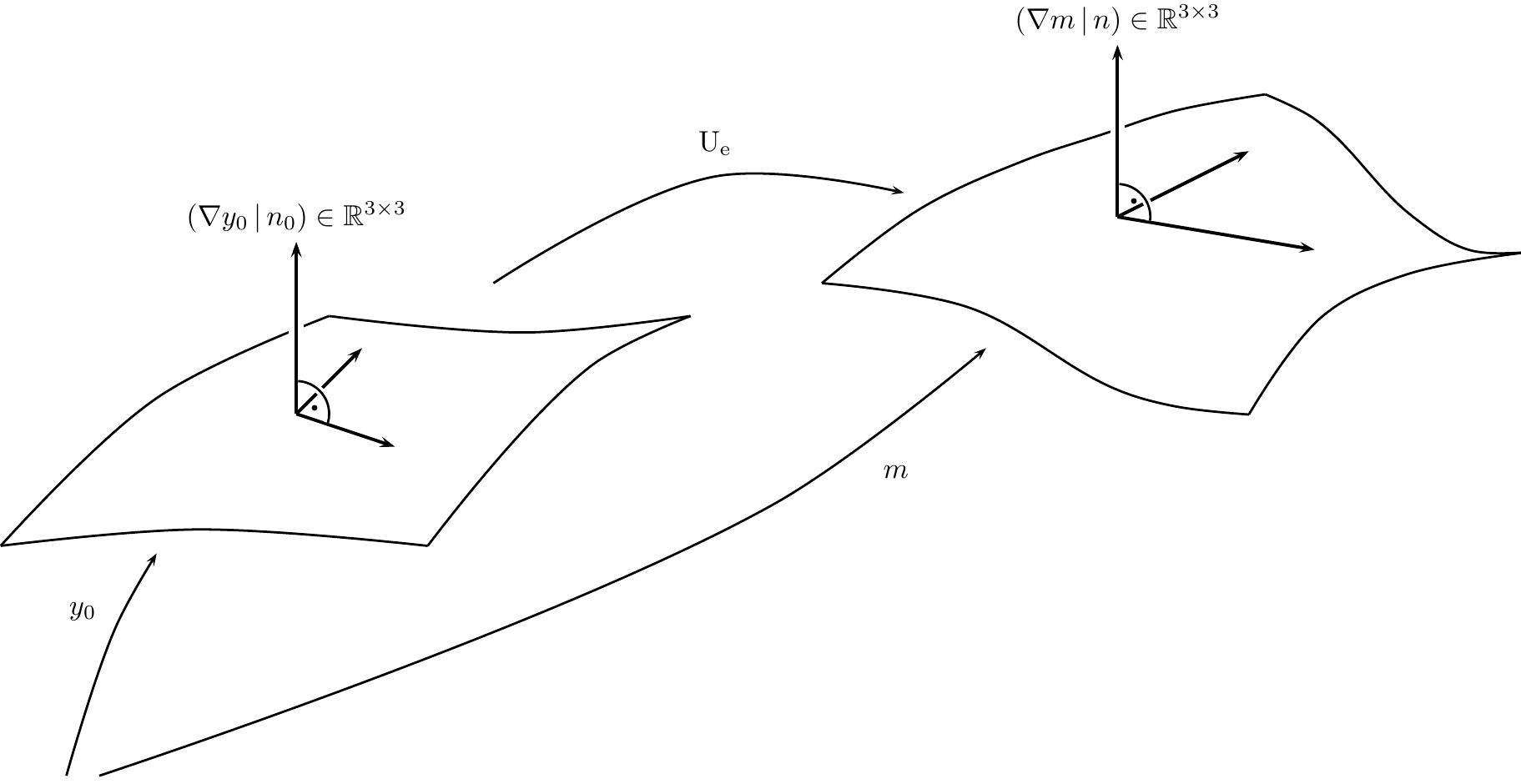}
  	\caption{\footnotesize [Multiplicative view] The initial curved midsurface  $y_0(\omega)$ and deformed midsurface $m(\omega)$, connected by a pure elastic stretch $\U\in {\rm Sym}^+(3)$. It leaves the normal unaltered if $\U\,(\nabla y_0\,|\,n_0)= (\nabla m \,|\, n_0)$, i.e., ~$n=\frac{\partial_{x_1}m\times \partial_{x_2}m}{\norm{\partial_{x_1}m\times\partial_{x_2}m}}=n_0$.}\label{stretching}
  \end{figure}
  \begin{example} Of special interest for us are deformations, which leave the unit normal field locally unaltered in a point. So, introducing locally orthogonal coordinates on the given curved initial configuration (which is always possible, cf.~e.g., ~\cite{Korn1914isoKoord})
  	\begin{equation}
  	\langle \partial_{x_1}y_0,\partial_{x_2}y_0\rangle=0, 
  	\end{equation}
  	the partial derivatives of $m$ must be linear combinations
  	\begin{subequations}
  		\begin{equation}\label{eq:linearkombi}
  		\partial_{x_1}m=\alpha\,\partial_{x_1}y_0 + \beta\,\partial_{x_2}y_0\quad \text{and} \quad  \partial_{x_2}m=a\,\partial_{x_1}y_0 + b\,\partial_{x_2}y_0 
  		\end{equation}
  		(in order to leave the unit normal field invariant), where we require additionally
  		\begin{equation}\label{eq:normalenerhaltendeBed}
  		\underbrace{\alpha\,b-\beta\,a}_{\rm orientation}>0 \quad \text{and} \quad \{\ \underset{\text{symmetry}}{\underbrace{a\,\norm{\partial_{x_1}y_0}^2=\beta\,\norm{\partial_{x_2}y_0}^2}}\quad \wedge \quad \underset{\text{positive-definiteness}}{\underbrace{\alpha\ge \beta\ge0, \ b \ge a\ge0}}
  		\}.
  		\end{equation}
  	\end{subequations}
  	The first condition \eqref{eq:normalenerhaltendeBed}$_1$ appears in the calculation of $\partial_{x_1}m \times \partial_{x_2}m$ and guarantees the existence and the orientation preservation of the normal, whereas the second condition \eqref{eq:normalenerhaltendeBed}$_2$ implies $(\nabla m\,|\,0) [\nabla\Theta \,]^{-1}$ is symmetric and positive-definite so that it is exactly the desired condition for the class of pure elastic stretches (in particular, $(\nabla m)^T\nabla y_0$ is symmetric). Indeed, we get (using orthogonal coordinates) for the inverse of $\nabla \Theta = (\nabla y_0 \,|\,n_0)$
  	\begin{equation}
  	[\nabla\Theta \,]^{-1} = \begin{pmatrix} \gamma\,(\partial_{x_1}y_0)^T \\ \delta\,(\partial_{x_2}y_0)^T \\ n_0^T \end{pmatrix}, \qquad \text{where}\quad
  	\gamma=\frac{1}{\norm{\partial_{x_1}y_0}^2} \quad  \text{ and } \quad \delta =\frac{1}{\norm{\partial_{x_2}y_0}^2}.\label{eq:coeffs}
  	\end{equation}
  	Recall, that for vectors $a,b,c,d,e,f\in\mathbb{R}^3$ it holds
  	\begin{equation}\label{eq:dyadicprod}
  	(a\,| \, b\, | \, c) \, (d\, | \, e\, | \,f)^T= a\otimes d + b\otimes e +c\otimes f
  	\end{equation}
  	so that
  	we obtain
  	\begin{align}\label{eq:nutzlich2}
  	\U&=(\nabla m|n)[\nabla\Theta \,]^{-1}  \overset{!}{=}(\alpha\,\partial_{x_1}y_0+\beta\,\partial_{x_2}y_0|a\,\partial_{x_1}y_0+b\,\partial_{x_2}y_0 |n_0)(\gamma\,\partial_{x_1}y_0|\delta\,\partial_{x_2}y_0|n_0)^T \notag\\ 
  	& = \alpha\,\gamma\,\partial_{x_1}y_0\otimes\partial_{x_1}y_0 +\beta\,\gamma\,\partial_{x_2}y_0\otimes \partial_{x_1}y_0+a\,\delta\,\partial_{x_1}y_0\otimes \partial_{x_2}y_0 +b\,\delta\,\partial_{x_2}y_0\otimes\partial_{x_2}y_0 +n_0\otimes n_0
  	\intertext{which is symmetric provided that  ~ $a\,\delta = \beta\,\gamma$, cf. \eqref{eq:normalenerhaltendeBed}$_2$ and \eqref{eq:coeffs}, since}
  	\mathrm{U}^T_{\mathrm{e}}&=\alpha\,\gamma\,\partial_{x_1}y_0\otimes\partial_{x_1}y_0 +a\,\delta\,\partial_{x_2}y_0\otimes \partial_{x_1}y_0+\beta\,\gamma\,\partial_{x_1}y_0\otimes \partial_{x_2}y_0 +b\,\delta\,\partial_{x_2}y_0\otimes\partial_{x_2}y_0 +n_0\otimes n_0\,.
  	\end{align}
  	Furthermore, $\U$ is positive-definite because
  	\begin{align}
  	\U&= (\alpha-\beta)\,\gamma\,\partial_{x_1}y_0\otimes\partial_{x_1}y_0 + \beta\,\gamma\,(\partial_{x_1}y_0\otimes\partial_{x_1}y_0 +\partial_{x_2}y_0\otimes\partial_{x_1}y_0)\notag\\
  	&\qquad + a\,\delta\,(\partial_{x_1}y_0\otimes \partial_{x_2}y_0 + \partial_{x_2}y_0\otimes \partial_{x_2}y_0 ) + (b-a)\,\delta\,\partial_{x_2}y_0\otimes \partial_{x_2}y_0 + n_0\otimes n_0 \notag\\
  	& \overset{\mathclap{a\,\delta = \beta\,\gamma}}{=}\quad (\alpha-\beta)\,\gamma\,\partial_{x_1}y_0\otimes\partial_{x_1}y_0 +  \beta\,\gamma\,(\partial_{x_1}y_0+\partial_{x_2}y_0)\otimes (\partial_{x_1}y_0+\partial_{x_2}y_0)\notag\\
  	&\quad\qquad + (b-a)\,\delta\,\partial_{x_2}y_0\otimes \partial_{x_2}y_0 + n_0\otimes n_0
    \end{align}
    is a sum of four positive (semi-)definite matrices since we required $\alpha\ge \beta\ge0$ and $\ b \ge a\ge0$. Acharya's examples \cite[sec. 6]{acharya2000nonlinear} satisfy both conditions \eqref{eq:normalenerhaltendeBed}.\footnote{Indeed, the biaxial stretching of a cylinder \cite[sec. 6.2]{acharya2000nonlinear} satisfies \eqref{eq:normalenerhaltendeBed} with $a=\beta=0$ (see \cite[eq. (34)]{acharya2000nonlinear}) and for the uniform normal deflection \cite[sec. 6.1]{acharya2000nonlinear} we have ~ $m=y_0+ c\, n_0$ ~ where $c\in\mathbb{R}$ is a fixed factor. Hence,
  		\begin{equation*}
  		\partial_{x_1}m=\partial_{x_1}y_0+c\,\partial_{x_1}n_0 \quad \text{and}\quad \partial_{x_2}m=\partial_{x_2}y_0+c\,\partial_{x_2}n_0.
  		\end{equation*}
  		So that, ~$
  		a\,\norm{\partial_{x_1}y_0}^2 \overset{\eqref{eq:linearkombi}}{=} \langle \partial_{x_2}m,\partial_{x_1}y_0\rangle = c\,\langle \partial_{x_2}n_0,\partial_{x_1}y_0\rangle \overset{(\ast)}{=} c\,\langle \partial_{x_1}n_0,\partial_{x_2}y_0\rangle = \langle \partial_{x_1}m,\partial_{x_2}y_0\rangle \overset{\eqref{eq:linearkombi}}{=}\beta\,\norm{\partial_{x_2}y_0}^2$ ~ where in $(\ast)$ we made use of the symmetry of the second fundamental form. Furthermore, since $\norm{n_0}^2=1$ implies that $\langle \partial_{x_i}n_0,n_0\rangle =0$ we have $\partial_{x_1}n_0=\theta_1\ \partial_{x_1}y_0+\theta_2\ \partial_{x_2}y_0$ and $\partial_{x_2}n_0=\vartheta_1\ \partial_{x_1}y_0+\vartheta_2\ \partial_{x_2}y_0$ so that $\partial_{x_1}m\times\partial_{x_2}m=((1+c\,\theta_1)(1+c\,\vartheta_2)-c^2\,\theta_2\vartheta_1)\,\partial_{x_1}y_0\times\partial_{x_2}y_0 $, implying that the range of the fixed value $c$ should be adjusted in such a way that the prefactor is positive and $1+c\,\theta_1\ge c\,\theta_2\ge0$ as well as $1+c\,\vartheta_2\ge c\,\vartheta_1\ge0$, which are exactly the requirements for the positive-definiteness improved in \eqref{eq:normalenerhaltendeBed}.
  	}
  \end{example}
  
  \subsection{Investigation of the invariance requirement for a bending tensor}\label{Ach2}
  
  Using \eqref{eq:dyadicprod} we have always (not only in orthogonal coordinates)
  \begin{equation}\label{eq:sehrnuetzlich}
  (\nabla y_{0}\, |\, 0)[\nabla\Theta]^{-1} = (\nabla y_{0}\,|\,n_0)[\nabla\Theta]^{-1}-(0\,|\,0\,|\,n_0)[\nabla\Theta]^{-1}\overset{\eqref{eq:dyadicprod}}{=} \id_3 - n_0\otimes n_0,
  \end{equation}
  since ~ $e_3^T [\nabla\Theta]^{-1}=n_0^T$  because $[\nabla\Theta]^{-1}=\begin{pmatrix}\cdots \\ \cdots \\ n_0^T  \end{pmatrix}$ ~ where $\nabla\Theta= (\nabla y_0 \, | \, n_0)$.
  
  We now show, that both tensors $\mathcal{R}_{\rm Acharya}$ and $\mathcal{R}_{\infty}^\flat$ satisfy the three requirements \textbf{AR1}, \textbf{AR2} and \textbf{AR3}$^*$. The expressions of both tensors \eqref{nAch} and \eqref{eq:unserTensor} fulfill condition \textbf{AR2}. Moreover, for a rigid deformation $y_0\to m =\widehat{Q}y_0$, $\widehat{Q}\in\operatorname{SO}(3)$ the fundamental forms coincide, 
  $
  {\rm I}_m = {\rm I}_{y_{0}} \quad \text{and} \quad {\rm II}_{m}={\rm II}_{y_{0}},
$
  due to the identities $\nabla m=\widehat{Q} \nabla y_0, \ \nabla n=\nabla(\widehat{Q}\, n_0)=\widehat{Q}\,\nabla\, n_0$.
  Hence,  for a rigid deformation, we obtain
  \begin{align}
  \sqrt{[\nabla\Theta \,]^{-T}\; {\rm I}_{m}^\flat \; [\nabla\Theta \,]^{-1}} &= \sqrt{[\nabla\Theta \,]^{-T}\; {\rm I}_{y_{0}}^\flat \; [\nabla\Theta \,]^{-1}} = \sqrt{[\nabla\Theta \,]^{-T}(\nabla y_{0} | 0)^T\; (\nabla y_{0} | 0)[\nabla\Theta \,]^{-1}}\notag\\
  &= \sqrt{((\nabla y_{0} | 0)[\nabla\Theta \,]^{-1})^T\; (\nabla y_{0} | 0)[\nabla\Theta \,]^{-1}} \overset{\eqref{eq:sehrnuetzlich}}{=} \sqrt{(\id_3 - n_0\otimes n_0)(\id_3 - n_0\otimes n_0)}\notag\\
  &=\id_3 - n_0\otimes n_0, \label{eq:Wurelgut}
  \end{align}
  since $\id_3 - n_0\otimes n_0$ is positive semi-definite because $\langle (\id_3 - n_0\otimes n_0)\,\xi,\xi\rangle = \norm{\xi}^2-\langle n_0,\xi\rangle^2 \ge0$ for all $\xi\in\mathbb{R}^3$.
  Keeping in mind that ~ $[\nabla\Theta \,]^{-1} n_0 = e_3$ ~ it follows that for a rigid deformation we have
  \begin{align}
  \widetilde{\mathcal{R}}_{\rm Acharya}&= - [\nabla\Theta \,]^{-T}\; {\rm II}_{m}^\flat \; [\nabla\Theta \,]^{-1}\!\! +\!\!
  \sqrt{[\nabla\Theta \,]^{-T}\; {\rm I}_{m}^\flat \; [\nabla\Theta \,]^{-1}}
  [\nabla\Theta \,]^{-T}\, {\rm II}_{y_0}^\flat \, [\nabla\Theta \,]^{-1} \\
  &\overset{\mathclap{\eqref{eq:Wurelgut}}}{=} \ -[\nabla\Theta \,]^{-T}{\rm II}_{y_{0}}^\flat[\nabla\Theta \,]^{-1}+(\id_3-n_0\otimes n_0)[\nabla\Theta \,]^{-T}{\rm II}_{y_{0}}^\flat[\nabla\Theta \,]^{-1}\notag\\
  & =-n_0\otimes n_0\,[\nabla\Theta \,]^{-T}{\rm II}_{y_{0}}^\flat[\nabla\Theta \,]^{-1}
  =n_0\otimes ([\nabla\Theta \,]^{-T}{\rm II}_{y_{0}}^\flat[\nabla\Theta \,]^{-1}n_0)=n_0\otimes ([\nabla\Theta \,]^{-T}{\rm II}_{y_{0}}^\flat e_3)= 0_3.\notag
  \end{align}
  Similarly, for a rigid deformation, we have
  \begin{align}
  \sqrt{[\nabla\Theta ]\,\widehat{\rm I}_m^{-1}[\nabla\Theta ]^{T}}  &= \{[\nabla\Theta \,]^{-T}\; \widehat{\rm I}_{m} \; [\nabla\Theta \,]^{-1}\}^{-\frac12} = \{[\nabla\Theta \,]^{-T}\;( {\rm I}_{m}^\flat +e_3\otimes e_3) \; [\nabla\Theta \,]^{-1}\}^{-\frac12}\notag\\
  &=\{[\nabla\Theta \,]^{-T}\; {\rm I}_{m}^\flat \; [\nabla\Theta \,]^{-1} + [\nabla\Theta \,]^{-T}\; e_3\otimes e_3 \; [\nabla\Theta \,]^{-1}\}^{-\frac12}\notag\\
  &\overset{\mathclap{\eqref{eq:Wurelgut}}}{=}\ \{(\id_3 - n_0\otimes n_0)^2+n_0\otimes n_0 \}^{-\frac12}= \{\id_3 -n_0\otimes n_0+n_0\otimes n_0\}^{-\frac12}= \id_3, \label{eq:Wurelgueter}
  \end{align}
  so that we conclude
  \begin{align}
  \mathcal{R}_{\infty}^\flat&=[\nabla\Theta \,]^{T}\Big(\sqrt{[\nabla\Theta ]\,\widehat{\rm I}_m^{-1}[\nabla\Theta ]^{T}}\,[\nabla\Theta ]^{-T} {\rm II}_m^\flat[\nabla\Theta ]^{-1} -[\nabla\Theta ]^{-T}{\rm II}_{y_0}^\flat [\nabla\Theta ]^{-1}\Big)\nabla\Theta\notag\\
  &{\overset{\mathclap{\eqref{eq:Wurelgueter}}}{=}}\ \  [\nabla\Theta \,]^{T}\Big(\id_3\,[\nabla\Theta ]^{-T}\,{\rm II}_{y_0}^\flat[\nabla\Theta ]^{-1} -[\nabla\Theta ]^{-T}{\rm II}_{y_0}^\flat [\nabla\Theta ]^{-1}\Big)\nabla\Theta = 0_3,
  \end{align}
  and we have shown, that both tensors vanish for rigid deformations, so that both satisfy  \textbf{AR1}.
  
  It remains to check \textbf{AR3}$^*$ for $\widetilde{\mathcal{R}}_{\rm Acharya}$ and $\mathcal{R}_{\infty}^\flat$ for pure elastic stretches that in addition leave the unit normal field unaltered in a point, i.e., ~we have
  \begin{align}
  \U&\,\,\,\coloneqq (\nabla m \,|\, n)\,(\nabla y_0\,|\,n_0)^{-1} \overset{\textcircled{\footnotesize 1}}{=} (\nabla m \,|\, n_0)\,[\nabla \Theta]^{-1} = (\nabla m \,|0)\,[\nabla \Theta]^{-1} + (0\,|\,0\,|\,n_0)\,[\nabla \Theta]^{-1}\notag\\
  &\overset{\eqref{eq:dyadicprod}}{=} (\nabla m \,|0)\,[\nabla \Theta]^{-1} + n_0\otimes n_0 \ \overset{\textcircled{\footnotesize 2}}{\in}\operatorname{Sym}^+(3), \label{eq:defU}
  \end{align}
  where \textcircled{\footnotesize 1} requires that the unit normal field is left unaltered and \textcircled{\footnotesize 2} requires $m$ to be obtained from a pure elastic stretch.
  Thus,
  \begin{align}
  \U\,n_0&= (\nabla m \,|0)\,[\nabla \Theta]^{-1}\,n_0 + (n_0\otimes n_0)\,n_0 = (\nabla m \,|0)\,e_3+n_0=n_0 \label{eq:Un}
  \end{align} and, since the matrix $\U$ is invertible, also ~$\mathrm{U}^{-1}_{\mathrm{e}}\,n_0=n_0.$ With ~ ${\rm I}_{m}^\flat=(\nabla m |0)^T(\nabla m | 0)$ ~ we obtain
  \begin{align} \label{eq:ersteWurzel}
  [\nabla\Theta \,]^{-T}\; {\rm I}_{m}^\flat \; [\nabla\Theta \,]^{-1} &= [\nabla\Theta \,]^{-T}\; (\nabla m |0)^T(\nabla m | 0)\; [\nabla\Theta \,]^{-1}\overset{\eqref{eq:defU}}{=}(\U-n_0\otimes n_0)^T\,(\U-n_0\otimes n_0)\notag\\
  &= (\U-n_0\otimes n_0)^2
  \end{align}
  since $\U$ is symmetric (by requirement \eqref{eq:defU}). Moreover, the positive definite square root of $\U$ fulfills also ~$\sqrt{\U}\,n_0=n_0 $, thus ~ $(\sqrt{\U}-n_0\otimes n_0)^2=(\sqrt{\U}-n_0\otimes n_0)(\sqrt{\U}-n_0\otimes n_0)^T=\U-n_0\otimes n_0$ ~ is positive semi-definite and we conclude
  \begin{align}
  \sqrt{[\nabla\Theta \,]^{-T}\; {\rm I}_{m}^\flat \; [\nabla\Theta \,]^{-1}}&\,[\nabla\Theta \,]^{-T}(\nabla y_0|0)^T = \sqrt{(\U-n_0\otimes n_0)^2}\,((\nabla y_{0} | 0)[\nabla\Theta]^{-1})^T\notag\\
  &\overset{\mathclap{\eqref{eq:sehrnuetzlich}}}{=} \ (\U-n_0\otimes n_0)(\id_3-n_0\otimes n_0) = \U -\U\,n_0\otimes n_0\overset{\eqref{eq:Un}}{=}\ \U -n_0\otimes n_0\,. \label{eq:ersterFaktor}
  \end{align}
  Using \eqref{eq:defU} and \eqref{eq:ersterFaktor} in the expression \eqref{nAch} it now follows that Acharya's bending tensor vanishes for pure elastic stretches that in addition leave the unit normal field unaltered. Indeed, having ~ ${\rm II}_{m}^\flat=-(\nabla m|0)^T(\nabla n|0) = -(\nabla m|0)^T(\nabla n_0|0)$, ~ so that again using the symmetry of $\U$ we obtain
  \begin{align}
  \widetilde{\mathcal{R}}_{\rm Acharya} & \overset{\eqref{nAch}}{=}[\nabla\Theta]^{-T}(\nabla m | 0)^T(\nabla n_0|0)[\nabla \Theta]^{-1}-\sqrt{[\nabla\Theta \,]^{-T}\; {\rm I}_{m}^\flat \; [\nabla\Theta \,]^{-1}}\; [\nabla\Theta \,]^{-T}(\nabla y_0|0)^T(\nabla n_0|0)[\nabla \Theta]^{-1}\notag\\
  &\overset{\eqref{eq:defU}}{\underset{\eqref{eq:ersterFaktor}}{=}} (\U -n_0\otimes n_0)^T\,(\nabla n_0|0)[\nabla \Theta]^{-1}-(\U -n_0\otimes n_0)\,(\nabla n_0|0)[\nabla \Theta]^{-1} =0_3,
  \end{align}
  all in all, we have shown, that $\widetilde{\mathcal{R}}_{\rm Acharya}$ satisfies \textbf{AR1}, \textbf{AR2} and \textbf{AR3}$^*$.
  
  Continuing, our derived bending tensor $\mathcal{R}_{\infty}^\flat$ has the same properties as Acharya's bending tensor $\mathcal{R}_{\rm Acharya}$. In order to show this, we deduce using
  \begin{equation}(\U - n_0\otimes n_0)^2=\mathrm{U}_{\mathrm{e}}^2 -\U\,n_0\otimes n_0 - n_0\otimes n_0 \,\U+  n_0\otimes n_0 \overset{\eqref{eq:Un}}{=} \mathrm{U}_{\mathrm{e}}^2-n_0\otimes n_0,\end{equation}
  that
  \begin{align}\label{eq:zweitewurzel}
  \sqrt{[\nabla\Theta ]\,\widehat{\rm I}_m^{-1}[\nabla\Theta ]^{T}}  &= \{[\nabla\Theta \,]^{-T}\; \widehat{\rm I}_{m} \; [\nabla\Theta \,]^{-1}\}^{-\frac12} =\{[\nabla\Theta \,]^{-T}\; {\rm I}_{m}^\flat \; [\nabla\Theta \,]^{-1} + [\nabla\Theta \,]^{-T}\; e_3\otimes e_3 \; [\nabla\Theta \,]^{-1}\}^{-\frac12}\notag\\
  &\overset{\mathclap{\eqref{eq:ersteWurzel}}}{=} \ \{(\U - n_0\otimes n_0)^2+ n_0\otimes n_0  \}^{-\frac12}
  = \{\mathrm{U}_{\mathrm{e}}^2-n_0\otimes n_0 + n_0\otimes n_0  \}^{-\frac12} = \mathrm{U}_{\mathrm{e}}^{-1},
  \end{align}
  where we have used that $\U$ is positive-definite. Together with the definition of the symmetric matrix $\U$ we get
  \begin{align}\label{eq:zweiterFaktor}
  \sqrt{[\nabla\Theta ]\,\widehat{\rm I}_m^{-1}[\nabla\Theta ]^{T}}\, [\nabla\Theta \,]^{-T}(\nabla m|0)^T \overset{\eqref{eq:zweitewurzel}}{\underset{\eqref{eq:defU}}{=}} \mathrm{U}_{\mathrm{e}}^{-1}\, (\U-n_0\otimes n_0) =  \id_3 - \mathrm{U}_{\mathrm{e}}^{-1}\,n_0\otimes n_0  \overset{\eqref{eq:Un}}{=}\id_3 -n_0\otimes n_0. 
  \end{align}
  Again, using the intermediate steps \eqref{eq:sehrnuetzlich} and \eqref{eq:zweiterFaktor} in the expression \eqref{eq:unserTensor} we obtain  $\mathcal{R}_{\infty}^\flat=0_3$ since
  \begin{align}
  \mathcal{R}_{\infty}^\flat & \overset{\eqref{eq:unserTensor}}{=}[\nabla\Theta \,]^{T}\Big(\sqrt{[\nabla\Theta ]\,\widehat{\rm I}_m^{-1}[\nabla\Theta ]^{T}}\,[\nabla\Theta ]^{-T} {\rm II}_m^\flat[\nabla\Theta ]^{-1} -[\nabla\Theta ]^{-T}{\rm II}_{y_0}^\flat [\nabla\Theta ]^{-1}\Big)\nabla\Theta\notag\\
  &\ \, =[\nabla\Theta \,]^{T}\Big(-\sqrt{[\nabla\Theta ]\,\widehat{\rm I}_m^{-1}[\nabla\Theta ]^{T}}\,[\nabla\Theta ]^{-T} (\nabla m|0)^T\,(\nabla n_0| 0)[\nabla\Theta ]^{-1} +[\nabla\Theta ]^{-T}(\nabla y_0|0)^T\,(\nabla n_0|0) [\nabla\Theta ]^{-1}\Big)\nabla\Theta\notag\\
  & \overset{\eqref{eq:sehrnuetzlich}}{\underset{\eqref{eq:zweiterFaktor}}{=}}[\nabla\Theta \,]^{T}\Big(-(\id_3 - n_0\otimes n_0)\,(\nabla n_0| 0)[\nabla\Theta ]^{-1} + (\id_3 - n_0\otimes n_0)\,(\nabla n_0|0) [\nabla\Theta ]^{-1}\Big)\nabla\Theta = 0_3,
  \end{align}
  so that also $\mathcal{R}_{\infty}^\flat $ satisfies \textbf{AR1}, \textbf{AR2} and  \textbf{AR3}$^*$. In fact, due to \eqref{reAca}, it is clear that a vanishing  tensor $\mathcal{R}_{\infty}^\flat $ leads to a vanishing tensor $\widetilde{\mathcal{R}}_{\rm Acharya}$. The reverse is not true, since the relation \eqref{reAca}  is not invertible. It is unclear  if there are deformations which do not preserve the normal and for which $\widetilde{\mathcal{R}}_{\rm Acharya}$ vanishes while $\mathcal{R}_{\infty}^\flat $ does not.
  
  We have already shown, that $\mathcal{R}_{\rm Koiter}= {\rm II}_{m}-{\rm II}_{y_0}$ satisfies  \textbf{AR1}, \textbf{AR2}  but it does not satisfy \textbf{AR3}$^*$ as we show presently. Consider $[\nabla\Theta ]^{-T}\,({\rm II}_{m}^\flat-{\rm II}_{y_0}^\flat)\,[\nabla\Theta ]^{-1}$. Indeed, the latter expression  does not satisfy \textbf{AR3}$^*$ since
  \begin{align}\label{Kn}
  [\nabla\Theta ]^{-T}&\,({\rm II}_{m}^\flat-{\rm II}_{y_0}^\flat)\,[\nabla\Theta ]^{-1} = -\,[\nabla\Theta ]^{-T} (\nabla m|0)^T\,(\nabla n_0| 0)[\nabla\Theta ]^{-1}+[\nabla\Theta ]^{-T}(\nabla y_0|0)^T\,(\nabla n_0|0) [\nabla\Theta ]^{-1} \\
  & \overset{\mathclap{\eqref{eq:defU}}}{\underset{\mathclap{\eqref{eq:sehrnuetzlich}}}{=}}\ -(\U - n_0\otimes n_0)\,(\nabla n_0| 0)[\nabla\Theta ]^{-1} + (\id_3 - n_0\otimes n_0)\,(\nabla n_0|0) [\nabla\Theta ]^{-1}  = -(\U-\id_3)\,(\nabla n_0|0) [\nabla\Theta ]^{-1} \notag
  \end{align}
  does not generally vanish for deformations obtained from a pure elastic stretch that leave the normal field unaltered. From \eqref{Kn} it follows that $\mathcal{R}_{\rm Koiter}^\flat= -[\nabla\Theta ]^{T}(\U-\id_3)\,(\nabla n_0|0) \neq 0$ under the same circumstances.

   Finally, let us strengthen further the additional invariance requirements on the bending tensor by postulating 
  \begin{description}[style=multiline,leftmargin=5em]
  	\item[AR3$^*_{\rm plate}$] \textit{For a planar reference geometry $({\rm II}_{y_0}\equiv0_2)$ the bending tensor should be invariant under the scaling $m\to \alpha\, m$, $\alpha>0$.}
  \end{description}
It is clear that the simple scaling $m\to \alpha\, m$, $\alpha>0$ corresponds to an additional in-plane stretch, while there is no additional bending involved. Therefore, \textbf{AR3$^*_{\rm plate}$} should be adopted as a suitable requirement for a true bending tensor expression. Furthermore, the new condition
\textbf{AR3$^*_{\rm plate}$} allows to differentiate between  Acharya's ad hoc bending tensors $\widetilde{\mathcal{R}}_{\rm Acharya}$ and ${\mathcal{R}}_{\rm Acharya}$ and our derived bending tensor $\mathcal{R}_{\infty}^\flat$. Indeed, 
  \begin{remark}
  	The bending tensor $\mathcal{R}_{\infty}^\flat$ satisfies {\rm \textbf{AR3$^*_{\rm plate}$}} while $\widetilde{\mathcal{R}}_{\rm Acharya}$ and ${\mathcal{R}}_{\rm Acharya}$ do not have this invariance property, since in the  planar referential configuration ($\nabla \Theta=\id$, ${\rm I}_{\rm id}=\id$, ${\rm II}_{\rm id}=0_2$) Acharya's bending tensors reduce to $-{\rm II}_m^\flat$, which violate {\rm \textbf{AR3$^*_{\rm plate}$}}.  Both $\mathcal{R}_{\infty}^\flat$ and  ${\mathcal{R}}_{\rm Acharya}$ reduce after linearisation {\rm \cite{GhibaNeffPartIV}} to the Sanders and Budiansky bending tensor of the  {\rm \cite{budiansky1962best,budiansky1963best,koiter1973foundations}} ``best first-order linear elastic shell theory''. 
  \end{remark}
  
  \begin{remark}
   Let us mention the explicit dependence of the bending tensors  on the two configurations $y_0$ and $m$ such as $\widetilde{\mathcal{R}}_{\rm Acharya}(m(x_1,x_2),y_0(x_1,x_2))$. Since the scaling $\alpha\,y_0$, $\alpha>0$, leaves the normals invariant we have of course $\widetilde{\mathcal{R}}_{\rm Acharya}(\alpha\, y_0,y_0)\equiv0_3$ and $\mathcal{R}_\infty(\alpha\, y_0,y_0)\equiv0_3$. Now, for a planar reference configuration $y_0(x_1,x_2)=(x_1, x_2,0)^T\eqqcolon {\rm id}_{1,2}$ we obtain the scaling
   \begin{equation}
    \widetilde{\mathcal{R}}_{\rm Acharya}(\alpha\,m,{\rm id}_{1,2}) = \mathcal{R}_{\rm Koiter}^\flat(\alpha\,m,{\rm id}_{1,2})= \alpha\,\mathcal{R}_{\rm Koiter}^\flat(m,{\rm id}_{1,2}) = \alpha\,\widetilde{\mathcal{R}}_{\rm Acharya}(m,{\rm id}_{1,2})
   \end{equation}
  but  we have the scaling invariance
 \begin{equation}
  \mathcal{R}_{\infty}^\flat (\alpha\,m,{\rm id}_{1,2}) =  \frac{\alpha}{|\alpha|}\mathcal{R}_{\infty}^\flat (m,{\rm id}_{1,2}) \overset{\alpha>0}{=} \mathcal{R}_{\infty}^\flat (m,{\rm id}_{1,2}).
 \end{equation}
  \end{remark}

\section{Conclusion}
We have thoroughly investigated a recently introduced isotropic nonlinear Cosserat shell model. We focussed on what happens  when the independent Cosserat rotation is made to coincide with the continuum rotation. This case can be steered in the model under consideration by sending the Cosserat couple modulus $\mu_{\rm c}\to \infty$. In this way we obtained a constrained Cosserat shell, which incorporates strong symmetry requirements on the appearing  strain and bending-curvature tensors. For such a class of constrained models we proposed conditional existence theorems for the resulting minimization problem: conditional, since we cannot exclude that the admissible set is empty. This led us to consider a modified shell model in which certain symmetry requirements are waived. Then unconditional existence follows easily. A conceptual advantage of the constrained Cosserat shell model is that it can be entirely expressed in ``classical'' quantities from differential geometry (first and second fundamental forms of the initial and deformed surface, respectively). 
We have provided the necessary calculations for this identification  which allowed us to compare the resulting model with other classical models, like the Koiter shell model. While the membrane part of the constrained model incorporates only the respective first fundamental forms it turns out that our constrained model incorporates a bending tensor which always couples bending and membrane effects. This was neither expected nor aimed at originally.
To our great surprise, the new bending tensor satisfies  a generalized invariance condition, previously introduced by Acharya, while the Koiter bending tensor does not. We added a new invariance condition in the same spirit which allows us to differentiate further between existing proposals for bending tensors.   By identifying a true bending tensor validated by the novel invariance requirements \textbf{AR3}$^*$ and \textbf{AR3}$^*_{\rm plate}$ we observe a clear advantage of the constrained Cosserat shell model compared to more classical approaches: the obtained bending tensor measures really only bending (change in normals) and the additional Cosserat curvature tensor measures the remaining total curvature. Classical shell approaches do not have the possibility to incorporate a curvature tensor and their proposed bending tensors are not invariant under scaling.   We also expect manifest differences in predicting the small-scale wrinkling behaviors between a theory that shows scaling invariance of the bending tensor versus a theory that does not.
Since in the development of the model we took great care to approximate the shell as exact as possible it seems to be the simplifying assumptions in the classical Koiter shell model that destroy the mentioned, physically
appealing invariance condition.  In this respect more research seems to be necessary to generally ascertain the invariance condition in any first order nonlinear shell model. In a follow up paper \cite{GhibaNeffPartIV}, we will linearize the model and compare it to available linear shell models. The obtained clear  and consistent structure of the Cosserat model together with the sound invariance conditions which are automatically satisfied make us confident that the new Cosserat shell model will find its place when isotropic thin shell theory need to be actually applied in a FEM-context. In that case, however, it will rather be the unconstrained Cosserat shell model that will be implemented in order to obtain second order equilibrium equations. 

\bigskip

\begin{footnotesize}
	\noindent{\bf Acknowledgements:}   This research has been funded by the Deutsche Forschungsgemeinschaft (DFG, German Research Foundation) -- Project no. 415894848: NE 902/8-1 (P. Neff and P. Lewintan) and
	BI 1965/2-1 (M. B\^irsan). The  work of I.D. Ghiba  was supported by a grant of the Romanian Ministry of Research
	and Innovation, CNCS--UEFISCDI, Project no.
	PN-III-P1-1.1-TE-2019-0397, within PNCDI III.

	\bibliographystyle{plain} 
	
	

	\appendix\setcounter{equation}{0}
	\section*{Appendix}\setcounter{section}{1}
	\addcontentsline{toc}{section}{Appendix}
	\subsection{Useful identities}\label{ApropAB}
	We provide some   properties of the tensors  in the variational formulation of the shell models from \cite{GhibaNeffPartI,GhibaNeffPartII}:
	\begin{remark}\label{propAB}{\rm }
		The following identities are satisfied :
		\begin{itemize}
			\item [i)] 
			
			$\tr[{\rm A}_{y_0}]\,=\,2,$ \quad	${\det}[{\rm A}_{y_0}]\,=\,0;\quad $
			$\tr[{\rm B}_{y_0}]\,=\,2\,{\rm H}\,$,\quad  ${\det}[{\rm B}_{y_0}]\,=\,0,$ \\
			${\rm A}_{y_0} = [\nabla\Theta ]^{-T}\; {\rm I}_{y_0}^\flat \; [\nabla\Theta ]^{-1}=\,\id_3-(0|0|\nabla\Theta .e_3)\,[	\nabla\Theta ]^{-1}\,=\,\id_3-(0|0|n_0)\,(0|0|n_0)^T$, \\ $
			{\rm B}_{y_0} = [\nabla\Theta ]^{-T}\; {\rm II}_{y_0}^\flat \; [\nabla\Theta ]^{-1}$
			
			\item[ii)] ${\rm B}_{y_0}$ satisfies the equation of Cayley-Hamilton type
			$
			{\rm B}_{y_0}^2-2\,{\rm H}\, {\rm B}_{y_0}+{\rm K}\, {\rm A}_{y_0}\,=\,0_3;
			$
			\item[iii)] ${\rm A}_{y_0}{\rm B}_{y_0}\,=\,{\rm B}_{y_0}{\rm A}_{y_0}\,=\,{\rm B}_{y_0}$, \quad  ${\rm A}_{y_0}^2\,=\,{\rm A}_{y_0}$, \quad  ${\rm C}_{y_0}\in \mathfrak{so}(3)$, $\quad {\rm C}_{y_0}^2\,=\,-{\rm A}_{y_0}$, \quad $\lVert {\rm C}_{y_0}\rVert ^2=2$;
			\item[iv)] $
			\overline{Q}_{e,s}^T\,(\nabla [\overline{Q}_{e,s}\nabla\Theta .e_3]\,|\,0)\,[\nabla\Theta ]^{-1}\,\,=\,\,{\rm C}_{y_0} \mathcal{K}_{e,s}-{\rm B}_{y_0};
			$
			\item[v)] ${\rm C}_{y_0} \mathcal{K}_{e,s} {\rm A}_{y_0}\,\,=\,\,{\rm C}_{y_0} \mathcal{K}_{e,s} $,\quad  $\mathcal{E}_{m,s} {\rm A}_{y_0}\,\,=\,\,\mathcal{E}_{m,s} $.
		\end{itemize}
	\end{remark}
	Further, in view of $  {\rm I}_{y_0}^{-1} {\rm II}_{y_0}= {\rm L}_{y_0} $ and considering also the third fundamental form defined by ${\rm III}_{y_0}=  {\rm II}_{y_0} {\rm L}_{y_0}$, we note the relations
	\begin{equation}\label{eq2}
	[\nabla\Theta ]^{-1}\;{\rm B}_{y_0} = {\rm L}_{y_0}^\flat \; [\nabla\Theta ]^{-1}\qquad \mbox{and}\qquad
	{\rm B}^2_{y_0} = [\nabla\Theta ]^{-T}\; {\rm III}_{y_0}^\flat \; [\nabla\Theta ]^{-1}.
	\end{equation}
\subsection{The classical nonlinear  Koiter shell model in Cartesian matrix notation}\label{AppendixKoiter}\setcounter{equation}{0}
In this subsection, we consider  the  variational problem for the geometrically nonlinear Koiter energy for a nonlinear elastic shell \cite[page 147]{Ciarlet2Diff-Geo2005} and we rewrite it in matrix format. The problem written in tensor format  \cite[page 147]{Ciarlet2Diff-Geo2005},  \cite[Eq. (1) and Eq. (101)]{Steigmann13} is to find a deformation of the midsurface
$m:\omega\subset\mathbb{R}^2\to\mathbb{R}^3$  minimizing on $\omega$:
\begin{equation}\label{Ap7}
\begin{array}{l}
\dd\frac{1}{2}\int_\omega \bigg\{h\bigl\langle   \mathbb{C}_{\rm shell}^{\rm iso}.\frac{1}{2}\big( {\rm I}_m-{\rm I}_{y_0}\big),\,\frac{1}{2}\big( {\rm I}_m-{\rm I}_{y_0}\big)\bigr\rangle  +\dd\frac{h^3}{12}\bigl\langle   \mathbb{C}_{\rm shell}^{\rm iso}.\big( {\rm II}_m-{\rm II}_{y_0}\big),\,\big( {\rm II}_m-{\rm II}_{y_0}\big)\bigr\rangle  \bigg\}\, {\rm det}\nabla\Theta \,\, {\rm d}a\quad  \mapsto\min,
\end{array}
\end{equation}
where the fourth order constitutive tensor $\mathbb{C}_{\rm shell}^{\rm iso}:{\rm Sym}(2)\to {\rm Sym}(2)$ for isotropic elastic shells in the Koiter model is given by \cite{Ciarlet00}
\begin{equation}\label{Ap1}
\mathbb{C}_{\rm shell}^{\rm iso}\,=\,
\Big[ \mu\big( a^{\alpha\gamma}a^{\beta\tau}+ a^{\alpha\tau}a^{\beta\gamma}\big)+\dfrac{2\,\lambda\,\mu}{ \,\lambda\,+ 2\, \mu} \, a^{\alpha\beta}a^{ \gamma\tau}\Big] e _\alpha\otimes e _\beta\otimes e _\gamma \otimes e _\tau\,.\end{equation}

Since our model is completely written  in matrix format, we also transform   the above classical minimization problem in matrix format. To this aim, 
let us remark that in \eqref{Ap7}, for a second order symmetric tensor $\,X \,=\,X_{\alpha\beta} e _\alpha\otimes e _\beta\,$,  we have 
\begin{align}\label{Ap2}
\bigl\langle   \mathbb{C}_{\rm shell}^{\rm iso}.X &,\,X \bigr\rangle   \,= \,
\Big[ \mu\,\big( a^{\alpha\gamma}a^{\beta\tau}+ a^{\alpha\tau}a^{\beta\gamma}\big)+\dfrac{2\,\mu\,\lambda\,}{2\,\mu+\lambda} \, a^{\alpha\beta}a^{ \gamma\tau}\Big]\,X_{\alpha\beta}X_{\gamma\tau} \\\notag
=& \,
\,\mu\big( a^{\alpha\gamma}a^{\beta\tau}X_{\alpha\beta}X_{\gamma\tau}+ a^{\alpha\tau}a^{\beta\gamma}X_{\alpha\beta}X_{\gamma\tau}\big)+\dfrac{2\,\mu\,\lambda\,}{2\,\mu+\lambda} \, \big(a^{\alpha\beta}X_{\alpha\beta}\big)\big(a^{ \gamma\tau}X_{\gamma\tau}  \big)  =
2\,\mu\big( a^{\alpha\gamma}a^{\beta\tau}X_{\alpha\beta}X_{\gamma\tau}\big)+\dfrac{2\,\mu\,\lambda\,}{2\,\mu+\lambda} \, \big(a^{\alpha\beta}X_{\alpha\beta}\big)^2.\notag
\end{align}
A little calculation shows
\begin{equation}\label{Ap4}
\begin{array}{l}
\lVert  [\nabla\Theta]^{-T} \,{X}^\flat \, [\nabla\Theta]^{-1}\rVert^2 \,=\, \lVert P^{-T}\widehat{X}\, P^{-1}\rVert^2\,=\, \lVert (a ^i\otimes e _i)\, ( X_{\alpha\beta} \, e _\alpha\otimes e _\beta)\, ( e _j\otimes a ^j)\rVert^2
=\rVert    X_{\alpha\beta} \,a ^\alpha \otimes a ^\beta\rVert^2  \vspace{6pt} \\
\qquad\qquad\,=\, \mathrm{tr} \Big[\big(X_{\alpha\beta} \,a ^\alpha \otimes a ^\beta\big)\, \big(X_{\gamma\delta} \,a ^\gamma \otimes a ^\delta\big)^T\Big]
\,=\, \mathrm{tr} \Big[X_{\alpha\beta} X_{\delta\gamma} a^{\beta\gamma} \,\big(a ^\alpha \otimes a ^\delta\big)\Big]\, \,=\, \, X_{\alpha\beta} X_{\gamma\delta}\, a^{\beta\gamma} a^{\alpha\delta}
\,=\, \,a^{\alpha\gamma} a^{\beta\tau}   X_{\alpha\beta} X_{\gamma\tau}\;,
\end{array}
\end{equation}
where $ P\,=\, \nabla\Theta \, $ ,
and similarly
\begin{align}\label{Ap5}
\mathrm{tr} \Big[ [\nabla\Theta]^{-T} \,{X}^\flat \, [\nabla\Theta]^{-1}\Big] &=\, \mathrm{tr} \Big[P^{-T}{X}^\flat\, P^{-1}\Big]\,=\, \mathrm{tr} \Big[(a ^i\otimes e _i)\, ( X_{\alpha\beta} \, e _\alpha\otimes e _\beta)\, ( e _j\otimes a ^j)\Big]
 \vspace{6pt} \\
&=\,\mathrm{tr} \Big[ X_{\alpha\beta} \,a ^\alpha \otimes a ^\beta\Big] \,=\, X_{\alpha\beta} \,\bigl\langle  a ^\alpha, a ^\beta\bigr\rangle   
\,=\, a^{\alpha\beta}\,X_{\alpha\beta} \,.\notag
\end{align}
If we substitute \eqref{Ap4} and \eqref{Ap5} into \eqref{Ap2},  we obtain
\begin{equation}\label{Ap6}
\begin{array}{l}
\bigl\langle   \mathbb{C}_{\rm shell}^{\rm iso}.X ,\,X \bigr\rangle   \,= \,
2\,\mu\rVert    [\nabla\Theta]^{-T} \,{X}^\flat \, [\nabla\Theta]^{-1}\rVert^2 +\dfrac{2\,\lambda\,\mu}{ \lambda+ 2 \,\mu} \, \mathrm{tr} \Big[ [\nabla\Theta]^{-T} \,{X}^\flat \, [\nabla\Theta]^{-1}\Big]^2,
\end{array}
\end{equation}
which holds for any symmetric tensor $\,X \,=\,X_{\alpha\beta} e _\alpha\otimes e _\beta\,$.

Writing the equation \eqref{Ap6} for the symmetric matrix $\,X\,=\, \frac{1}{2}({\rm I}_m-{\rm I}_{y_0})\,$ and respectively the matrix $\,X\,=\, {\rm II}_m-{\rm II}_{y_0}\,$, then we obtain the following relations
\begin{align}\label{Ap7A}
\bigl\langle   \mathbb{C}_{\rm shell}^{\rm iso}.\frac{1}{2}\big( {\rm I}_m-{\rm I}_{y_0}\big),\,\frac{1}{2}\big( {\rm I}_m-{\rm I}_{y_0}\big)\bigr\rangle   \,=& \,
2\,\mu\,\rVert    [\nabla\Theta]^{-T} \,\frac{1}{2}\big({\rm I}_m^\flat-{\rm I}_{y_0}^\flat\big) \, [\nabla\Theta]^{-1}\rVert^2  +\dfrac{2\,\mu\,\lambda\,}{2\,\mu+\lambda} \, \mathrm{tr} \Big[ [\nabla\Theta]^{-T} \,\frac{1}{2}\,\big({\rm I}_m^\flat-{\rm I}_{y_0}^\flat\big) \, [\nabla\Theta]^{-1}\Big]^2, \vspace{6pt} \\
\bigl\langle   \mathbb{C}_{\rm shell}^{\rm iso}.\big( {\rm II}_m-{\rm II}_{y_0}\big),\,\big( {\rm II}_m-{\rm II}_{y_0}\big)\bigr\rangle  \, =& \,
2\,\mu\,\rVert    [\nabla\Theta]^{-T} \,\big({\rm II}_m^\flat-{\rm II}_{y_0}^\flat\big) \, [\nabla\Theta]^{-1}\rVert^2  +\dfrac{2\,\mu\,\lambda\,}{2\,\mu+\lambda} \, \mathrm{tr} \Big[ [\nabla\Theta]^{-T} \,\big({\rm II}_m^\flat-{\rm II}_{y_0}^\flat\big) \, [\nabla\Theta]^{-1}\Big]^2.\notag
\end{align}

Hence, putting all together, in matrix format and for a nonlinear elastic shell, the  variational problem for the Koiter energy   is to find a deformation of the midsurface
$m:\omega\subset\mathbb{R}^2\to\mathbb{R}^3$  minimizing on $\omega$
\begin{equation}\label{Ap7matrixcon}
\begin{array}{l}
\dd\int_\omega \bigg\{h\,\bigg(
\mu\rVert    [\nabla\Theta]^{-T} \,\underbrace{\frac{1}{2}\big({\rm I}_m^\flat-{\rm I}_{y_0}^\flat\big)}_{\mathcal{G}_{\rm Koiter}} \, [\nabla\Theta]^{-1}\rVert^2  +\dfrac{\,\lambda\,\mu}{\lambda+2\,\mu} \, \mathrm{tr} \Big[ [\nabla\Theta]^{-T} \,\big({\rm I}_m^\flat-{\rm I}_{y_0}^\flat\big) \, [\nabla\Theta]^{-1}\Big]^2\bigg) \vspace{6pt} \\
\quad\quad+\dd\frac{h^3}{12}\bigg(
\mu\rVert    [\nabla\Theta]^{-T} \,\underbrace{\big({\rm II}_m^\flat-{\rm II}_{y_0}^\flat\big)}_{\mathcal{R}_{\rm Koiter}} \, [\nabla\Theta]^{-1}\rVert^2 +\dfrac{\,\lambda\,\mu}{\lambda+2\,\mu} \, \mathrm{tr} \Big[ [\nabla\Theta]^{-T} \,\big({\rm II}_m^\flat-{\rm II}_{y_0}^\flat\big) \, [\nabla\Theta]^{-1}\Big]^2\bigg)\bigg\}\,{\rm det}\nabla\Theta \,\, {\rm d}a.
\end{array}
\end{equation}

The main feature of the classical Koiter model is that it is just the sum of the correctly identified membrane term and bending terms (under inextensional deformation).

\subsection{Thickness versus invertibility and coercivity}
\subsubsection{Invertibility conditions for the parametrized initial surface $\Theta$}\label{invertAppendix}
We note that $\det\nabla\Theta(x_3)= 1-2\, H\,x_3+K\, x_3^2=(1-\kappa_1\,x_3)(1-\kappa_2\, x_3)>0$ $\forall \, x_3\in \left[-\nicefrac{h}{2},\nicefrac{h}{2}\right]$ if and only if $1-\kappa_1\,x_3$ and $1-\kappa_2\, x_3$ have the same sign. However,  $1-\kappa_1\,x_3$ cannot be negative, since for $\kappa_1<0$  this will imply that $1<\kappa_1\,x_3$ $\forall \, x_3\in \left[-\nicefrac{h}{2},\nicefrac{h}{2}\right]$ which is not true if $x_3>0$, while for $\kappa_1\geq 0$  this will imply that $1<\kappa_1\,x_3$ $\forall \, x_3\in \left[-\nicefrac{h}{2},\nicefrac{h}{2}\right]$ which is not true if $x_3<0$. Therefore, $1-2\, H\,x_3+K\, x_3^2>0$ $\forall \, x_3\in \left[-\nicefrac{h}{2},\nicefrac{h}{2}\right]$ if and only if $1>\kappa_1\,x_3$ and $1>\kappa_2\, x_3$ $\forall \, x_3\in \left[-\nicefrac{h}{2},\nicefrac{h}{2}\right]$. These conditions are equivalent with  $|\kappa_1|\, \nicefrac{h}{2}<1$ and $|\kappa_2|\, \nicefrac{h}{2}<1$, i.e., ~equivalent with  \eqref{ch5in}.

\label{Appendixrelaxh}
\subsubsection{Coercivity for the $O(h^5)$ model}\label{coerh5Appendix}

The decisive point in the proof of the existence where the condition on the thickness is used is only in the proof of the coercivity of the internal energy density. Therefore, in this appendix, we extend the result regarding the coercivity of the internal energy to the following result:
\begin{proposition}\label{propcoerh5} {\rm [Coercivity in the theory including terms up to order $O(h^5)$]} For sufficiently small values of the thickness $h$ such that  
	\begin{align}\label{rcondh5}
h\max\{\sup_{x\in\omega}|\kappa_1|, \sup_{x\in\omega}|\kappa_2|\}<\alpha \qquad \text{with}\qquad  \alpha<\sqrt{\frac{2}{3}(29-\sqrt{761})}\simeq 0.97083
\end{align} 
	and for constitutive coefficients  satisfying  $\mu>0, \,\mu_{\rm c}>0$, $2\,\lambda+\mu> 0$, $b_1>0$, $b_2>0$ and $b_3>0$,   the  energy density
	\begin{align}W(\mathcal{E}_{m,s}, \mathcal{K}_{e,s})=W_{\mathrm{memb}}\big(  \mathcal{E}_{m,s} \big)+W_{\mathrm{memb,bend}}\big(  \mathcal{E}_{m,s} ,\,  \mathcal{K}_{e,s} \big)+W_{\mathrm{bend,curv}}\big(  \mathcal{K}_{e,s}    \big)
	\end{align}
	is coercive in the sense that  there exists a constant   $a_1^+>0$  such that
	$	W(\mathcal{E}_{m,s}, \mathcal{K}_{e,s})\,\geq\, a_1^+\, \big( \lVert \mathcal{E}_{m,s}\rVert ^2 + \lVert \mathcal{K}_{e,s}\rVert ^2\,\big)$,
	where
	$a_1^+$ depends on the constitutive coefficients.
\end{proposition}
\begin{proof}
	From the assumptions
	$
	h\, |\kappa_1|<\alpha,$ $ \  h\, |\kappa_2|<\alpha,
	$
	it follows that
	\begin{align}\label{condition}
	h^2|K|=h^2\, |\kappa_1|\,|\kappa_2|<\alpha^2\qquad \text{and}\qquad 2\,h\, |H|=h\, |\kappa_1+\kappa_2|<2\, \alpha.
	\end{align}
	Therefore, for $\alpha<\sqrt{\frac{20}{3}}$ it follows
	$
	h-{\rm K}\,\frac{h^3}{12}>	h-\,\frac{h}{12}\alpha^2>0 \ \ \textrm{and}  \   \ \frac{h^3}{12}\,-\frac{h^3}{80}\alpha^2>0.
	$ On the other hand, from \cite[Proposition 3.1.]{GhibaNeffPartII} we have
	\begin{align}
	W(\mathcal{E}_{m,s}, \mathcal{K}_{e,s})
	\geq\,&
	\Big(h-\dfrac{h}{12}\delta+{\rm K}\,\dfrac{h^3}{12}-
	\dfrac{h^2}{6}\varepsilon\,|\mathrm{H}|\Big)\,{W}_{\mathrm{shell}}  \big(  \mathcal{E}_{m,s}\big)+\Big(\dfrac{h^3}{12}\,-{\rm K}\,\dfrac{h^5}{80}-\dfrac{h^4}{6\, \varepsilon}\,\,|\mathrm{H}|\Big)\, {W}_{\mathrm{shell}}  \big(
	\mathcal{E}_{m,s}{\rm B}_{y_0}+{\rm C}_{y_0}\, \mathcal{K}_{e,s}  \big)\\&+ \Big(\,\dfrac{h^5}{80}-
	\dfrac{h^5}{12\, \delta}\Big)\,\,
	W_{\mathrm{shell}} \big((  \mathcal{E}_{m,s} \, {\rm B}_{y_0} +  {\rm C}_{y_0} \mathcal{K}_{e,s} )   {\rm B}_{y_0} \,\big) +\Big(h-{\rm K}\,\dfrac{h^3}{12}\Big)\,
	W_{\mathrm{curv}}\big(  \mathcal{K}_{e,s} \big)\qquad \  \forall\, \varepsilon>0\  \text{ and } \ \  \delta>0\notag.
	\end{align}
Using the inequalities \eqref{condition}, we deduce 
		\begin{align}\label{B5}
	W(\mathcal{E}_{m,s}, \mathcal{K}_{e,s})
	\geq\,&
	\dfrac{h}{12}\Big(12-\delta-\alpha^2-
	2\,\varepsilon\,\alpha\Big)\,{W}_{\mathrm{shell}}  \big(  \mathcal{E}_{m,s}\big)+\dfrac{h^3}{12}\,\Big(1-\dfrac{3}{20}\,\alpha^2-\dfrac{1}{3\, \varepsilon}\,\alpha\Big)\, {W}_{\mathrm{shell}}  \big(
	\mathcal{E}_{m,s}{\rm B}_{y_0}+{\rm C}_{y_0}\, \mathcal{K}_{e,s}  \big)\\&+ \dfrac{h^5}{80}\,\Big(1-
	\dfrac{20 }{3\, \delta}\Big)\,\,
	W_{\mathrm{shell}} \big((  \mathcal{E}_{m,s} \, {\rm B}_{y_0} +  {\rm C}_{y_0} \mathcal{K}_{e,s} )   {\rm B}_{y_0} \,\big)+\Big(h-{\rm K}\,\dfrac{h^3}{12}\Big)\,
	W_{\mathrm{curv}}\big(  \mathcal{K}_{e,s} \big)\qquad \  \forall\, \varepsilon>0\  \text{ and } \ \  \delta>0\notag.
	\end{align}
Let us remark for $0<\alpha<\sqrt{\frac{20}{3}}$, if there exist $\delta>\frac{20}{3}$ and	 $\varepsilon>\frac{40\, \alpha}{20-3\, \alpha^2}$ 
such that
\begin{align}\label{conditionaed}
12-\delta-\alpha^2-
2\,\varepsilon\,\alpha>0 \ \ \Leftrightarrow\ \ 12-\alpha^2>\delta+
2\,\varepsilon\,\alpha>\frac{20}{3}+
2\,\frac{40\,\alpha^2}{20-3\, \alpha^2}\,
\end{align}
then  $\alpha<\sqrt{\frac{2}{3}(29-\sqrt{761})}\simeq 0.97083$. Since the map $(\delta, \varepsilon)\mapsto \delta+
2\,\varepsilon\,\alpha$ is linear, it increases in gradient direction and hence, the condition $\alpha<\sqrt{\frac{2}{3}(29-\sqrt{761})}\simeq 0.97083$ is also sufficient for the existence of $\delta>\frac{20}{3}$ and	 $\varepsilon>\frac{40\, \alpha}{20-3\, \alpha^2}$  such that \eqref{conditionaed} is satisfied.
 
	\begin{align}
	W(\mathcal{E}_{m,s}, \mathcal{K}_{e,s})
	\geq\,&h\,\Big[\dfrac{1}{3}-{\rm K}\,\dfrac{h^2}{12}-
	\dfrac{h}{3}\,|\mathrm{H}|\Big]\,  {W}_{\mathrm{shell}}  \big(  \mathcal{E}_{m,s}\big)+\dfrac{h^3}{12}\Big(1\,-|{\rm K}|\,\dfrac{12\,h^2}{80}-h\,|\mathrm{H}|\Big)\, {W}_{\mathrm{shell}}  \big(
	\mathcal{E}_{m,s}{\rm B}_{y_0}+{\rm C}_{y_0}\, \mathcal{K}_{e,s}  \big)\notag\\&+\Big(h-|{\rm K}|\,\dfrac{h^3}{12}\Big)\,
	W_{\mathrm{curv}}\big(  \mathcal{K}_{e,s} \big).
	\end{align}
	In view of \eqref{B5} and \eqref{pozitivdef}, we see that there exist positive constants $b_1^+, b_2^+,b_3^+>0$ such that
	\begin{align}
	W(\mathcal{E}_{m,s}, \mathcal{K}_{e,s})
	\geq\,&\dfrac{h}{12}\,b_1^+\,  {W}_{\mathrm{shell}}  \big(  \mathcal{E}_{m,s}\big)+\dfrac{h^3}{12}\,b_2^+\, {W}_{\mathrm{shell}}  \big(
	\mathcal{E}_{m,s}{\rm B}_{y_0}+{\rm C}_{y_0}\, \mathcal{K}_{e,s}  \big)+h\,b_3^+\,
	W_{\mathrm{curv}}\big(  \mathcal{K}_{e,s} \big)\notag\\
	\geq \,&\dfrac{h}{12}\,b_1^+\, \min\{c_1^+,\mu_{\rm c}\}\, \lVert   \mathcal{E}_{m,s}\rVert ^2+\dfrac{h}{12}\,b_2^+\,\min\{c_1^+,\mu_{\rm c}\}\, \lVert
	\mathcal{E}_{m,s}{\rm B}_{y_0}+{\rm C}_{y_0}\, \mathcal{K}_{e,s}  \rVert ^2+h\,b_3^+\,  { c_2^+}
	\lVert   \mathcal{K}_{e,s} \rVert ^2.
	\end{align}
	The desired constant  $a_1^+$ from the conclusion can be chosen as $a_1^+=\min\big\{\dfrac{h}{12}\,b_1^+\, \min\{c_1^+,\mu_{\rm c}\},\dfrac{h}{12}\,b_2^+\,\min\{c_1^+,\mu_{\rm c}\}, h\,b_3^+\,  { c_2^+} \big\}$.
\end{proof}

\subsubsection{Coercivity for the $O(h^3)$ model}\label{coerh3Appendix}

In this subsection we investigate if the  conditions which assure the existence of the solution may be relaxed in the Cosserat shell model up to $O(h^3)$, too. In fact, as in the previous subsection, it is enough to prove some new coercivity results under weakened conditions on the thickness. 
We recall that in the Cosserat shell model up to $O(h^3)$
the shell energy density $W^{(h^3)}(\mathcal{E}_{m,s}, \mathcal{K}_{e,s})$ is given by \eqref{h3energy}.
\begin{proposition}\label{coerh3r}{\rm [The first coercivity result in the theory including terms up to order $O(h^3)$]}  For sufficiently small values of the thickness $h$ such that  	
	\begin{align}\label{fcond}
	h\max\{\sup_{x\in\omega}|\kappa_1|, \sup_{x\in\omega}|\kappa_2|\}<\alpha\qquad \text{and}\qquad 	h^2<\frac{(5-2\sqrt{6})(\alpha^2-12)^2}{4\, \alpha^2}\frac{ {c_2^+}}{\max\{C_1^+,\mu_{\rm c}\}}\qquad \text{with} \quad 0<\alpha<2\sqrt{3}
	\end{align}
	and for constitutive coefficients  satisfying the constitutive coefficients are  such that $\mu>0, \,\mu_{\rm c}>0$, $2\,\lambda+\mu> 0$, $b_1>0$, $b_2>0$ and $b_3>0$ and where $c_2^+$  denotes the smallest eigenvalue  of
	$
	W_{\mathrm{curv}}(  S ),
	$
	and $ C_1^+>0$ denotes  the largest eigenvalues of the quadratic form $W_{\mathrm{shell}}^\infty(  S)$,
	the total energy density $W^{(h^3)}(\mathcal{E}_{m,s}, \mathcal{K}_{e,s})$
	is coercive, in the sense that  there exists   a constant $a_1^+>0$ such that \linebreak
	$	W^{(h^3)}(\mathcal{E}_{m,s}, \mathcal{K}_{e,s})\,\geq\, a_1^+\, \big(  \lVert \mathcal{E}_{m,s}\rVert ^2 + \lVert \mathcal{K}_{e,s}\rVert ^2\,\big) ,
	$ where
	$a_1^+$ depends on the constitutive coefficients.
\end{proposition}
\begin{proof}Similarly as in the proof presented in \cite[Proposition 4.1]{GhibaNeffPartII}, we obtain the estimate
	\begin{align}
	W^{(h^3)}(\mathcal{E}_{m,s}, \mathcal{K}_{e,s})	\geq \, &\Big(h-{\rm K}\,\dfrac{h^3}{12}-
	\dfrac{1}{12}\, \varepsilon \, h-
	\dfrac{1}{6}\,\delta \,|{\rm H}|\,  h^2\Big)\,{W}_{\mathrm{shell}}  \big(  \mathcal{E}_{m,s}\big)-\dfrac{1}{12\,\varepsilon}\,  h^5\,{W}_{\mathrm{shell}}  \big(
	{\rm C}_{y_0}\, \mathcal{K}_{e,s} {\rm B}_{y_0} \big)\notag\\&-\dfrac{1}{6\,\delta} \,|{\rm H}|\,  h^4\,{W}_{\mathrm{shell}}  \big(
	{\rm C}_{y_0}\, \mathcal{K}_{e,s}  \big)+\dfrac{h^3}{12}\,
	W_{\mathrm{shell}}  \big(   \mathcal{E}_{m,s} \, {\rm B}_{y_0} +   {\rm C}_{y_0} \mathcal{K}_{e,s} \big)
	\notag\\&+\Big(h-{\rm K}\,\dfrac{h^3}{12}\Big)\,
	W_{\mathrm{curv}}\big(  \mathcal{K}_{e,s} \big) +  \dfrac{h^3}{12}\,
	W_{\mathrm{curv}}\big(  \mathcal{K}_{e,s}   {\rm B}_{y_0} \,  \big)
	\quad \forall\, \varepsilon>0\  \text{and} \ \delta>0,
	\end{align}
which after imposing the conditions $-h^2 |K|>-\alpha^2$ and $-h\, |H|>-\alpha$,	using  \eqref{condition}, \eqref{pozitivdef} and 	since the Frobenius norm is sub-multiplicative and 
$\lVert C_{y_0}\rVert ^2=2
$, leads to
	\begin{align}
	W^{(h^3)}(\mathcal{E}_{m,s}, \mathcal{K}_{e,s})\geq \, &
	\dfrac{h}{12}	\Big(12-\alpha^2-\varepsilon-2\,\alpha\, \delta
\Big)\, {W}_{\mathrm{shell}}  \big(  \mathcal{E}_{m,s}\big)+\dfrac{h^3}{12}\, {W}_{\mathrm{shell}}  \big(
	\mathcal{E}_{m,s}{\rm B}_{y_0}+{\rm C}_{y_0}\, \mathcal{K}_{e,s}  \big)\notag\\&-\dfrac{1}{6\,\delta} \, \alpha\,h^3\,{W}_{\mathrm{shell}}  \big(
	{\rm C}_{y_0}\, \mathcal{K}_{e,s}  \big)-\dfrac{1}{12\,\varepsilon}\,  h^5\,{W}_{\mathrm{shell}}  \big(
	{\rm C}_{y_0}\, \mathcal{K}_{e,s} {\rm B}_{y_0} \big)
	+\dfrac{h}{12}(12-\alpha^2)\,
	W_{\mathrm{curv}}\big(  \mathcal{K}_{e,s} \big) +  \dfrac{h^3}{12}\,
	W_{\mathrm{curv}}\big(  \mathcal{K}_{e,s}   {\rm B}_{y_0} \,  \big)\notag\\
	\geq&\,
	\dfrac{h}{12}	\Big(12-\alpha^2-\varepsilon-2\,\alpha\, \delta
	\Big)\,  \min\{c_1^+,\mu_{\rm c}\} \lVert \mathcal{E}_{m,s}\rVert ^2+\dfrac{h^3}{12}\, \min\{c_1^+,\mu_{\rm c}\}  \lVert
	\mathcal{E}_{m,s}{\rm B}_{y_0}+{\rm C}_{y_0}\mathcal{K}_{e,s}  \rVert ^2\\&-\dfrac{1}{3\delta} \alpha h^3\max\{C_1^+,\mu_{\rm c}\} \lVert \mathcal{K}_{e,s}\rVert ^2-\dfrac{1}{6\,\varepsilon} h^5\max\{C_1^+,\mu_{\rm c}\}  \lVert \mathcal{K}_{e,s}{\rm B}_{y_0} \rVert ^2+\dfrac{h}{12}(12-\alpha^2)c_2^+\lVert  \mathcal{K}_{e,s}\rVert ^2 +  \dfrac{h^3}{12}
c_2^+\lVert    \mathcal{K}_{e,s}   {\rm B}_{y_0}\rVert ^2\notag
	\end{align}
for all $\alpha, \delta, \varepsilon>0$ such that ~$12-\alpha^2-\varepsilon-2\,\alpha\, \delta>0$ and $12-\alpha^2>0$.~ Since $\lVert B_{y_0}\rVert ^2=4\, {\rm H}^2-2\,{\rm K}$, it follows
 \begin{align}
 W^{(h^3)}(\mathcal{E}_{m,s}, \mathcal{K}_{e,s})\geq \, &\,
 \dfrac{h}{12}	\Big(12-\alpha^2-\varepsilon-2\,\alpha\, \delta
 \Big)\,  \min\{c_1^+,\mu_{\rm c}\}  \lVert \mathcal{E}_{m,s}\rVert ^2+\dfrac{h^3}{12}\, \min\{c_1^+,\mu_{\rm c}\}  \lVert
 \mathcal{E}_{m,s}{\rm B}_{y_0}+{\rm C}_{y_0}\, \mathcal{K}_{e,s}  \rVert ^2\notag\\&+\dfrac{h}{12}\Big[(12-\alpha^2)\,\,c_2^+-\dfrac{4}{\delta} \,\alpha\, \max\{C_1^+,\mu_{\rm c}\} \,h^2-\dfrac{2}{\varepsilon}\,  h^4\,\,\max\{C_1^+,\mu_{\rm c}\} \, (4\, {\rm H}^2-2\,{\rm K})\Big]\lVert \mathcal{K}_{e,s} \rVert ^2 +  \dfrac{h^3}{12}\,
 c_2^+\lVert    \mathcal{K}_{e,s}   {\rm B}_{y_0}\rVert ^2\notag
 \end{align}
Using again that $h$ is small,
	we obtain 
$
	-h^2\,(4\, {\rm H}^2-2\,{\rm K})\geq -h^2\,(4\, {\rm H}^2+2\,|{\rm K}|)\geq -6\,\alpha^2
$
	and
 \begin{align}
W^{(h^3)}(\mathcal{E}_{m,s}, \mathcal{K}_{e,s})\geq \, &\,
\dfrac{h}{12}	\Big(12-\alpha^2-\varepsilon-2\,\alpha\, \delta
\Big)\,  \min\{c_1^+,\mu_{\rm c}\}  \lVert \mathcal{E}_{m,s}\rVert ^2\\&+\dfrac{h}{12}\Big[(12-\alpha^2)\,\,c_2^+-\dfrac{4}{\delta} \,\alpha\, \max\{C_1^+,\mu_{\rm c}\} \,h^2-\dfrac{2}{\varepsilon}\,  h^2\,\,\max\{C_1^+,\mu_{\rm c}\} \, 6\,\alpha^2\Big]\lVert \mathcal{K}_{e,s} \rVert ^2.\notag
\end{align}
	We consider $\delta=\gamma \, \varepsilon$ and we choose $\epsilon>0$ and $\gamma>0$ such that
	\begin{align}
	\frac{12-\alpha^2}{1+2\,\alpha\gamma}>\varepsilon >4\frac{\alpha+3\,   \alpha^2\gamma}{\gamma(12-\alpha^2)}\,\frac{\max\{C_1^+,\mu_{\rm c}\} }{c_2^+}\,h^2.
	\end{align}
	This choice of the variable $\varepsilon>0$ is possible  if and
	only if
	$
\frac{(12-\alpha^2)^2\gamma}{4\,(a+2\,\alpha\gamma)\,(\alpha+3\,   \alpha^2\gamma)}>\frac{h^2\,\max\{C_1^+,\mu_{\rm c}\} }{c_2^+}.
	$
	At this point we use that
	$$
	\max_{\gamma>0}\frac{(12-\alpha^2)^2\gamma}{4\,(a+2\,\alpha\gamma)\,(\alpha+3\,   \alpha^2\gamma)}=\frac{(5-2\sqrt{6})(\alpha^2-12)^2}{4\, \alpha^2},
	$$
	and this maximum value is attained for $\gamma=\frac{1}{\sqrt{6}\,a}$. Hence, we arrive at the following condition on the thickness $h$:
	\begin{align}
	h^2<\frac{(5-2\sqrt{6})(\alpha^2-12)^2}{4\, \alpha^2}\frac{ {c_2^+}}{\max\{C_1^+,\mu_{\rm c}\} },
	\end{align}
	which proves the coercivity if the condition from the hypothesis  is satisfied.
\end{proof}	
The condition \eqref{fcond} on the thickness does not represent a relaxation of the condition imposed in \cite[Proposition 4.1]{GhibaNeffPartII}, i.e., not in the sense of the relaxed condition \eqref{rcondh5} which is found for the $O(h^5)$ Cosserat shell model. Indeed, since the map $\alpha\mapsto \frac{(5-2\sqrt{6})(\alpha^2-12)^2}{4\, \alpha^2}\frac{ {c_2^+}}{\max\{C_1^+,\mu_{\rm c}\} }$ is monotone decreasing on $[0,2]$ (the interval of the values of the parameter $\alpha$ for which the construction of the model has sense), a large value for $\alpha$ will relax the first condition \eqref{fcond}$_1$ while the other condition \eqref{fcond}$_2$ on the thickness will become more restrictive. Since the second condition \eqref{fcond}$_2$ is expressed in terms of all  constitutive parameters, through $c_2^+$ and $\max\{C_1^+,\mu_{\rm c}\} $, while the first condition \eqref{fcond}$_1$ depends on the curvatures of the referential configuration, the largest value of the parameter $\alpha$ in the coercivity result would be chosen as the best compromise between the conditions \eqref{fcond}$_1$  and \eqref{fcond}$_2$. In conclusion, in comparison to the conditions imposed in \cite[Proposition 4.1]{GhibaNeffPartII}, for a specific material and a specific referential configuration the new condition \eqref{fcond} would offer a largest interval of values for the upper bound of the thickness.

We also note that in \cite[Proposition 4.1]{GhibaNeffPartII}, because the condition of the form \eqref{fcond}$_2$ from the hypothesis of Proposition  \ref{coerh3r} depends on the length scale $L_c$, we have proved another coercivity result which avoids this aspect:
	\begin{proposition}\label{coerh3r2}{\rm [The second coercivity result in the theory including terms up to order $O(h^3)$]}  For sufficiently small values of the thickness $h$ such that  	
		\begin{align}
		h\max\{\sup_{x\in\omega}|\kappa_1|, \sup_{x\in\omega}|\kappa_2|\}<\frac{1}{a}\qquad \qquad \text{with} \quad a>\max\Big\{1 + \frac{\sqrt{2}}{2},\frac{1+\sqrt{1+3\frac{\max\{C_1^+,\mu_{\rm c}\} }{\min\{c_1^+,\mu_{\rm c}\} }}}{2}\Big\},
		\end{align}
		and for constitutive coefficients  satisfying the constitutive coefficients are  such that $\mu>0, \,\mu_{\rm c}>0$, $2\,\lambda+\mu> 0$, $b_1>0$, $b_2>0$ and $b_3>0$ and let  $c_1^+$ and $ \max\{C_1^+,\mu_{\rm c}\} >0$ denote the smallest and the largest eigenvalues of the quadratic form $W_{\mathrm{shell}}^\infty(  S)$,
		the total energy density $W^{(h^3)}(\mathcal{E}_{m,s}, \mathcal{K}_{e,s})$
		is coercive, in the sense that  there exists   a constant $a_1^+>0$ such that \linebreak
		$W^{(h^3)}(\mathcal{E}_{m,s}, \mathcal{K}_{e,s})\,\geq\, a_1^+\, \big(  \lVert \mathcal{E}_{m,s}\rVert ^2 + \lVert \mathcal{K}_{e,s}\rVert ^2\,\big)$, where
		$a_1^+$ depends on the constitutive coefficients.
	\end{proposition}

\subsection{Acharya's bending tensor in Cartesian matrix notation}\label{AppAcharya}

For comparison purpose, let us mark all the  vectors and tensors   defined by Acharya \cite[Page 5520]{acharya2000nonlinear} with the subscript  $A$ \linebreak (e.g., $ \boldsymbol f_A $, $ \boldsymbol b^*_A $, $ (\boldsymbol E_\alpha)_A $). We  express them using our notations as follows:
\begin{align}
\boldsymbol X_A =   y_0\,,\qquad \boldsymbol N_A = n_0\,,\qquad \boldsymbol x_A = m\,,\qquad \boldsymbol n_A= n\,,\qquad (\boldsymbol E_\alpha)_A = a_\alpha=\partial_{x_\alpha} y,\qquad (\boldsymbol E^\alpha)_A = a^\alpha\;\;\mbox{etc.}
\end{align}
Let us denote by $ [ \boldsymbol T_A] $ the $ 3\times 3 $ matrix of the components in the basis $ \{ e_i\otimes e_j\} $ for any tensor $ \boldsymbol T $ defined by Acharya \cite[Page 5519]{acharya2000nonlinear}.
Then, we have
\begin{equation}\label{eq1}
	[ \boldsymbol f_A ] = \Big[\Big( \frac{\partial \boldsymbol x}{\partial \boldsymbol X} \Big)_A\, \Big] 
	= [ m_{,\alpha} \otimes a^\alpha] = [ (m_{,\alpha} \otimes e_\alpha)(  e_i \otimes a^i) ] = (\nabla m | 0)\, [ \nabla_x \Theta ]^{-1}
\end{equation}
and
\begin{equation}\label{eq2}
	[ \boldsymbol B_A ]  = \Big[\Big( \frac{\partial \boldsymbol N}{\partial \boldsymbol X} \Big)_A\, \Big]  = [ n_{0,\alpha} \otimes a^\alpha] = [ (n_{0,\alpha} \otimes e_\alpha)(  e_i \otimes a^i) ] = (\nabla n_0 | 0)\, [ \nabla_x \Theta ]^{-1},
\end{equation}
so that
\begin{equation}\label{eq3}
	[ \boldsymbol B_A ]  = - {\rm B}_{y_0} =  
	-  [ \nabla_x \Theta ]^{-T}\; {\rm II}_{y_0}^\flat \; [ \nabla_x \Theta ]^{-1}.
\end{equation}
Furthermore, we have
\[ 
[ \boldsymbol b^*_A ]  = \Big[\boldsymbol f_A^T\; \Big( \frac{\partial \boldsymbol n}{\partial \boldsymbol x} \Big)_A\; \boldsymbol f_A\, \Big]  =
\Big[\boldsymbol f_A^T\; \Big( \frac{\partial \boldsymbol n}{\partial \boldsymbol X} \Big)_A \, \Big]  = 
[\boldsymbol f_A ]^T\;\Big[ \Big( \frac{\partial \boldsymbol n}{\partial \boldsymbol X} \Big)_A \, \Big]  
\]
and using \eqref{eq1} we get
\begin{equation}\label{eq4}
	[ \boldsymbol b^*_A ]  =   
	[ \nabla_x \Theta ]^{-T}\; (\nabla m | 0)^T\, (\nabla n | 0)\,  [ \nabla_x \Theta ]^{-1}
	=  - [ \nabla_x \Theta ]^{-T}\; {\rm II}_{m}^\flat \; [ \nabla_x \Theta ]^{-1}.
\end{equation}
By virtue of \eqref{eq3} and \eqref{eq4}, the classical bending strain measure can be written as \cite[Eq. 3]{acharya2000nonlinear} 
\begin{equation}\label{eq5}
	[\boldsymbol K_A] = [ \boldsymbol b^*_A ] - [ \boldsymbol B_A ] 
	=  - [ \nabla_x \Theta ]^{-T}\; ({\rm II}_{m} - {\rm II}_{y_0} )^\flat \; [ \nabla_x \Theta ]^{-1}.
\end{equation}
For the tensor $ \boldsymbol U_A $ we obtain using equation \eqref{eq3} the expression
\begin{equation}\label{eq6}
	\begin{array}{rl}
		[\boldsymbol U_A] & =  \sqrt{  [ \boldsymbol f_A^T\, \boldsymbol f_A\, ]} =
		\sqrt{  [ \boldsymbol f_A]^T\, [\boldsymbol f_A\, ]} =
		\sqrt{[ \nabla_x \Theta ]^{-T}\; (\nabla m | 0)^T\, (\nabla m | 0)\,  [ \nabla_x \Theta ]^{-1} }
		\vspace{6pt}
		= 
		\sqrt{[ \nabla_x \Theta ]^{-T}\; {\rm I}_{m}^\flat \;   [ \nabla_x \Theta ]^{-1} }\,.
	\end{array}
\end{equation}
Thus, the bending tensor defined by Acharya \cite[Eq. 8]{acharya2000nonlinear}  is given by
\[ 
[\widetilde{\boldsymbol K}_A] = [ \boldsymbol b^*_A ] - \Big[\boldsymbol U_A\;\Big( \frac{\partial \boldsymbol N}{\partial \boldsymbol X} \Big)_A\, \Big] = 
[ \boldsymbol b^*_A ] - [\boldsymbol U_A]\,[\boldsymbol B_A]\,.
\]
Inserting here the relations \eqref{eq3}, \eqref{eq4} and \eqref{eq6}, we obtain that
the bending tensor $ \widetilde{\boldsymbol K}_A $  defined in the paper by Acharya \cite[Eq. 7 and 8]{acharya2000nonlinear}   can be written in our matrix notation as follows:
\begin{equation}\label{eq7}
\widetilde{\mathcal{R}}_{\rm Acharya}:=	\tilde{\boldsymbol K}_A = - [ \nabla_x \Theta ]^{-T}\; {\rm II}_{m}^\flat \; [ \nabla_x \Theta ]^{-1} +
	\sqrt{[ \nabla_x \Theta ]^{-T}\; {\rm I}_{m}^\flat \; [ \nabla_x \Theta ]^{-1}}\;
	[ \nabla_x \Theta ]^{-T}\; {\rm II}_{y_0}^\flat \; [ \nabla_x \Theta ]^{-1}.
\end{equation}

Since $\id_2^\flat  {\rm II}_{m}^\flat={\rm II}_{m}^\flat$ and $\id_2^\flat [\nabla\Theta \,]^{T}\widehat{\rm I}_{m}=  {\rm I}_{m}^\flat$, the following relation between the nonlinear bending tensor $\widetilde{\mathcal{R}}_{\rm Acharya}$ introduced by Acharya and our nonlinear bending tensor $\mathcal{R}_{\infty}^\flat$ holds
\begin{align}\label{reAc}
\widetilde{\mathcal{R}}_{\rm Acharya}
\! =&\!- \!\!\left([\nabla\Theta \,]^{-T} \id_2^\flat [\nabla\Theta \,]^{T} [\nabla\Theta \,]^{-T} {\rm II}_{m}^\flat  [\nabla\Theta \,]^{-1}\!\! -\!\!
\sqrt{[\nabla\Theta \,]^{-T} \id_2^\flat [\nabla\Theta \,]^{T} [\nabla\Theta \,]^{-T}\widehat{\rm I}_{m} \; [\nabla\Theta \,]^{-1}}
[\nabla\Theta \,]^{-T} {\rm II}_{y_0}^\flat \, [\nabla\Theta \,]^{-1}\!\!\right) \notag\\
=&- \!\![\nabla\Theta \,]^{-T} \id_2^\flat [\nabla\Theta \,]^{T}\left( [\nabla\Theta \,]^{-T} {\rm II}_{m}^\flat  [\nabla\Theta \,]^{-1}\!\! -\!\!
\sqrt{ [\nabla\Theta \,]^{-T}\widehat{\rm I}_{m} \; [\nabla\Theta \,]^{-1}}
[\nabla\Theta \,]^{-T} {\rm II}_{y_0}^\flat \, [\nabla\Theta \,]^{-1}\!\!\right)
\\
=&- \!\![\nabla\Theta \,]^{-T} \id_2^\flat [\nabla\Theta \,]^{T}\sqrt{[\nabla\Theta ]^{-T}\,\widehat{\rm I}_m[\nabla\Theta ]^{-1}}[\nabla\Theta \,]^{-T}\mathcal{R}_{\infty}^\flat[\nabla\Theta \,]^{-1} \notag\\
=&- \sqrt{[\nabla\Theta ]^{-T}\,\id_2^\flat\widehat{\rm I}_m[\nabla\Theta ]^{-1}}[\nabla\Theta \,]^{-T}\mathcal{R}_{\infty}^\flat[\nabla\Theta \,]^{-1}
=- \sqrt{[\nabla\Theta ]^{-T}\,{\rm I}_m^\flat[\nabla\Theta ]^{-1}}[\nabla\Theta \,]^{-T}\mathcal{R}_{\infty}^\flat[\nabla\Theta \,]^{-1},\notag
\end{align}
where we have used $\sqrt{[\nabla\Theta \,]^{-T} \id_2^\flat [\nabla\Theta \,]^{T}}=[\nabla\Theta \,]^{-T} \id_2^\flat [\nabla\Theta \,]^{T}\Leftrightarrow ([\nabla\Theta \,]^{-T} \id_2^\flat [\nabla\Theta]^{T})^2=[\nabla\Theta \,]^{-T} \id_2^\flat [\nabla\Theta]^{T}[\nabla\Theta \,]^{-T} \id_2^\flat [\nabla\Theta]^{T}=[\nabla\Theta \,]^{-T} (\id_2^\flat)^2   [\nabla\Theta]^{T}=[\nabla\Theta \,]^{-T} \id_2^\flat [\nabla\Theta]^{T}$ and \eqref{eq:unserTensor}.

\subsection{Alternative representation of energy in terms of the new strain tensors}

We present in the following an alternative formulation of the unconstrained Cosserat-shell model from \eqref{minvarmc}.
In this section, for a matrix of the form $X\,=\,(X^S\, |\, 0)\; [\nabla\Theta \,]^{-1}\in \mathbb{R}^{3\times 3}$, we consider the matrix 
$
X^C \,=\,  [\nabla\Theta \,]^{T} X^S\in \mathbb{R}^{3\times 2}$, i.e., \begin{align}    X\,= \, [\nabla\Theta \,]^{-T} (X^C \,|\, 0)\; [\nabla\Theta \,]^{-1}.\end{align}

For the bilinear form $ W_{\mathrm{shell}}(  X,  Y) $ given in \eqref{quadraticforms} we have then the transformations
\begin{equation}\label{eq1}
W_{\mathrm{shell}}(  X,  Y)  =  W^S_{\mathrm{shell}}(  X^S,  Y^S)  =  W^C_{\mathrm{shell}}(  X^C,  Y^C) ,
\end{equation}
where, e.g.,
\begin{align} 
W^C_{\mathrm{shell}}(  X^C,  Y^C) \,=\,&  \mu\,\langle\, \mathrm{sym}\,  \big( [\nabla\Theta \,]^{-T} (X^C \,| \,0)\; [\nabla\Theta \,]^{-1}\big),\,\mathrm{sym}\,  \big([\nabla\Theta \,]^{-T} (Y^C \,| \,0)\; [\nabla\Theta \,]^{-1}\big)\,\rangle \vspace{6pt}\notag\\
& +    \mu_c\langle \,\mathrm{skew}\,  \big( [\nabla\Theta \,]^{-T} (X^C \,| \,0)\; [\nabla\Theta \,]^{-1}\big),\,\mathrm{skew}\,   \big([\nabla\Theta \,]^{-T} (Y^C \,| \,0)\; [\nabla\Theta \,]^{-1}\big)\,\rangle \vspace{6pt}\notag\\
& +  \,\dfrac{\lambda\,\mu}{\lambda+2\mu}\,\mathrm{tr}  ([\nabla\Theta \,]^{-T} (X^C \,| \,0)\; [\nabla\Theta \,]^{-1})\,\mathrm{tr}  ([\nabla\Theta \,]^{-T} (Y^C \,| \,0)\; [\nabla\Theta \,]^{-1}).
\end{align}
We note that the last relation can be written using the first fundamental form $ {\rm I}_{y_0} $ and its square root in the following alternative ways:
\begin{align} 
W^C_{\mathrm{shell}}(  X^C,  Y^C) \,=\,& \, \mu\,\langle\, \mathrm{sym}\,    (X^C \,| \,0),\,\mathrm{sym}\,  \big(\hat{\rm I}_{y_0}^{-1} (Y^C \,| \,0)\, \hat{\rm I}_{y_0}^{-1}\big)\,\rangle +    \mu_c\langle \,\mathrm{skew}\,  (X^C \,| \,0),\,\mathrm{skew}\,  \big(\hat{\rm I}_{y_0}^{-1} (Y^C \,| \,0)\, \hat{\rm I}_{y_0}^{-1}\big)\,\rangle \vspace{6pt}\notag\\
&   + \,\dfrac{\lambda\,\mu}{\lambda+2\mu}\,\mathrm{tr}  \big( (X^C \,| \,0)\hat{\rm I}_{y_0}^{-1}\big)\,\mathrm{tr}  \big( (Y^C \,| \,0)\, \hat{\rm I}_{y_0}^{-1}\big)
\vspace{10pt}\notag\\
\,=\,& \, \mu\,\langle\, \mathrm{sym}\,  \big( \hat{\rm I}_{y_0}^{-1/2} (X^C \,| \,0)\hat{\rm I}_{y_0}^{-1/2}\big) ,\,\mathrm{sym}\,  \big(\hat{\rm I}_{y_0}^{-1/2} (Y^C \,| \,0)\, \hat{\rm I}_{y_0}^{-1/2}\big)\,\rangle 
\vspace{6pt}\\
&   +     \mu_c\;\langle \,\mathrm{skew}\, \big(\hat{\rm I}_{y_0}^{-1/2} (X^C \,| \,0)\hat{\rm I}_{y_0}^{-1/2}\big) ,\,\mathrm{skew}\,  \big(\hat{\rm I}_{y_0}^{-1/2} (Y^C \,| \,0)\, \hat{\rm I}_{y_0}^{-1/2}\big)\,\rangle \vspace{6pt}\notag\\
&   + \,\dfrac{\lambda\,\mu}{\lambda+2\mu}\,\mathrm{tr}  \big( \hat{\rm I}_{y_0}^{-1/2} (X^C \,| \,0)\hat{\rm I}_{y_0}^{-1/2}\big)\,\mathrm{tr}  \big( \hat{\rm I}_{y_0}^{-1/2} (YX^C \,| \,0)\, \hat{\rm I}_{y_0}^{-1/2}\big).\notag
\end{align}

Moreover, if we decompose the matrix $ X^C \in \mathbb{R}^{3\times 2}$ in two block matrices (the  matrix $ \id_{2\times 3}\,X^C \in \mathbb{R}^{2\times 2}$ and the matrix $ e_3^T\, X^C \in \mathbb{R}^{1\times 2}$ ), then we can write the bilinear form as
\begin{equation}
W^C_{\mathrm{shell}}(  X^C,  Y^C)  = W_{\mathrm{inplane}}(  \id_{2\times 3}X^C,  \id_{2\times 3}Y^C) + \dfrac{\mu+\mu_c}{2}\; \bigl\langle e_3^T\, X^C\,{\rm I}_{y_0}^{-1} , e_3^T\, Y^C\bigr\rangle,
\end{equation}
where we define for any  $ X, Y\in \mathbb{R}^{2\times 2} $ the bilinear form
\begin{align}\label{e1,5}
W_{\mathrm{inplane}}( X, Y) \,=\,&  \mu\,\langle\, \mathrm{sym}\,    X,\,\mathrm{sym}\,  \big({\rm I}_{y_0}^{-1}\, Y\, {\rm I}_{y_0}^{-1}\big)\,\rangle +    \mu_c\langle \,\mathrm{skew}\,  X,\,\mathrm{skew}\,  \big({\rm I}_{y_0}^{-1}\, Y\, {\rm I}_{y_0}^{-1}\big)\,\rangle  + \,\dfrac{\lambda\,\mu}{\lambda+2\mu}\,\mathrm{tr}  \big( X\,{\rm I}_{y_0}^{-1}\big)\,\mathrm{tr}  \big( Y\, {\rm I}_{y_0}^{-1}\big)
\vspace{10pt}\notag\\=\,&  \mu\,\langle\, \mathrm{sym}\,  \big( {\rm I}_{y_0}^{-1/2}\, X\, {\rm I}_{y_0}^{-1/2}\big) ,\,\mathrm{sym}\,  \big({\rm I}_{y_0}^{-1/2}\, Y\, {\rm I}_{y_0}^{-1/2}\big)\,\rangle 
+    \mu_c\;\langle \,\mathrm{skew}\, \big( {\rm I}_{y_0}^{-1/2}\, X\, {\rm I}_{y_0}^{-1/2}\big) ,\,\mathrm{skew}\,  \big( {\rm I}_{y_0}^{-1/2}\, Y\, {\rm I}_{y_0}^{-1/2}\big)\,\rangle \vspace{6pt}\notag\\
&   + \,\dfrac{\lambda\,\mu}{\lambda+2\mu}\,\mathrm{tr}  \big(  {\rm I}_{y_0}^{-1/2}\, X\, {\rm I}_{y_0}^{-1/2}\big)\,\mathrm{tr}  \big(  {\rm I}_{y_0}^{-1/2}\, Y\, {\rm I}_{y_0}^{-1/2}\big).
\end{align}

Analogous results hold for the quadratic form $ W_{\mathrm{curv}}(  X,  X) $.

We can decompose the $ 3\times 2 $ matrices in two block matrices ($ 2\times 2 $ and $ 1\times 2 $) and express the strain energy densities $ W_{\mathrm{memb}} $ and $ W_{\mathrm{memb,bend}}$ as  functions of the matrices $ \mathcal{G}_\infty $ and $ \mathcal{R}_\infty $ (we have used that  $ \mathcal{T}_\infty$ vanishes) in the following form
\begin{align}\label{e4}
&  W_{\mathrm{memb}}\big(  \mathcal{E}_{m,s} \big)+  
W_{\mathrm{memb,bend}}\big(  \mathcal{E}_{m,s} ,\,  \mathcal{K}_{e,s} \big)=
\widetilde W_{\mathrm{memb}}\big( \mathcal{G}_\infty  \big)+   
\widetilde W_{\mathrm{memb,bend}}\big(  \mathcal{G}_\infty , \mathcal{R}_\infty \big)\notag
\vspace{6pt}\\
&=  \underbrace{\Big(h+{\rm K}\,\dfrac{h^3}{12}\Big)\,
	W_{\mathrm{inplane}}\big(  2\, \mathcal{G}_\infty\big)}_{\textrm{in-plane deformation}}\vspace{6pt}\\
&\ \ \  + \underbrace{\Big(\dfrac{h^3}{12}\,-{\rm K}\,\dfrac{h^5}{80}\Big)\, W_{\mathrm{inplane}}\big(  2 \,\mathcal{G}_\infty\, {\rm L}_{y_0} - \mathcal{R}_\infty\big) 
	- \dfrac{h^3}{12}\,\,2\,  W_{\mathrm{inplane}}\big( 2\, \mathcal{G}_\infty , (2 \,\mathcal{G}_\infty\, {\rm L}_{y_0} - \mathcal{R}_\infty)\,{\rm L}_{y_0}^*\big) +
	\,\dfrac{h^5}{80}\;W_{\mathrm{inp\lambda}}\big( 2 \,\mathcal{G}_\infty\,, ({\rm L}_{y_0} - \mathcal{R}_\infty)\,{\rm L}_{y_0}\big)}_{\textrm{in-plane deformation-bendings coupling terms}},\notag
\end{align}
where $ \; {\rm L}_{y_0}^* = \mbox{Cof}\,{\rm L}_{y_0} = 2\,\mbox{H}\,\id_{2} - {\rm L}_{y_0}$, $ W_{\mathrm{inplane}} $ is given by \eqref{e1,5} and we have denoted
\begin{align} 
W_{\mathrm{inp\lambda}}( X, Y) :=&  \,\mu\,\bigl\langle  \mathrm{sym}\,    X,\,\mathrm{sym}\,  \big({\rm I}_{y_0}^{-1}\, Y\, {\rm I}_{y_0}^{-1}\big)\bigr\rangle +    \mu_c\bigl\langle \mathrm{skew}\,  X,\,\mathrm{skew}\,  \big({\rm I}_{y_0}^{-1}\, Y \,{\rm I}_{y_0}^{-1}\big)\bigr\rangle   + \,\dfrac{\lambda}{2}\,\mathrm{tr}  \big( X\,{\rm I}_{y_0}^{-1}\big)\,\mathrm{tr}  \big( Y\, {\rm I}_{y_0}^{-1}\big)
\notag\\=\,& \, \mu\,\langle\, \mathrm{sym}\,  \big( {\rm I}_{y_0}^{-1/2}\, X\, {\rm I}_{y_0}^{-1/2}\big) ,\,\mathrm{sym}\,  \big({\rm I}_{y_0}^{-1/2}\, Y\, {\rm I}_{y_0}^{-1/2}\big)\,\rangle 
+    \mu_c\;\langle \,\mathrm{skew}\, \big( {\rm I}_{y_0}^{-1/2}\, X\, {\rm I}_{y_0}^{-1/2}\big) ,\,\mathrm{skew}\,  \big( {\rm I}_{y_0}^{-1/2}\, Y\, {\rm I}_{y_0}^{-1/2}\big)\,\rangle \vspace{6pt}\\
&   + \,\dfrac{\lambda}{2}\,\mathrm{tr}  \big(  {\rm I}_{y_0}^{-1/2}\, X\, {\rm I}_{y_0}^{-1/2}\big)\,\mathrm{tr}  \big(  {\rm I}_{y_0}^{-1/2}\, Y\, {\rm I}_{y_0}^{-1/2}\big).\notag
\end{align}
We can see in the expression \eqref{e4} the different parts of the energy corresponding to in-plane deformation, transverse shear, or coupling terms with bending.
\bigskip

Finally, the bending-curvature energy density $ W_{\mathrm{bend,curv}} $  can be written as function of $ \mathcal{R}_\infty $ and $  \mathcal{N}_\infty $ as follows
\begin{align}\label{e6}
W_{\mathrm{bend,curv}}&\big(  \mathcal{K}_{e,s} \big)=\,
\widetilde W_{\mathrm{bend,curv}}\big(  \mathcal{R}_\infty ,  \mathcal{N}_\infty \big)
\vspace{6pt}\notag\\
=\, & \underbrace{\Big(h-{\rm K}\,\dfrac{h^3}{12}\Big)\,
	W_{\mathrm{curvpls}}\big( \mathcal{R}_\infty \big)    +  \Big(\dfrac{h^3}{12}\,-{\rm K}\,\dfrac{h^5}{80}\Big)\,
	W_{\mathrm{curvpls}}\big(  \mathcal{R}_\infty {\rm L}_{y_0}  \big)  + \,\dfrac{h^5}{80}\,\,
	W_{\mathrm{curvpls}}\big(  \mathcal{R}_\infty {\rm L}_{y_0} ^2  \big)}_{\textrm{bending strain tensor}}
\vspace{6pt}\\
&+ \underbrace{\mu\, L_c^2\,\dfrac{b_1+b_2}{2}\,\Big[ 
	\Big(h-{\rm K}\,\dfrac{h^3}{12}\Big) \bigl\langle  \mathcal{N}_\infty\,  {\rm I}_{y_0}^{-1},  \mathcal{N}_\infty\bigr\rangle + \Big(\dfrac{h^3}{12}\,-{\rm K}\,\dfrac{h^5}{80}\Big)  \bigl\langle( \mathcal{N}_\infty\, {\rm L}_{y_0} ) {\rm I}_{y_0}^{-1} , ( \mathcal{N}_\infty\, {\rm L}_{y_0} )\bigr\rangle + \dfrac{h^5}{80}\, \bigl\langle( \mathcal{N}_\infty\, {\rm L}_{y_0} ^2) {\rm I}_{y_0}^{-1},( \mathcal{N}_\infty\, {\rm L}_{y_0} ^2)\bigr\rangle
	\Big]}_{\textrm{drilling bendings}},\notag
\end{align}
where we have denoted for any  $ 2\times 2  $ matrix $ X $ the quadratic form (positive definite)
\begin{align}\label{e7}
W_{\mathrm{curvpls}}( X) :=&\,  \mu\, L_c^2\Big[\,b_1\langle\, \mathrm{sym}\,    X,\,\mathrm{sym}\,  \big({\rm I}_{y_0}^{-1}\, X\, {\rm I}_{y_0}^{-1}\big)\,\rangle +    \Big(8\,b_3+\dfrac{b_1}{3}\Big)\langle \,\mathrm{skew}\,  X,\,\mathrm{skew}\,  \big({\rm I}_{y_0}^{-1}\, X\, {\rm I}_{y_0}^{-1}\big)\,\rangle+ \,\dfrac{b_2-b_1}{2}\,\big[\mathrm{tr}  \big( X\,{\rm I}_{y_0}^{-1}\big)\big]^2\,\Big]
\vspace{10pt}\\
=& \, \mu \,L_c^2\Big[ b_1\,\lVert \mathrm{sym}\,  \big( {\rm I}_{y_0}^{-1/2}\, X\, {\rm I}_{y_0}^{-1/2}\big) \rVert^2 
+     \Big(8\,b_3+\dfrac{b_1}{3}\Big)\;\lVert \,\mathrm{skew}\, \big( {\rm I}_{y_0}^{-1/2} \,X \,{\rm I}_{y_0}^{-1/2}\big) \rVert^2 + \,\dfrac{b_2-b_1}{2}\,\big[\mathrm{tr}  \big(  {\rm I}_{y_0}^{-1/2}\, X\, {\rm I}_{y_0}^{-1/2}\big)\,\big]^2\,\Big].\notag
\end{align}

Thus, the model can be expressed entirely in terms of the change of metric tensor $ \mathcal{G}_\infty $,  the bending strain tensor $ \mathcal{R}_\infty $ and the vector of drilling bendings $  \mathcal{N}_\infty $, through the relations \eqref{e4} and \eqref{e6}. 
\end{footnotesize}
\end{document}